\documentclass[12pt]{amsart}
\usepackage{amsmath}
\usepackage{amsfonts}
\usepackage{amssymb,color}
\usepackage{url}
\usepackage[hidelinks]{hyperref}
\usepackage[hyphenbreaks]{breakurl}
\usepackage{amsthm,doi}
\usepackage{thmtools}
\usepackage{physics}
\usepackage{tikz}
\usepackage{enumitem}

\usepackage[a4paper, hmargin={1.9cm,1.9cm},vmargin={2.9cm,2.5cm}]{geometry}

\makeatletter
\newcommand\RedeclareMathOperator{%
	\@ifstar{\def\rmo@s{m}\rmo@redeclare}{\def\rmo@s{o}\rmo@redeclare}%
}
\newcommand\rmo@redeclare[2]{%
	\begingroup \escapechar\m@ne\xdef\@gtempa{{\string#1}}\endgroup
	\expandafter\@ifundefined\@gtempa
	{\@latex@error{\noexpand#1undefined}\@ehc}%
	\relax
	\expandafter\rmo@declmathop\rmo@s{#1}{#2}}
\newcommand\rmo@declmathop[3]{%
	\DeclareRobustCommand{#2}{\qopname\newmcodes@#1{#3}}%
}
\@onlypreamble\RedeclareMathOperator
\makeatother

\RedeclareMathOperator{\var}{Var}
\DeclareMathOperator{\cov}{Cov}
\DeclareMathOperator{\sgn}{sgn}
\DeclareMathOperator{\jac}{Jac}

\newcommand{\C}{\mathbb{C}}

\newcommand{\charge}{\kappa}
\newcommand{\newcharge}{\mu}

\newcommand{\E}{\mathbb{E}}
\newcommand{\al}{\alpha}
\newcommand{\bl}{\beta}
\newcommand{\cl}{\gamma}
\newcommand{\dl}{\delta}

\newtheorem{lemma}{Lemma}[section]
\newtheorem{theorem}[lemma]{Theorem}
\newtheorem{coro}[lemma]{Corollary}
\newtheorem{prop}[lemma]{Proposition}

\theoremstyle{definition}

\numberwithin{equation}{section}
\newtheorem{rem}[lemma]{Remark}

\allowdisplaybreaks

\keywords{Gaussian Weyl-Heisenberg Function, zero set, spectrogram zero, twisted stationarity, short-time Fourier transform, hyperuniformity}

\subjclass[2020]{60G15, 60G55, 94A12, 42A61}

\title[Hyperuniformity and non-hyperuniformity of zeros of GWHF]{Hyperuniformity and non-hyperuniformity of zeros of Gaussian Weyl-Heisenberg Functions}

\author[N. Feldheim]{Naomi Feldheim}
\address[N. Feldheim]{Department of Mathematics, Bar-Ilan University, Israel}
\email{naomi.feldheim@biu.ac.il}

\author[A. Haimi]{Antti Haimi}
\address[A. Haimi]
{Department of Mathematics, \r{A}bo Akademi University, Tuomiokirkontori 3, 20500 Turku, Finland}
\email{antti.haimi@abo.fi}

\author[G. Koliander]{G\"{u}nther Koliander}
\address[G. Koliander]{Acoustics Research Institute, Austrian Academy of Sciences, Dominikanerbastei 16,	1010 Vienna, Austria}
\email{guenther.koliander@oeaw.ac.at}

\author[J. L. Romero]{Jos\'{e} Luis Romero}
\address[J. L. Romero]{Faculty of Mathematics, University of Vienna, Oskar-Morgenstern-Platz 1, A-1090 Vienna, Austria, and Acoustics Research Institute, Austrian Academy of Sciences, Dominikanerbastei 16,	1010 Vienna, Austria}
\email{jose.luis.romero@univie.ac.at}

\begin{document}\setlength{\jot}{10pt}
	
\begin{abstract}
We study zero sets of twisted stationary Gaussian random functions on the complex plane, i.e., Gaussian random functions that are stochastically invariant under the action of the Weyl-Heisenberg group. This model includes translation-invariant Gaussian entire functions (GEFs), and also many other non-analytic examples, in which case winding numbers around zeros can be either positive or negative. We investigate zero statistics both when zeros are weighted with their winding numbers (charged zero set) and when they are not (uncharged zero set).

We show that the variance of the charged zero statistic always grows linearly with the radius of the observation disk (hyperuniformity). Importantly, this holds for functions with possibly non-zero means and without assuming additional symmetries such as radiality. 
With respect to uncharged zero statistics, we provide an example for which the variance grows with the area of the observation disk (non-hyperuniformity). This is used to show that, while the zeros of GEFs are hyperuniform, the set of their critical points fails to be so.

Our work contributes to recent developments in statistical signal processing, where the time-frequency profile of a non-stationary signal embedded into noise is revealed by performing a statistical test on the zeros of its spectrogram (``silent points''). We show that empirical spectrogram zero counts enjoy moderate deviations from their ensemble averages over large observation windows (something that was previously known only for pure noise). In contrast, we also show that spectrogram maxima (``loud points'') fail to enjoy a similar property. This gives the first formal evidence for the statistical superiority of silent points over the competing feature of loud points, a fact that has been noted by practitioners. In the same vein, our second order asymptotics for spectrogram maxima show that certain heuristic proxy models used in signal processing are inaccurate at large scales.
\end{abstract}

\maketitle

\section{Introduction and Results}
\subsection{Gaussian Weyl-Heisenberg functions}
We study random functions on the complex plane $F\colon \mathbb{C} \to \mathbb{C}$ and their zeros. Specifically, we consider 
\begin{align}\label{eq_F}
F=F_0 + F_1,
\end{align}
where $F_1\colon \mathbb{C} \to \mathbb{C}$ is deterministic and $F_0$ is a circularly symmetric complex Gaussian random field on $\mathbb{C}$ with covariance kernel of the form
\begin{align}\label{eq_cov}
\mathbb{E} \big[ F_0(z) \cdot \overline{F_0(w)} \big] =
H(z-w) \cdot e^{i \mathrm{Im}(z \bar w)}=H(z-w) \cdot e^{\frac12(z \bar w- w \bar z)}, \qquad z,w\in\mathbb{C}.
\end{align}
Here, $H\colon \mathbb{C} \to \mathbb{C}$ is a suitably smooth function called \emph{twisted kernel}. 

The covariance structure \eqref{eq_cov} means that $F_0$ is \emph{twisted stationary}, that is, the distribution of $F_0$ is invariant under all \emph{twisted shifts}
\begin{align}\label{eq_ts}
\mathcal{T}_{\xi}F_0(z)= F_0(z-\xi)e^{i \mathrm{Im}(z \bar \xi)}, \qquad \xi,z \in \mathbb{C}.
\end{align}
The random functions $F_0$ were introduced in \cite{hkr22} and named \emph{Gaussian Weyl-Heisenberg Functions} (GWHF), as the operators \eqref{eq_ts} generate the (reduced) Weyl-Heisenberg group \cite{folland89}. Here we extend that nomenclature to include the random functions \eqref{eq_F}, which have a possibly non-trivial mean~$F_1$.

We mention en passant some first examples. For the special choice 
\begin{align*}
H(z)=e^{-\tfrac{1}{2}\abs{z}^2},
\end{align*}
$F_0(z)$ can be identified with a \emph{translation-invariant Gaussian entire function} \cite{NSwhat, gafbook} as follows:
\begin{align*}
F_0(z) = e^{-\tfrac{1}{2}\abs{z}^2} G_0(z),
\end{align*}
with
\begin{align}\label{eq_g0}
G_0(z) = \sum_{k \geq 0} \frac{\xi_k}{\sqrt{k!}} z^k,
\end{align}
and $\xi_k$ independent standard complex random variables. 
Other choices of $H$ may lead to non-analytic random functions. For example, if
\begin{align*}
H(z)=(1-|z|^2) e^{-\tfrac{1}{2}\abs{z}^2},
\end{align*}
then
\begin{align}\label{eq_g00}
F_0 (z) = e^{-\tfrac{1}{2}\abs{z}^2} \big[ \bar{z} \,G_0(z) - \partial G_0(z)\big],
\end{align}
where $G_0$ is the Gaussian entire function \eqref{eq_g0} and $\partial G_0$ is its derivative; see \cite[Section 6.5]{hkr22}. The expression in brackets in \eqref{eq_g00} is the \emph{covariant derivative} of $G_0$ and is instrumental in studying critical points of the weighted magnitude
$e^{-\abs{z}^2/2} | G_0(z)|$ (see Section \ref{sec_gef}). Examples of GWHF relevant in signal processing are discussed in Section \ref{sec_intro_tf}.

\subsection{Assumptions}
\label{sec_intro_h}
To describe smoothness, we will employ the differential operators 
\begin{align}\label{eq_D}
	\mathcal{D}_1F(z) = \partial F(z) -  \tfrac{\bar z}{2}F(z), \quad \mathcal{D}_2F(z) = \bar \partial F(z) + \tfrac{z}{2}F(z),
\end{align}
called \emph{twisted derivatives},
which commute with the twisted shifts \eqref{eq_ts}:
\begin{align}\label{eq_commute}
	\mathcal{D}_j \mathcal{T}_{\xi} = \mathcal{T}_{\xi} \mathcal{D}_j, \qquad \xi \in \mathbb{C}, \quad  j=1,2.
\end{align}
Here, we use the Wirtinger differential operators 
\begin{align*}
\partial = \frac12( \partial_x- i \partial_y), \quad \bar \partial = \frac12(\partial_x + i \partial_y).
\end{align*}
Throughout, we make the following assumptions, which are similar to those made in \cite{hkr22}:
\begin{itemize}[leftmargin=1em,itemsep=4pt]
    \item  We assume that the deterministic function $F_1$ is $C^2$ and
\begin{align}\label{A1}
 \sup_{z \in \mathbb{C}} |F_1(z)|,\, \sup_{z \in \mathbb{C}} |\mathcal{D}_j F_1(z)| < \infty, \qquad j=1,2.
\end{align}
\item  For the twisted kernel we assume the positive-definiteness condition
\begin{align}\label{A2}
	\Big(
	H(z_k - z_j) \cdot e^{i \Im (z_k \overline{z_j})} \Big)_{1\le j,k\le n} \geq 0 \qquad \mbox{for all } z_1,\ldots, z_n \in \mathbb{C},
\end{align}
which guarantees that \eqref{eq_cov} is indeed a covariance kernel.
This implies that
\begin{align} \label{eq_symm}
	H(-z)&= \overline{H(z)}, \qquad z \in \mathbb{C},
\end{align}
and  $H(0) \geq 0$.

\item  We further impose the normalization
\begin{align}\label{A3}
	H(0)&=1,
\end{align}
which means that $F(z)$ has unit variance.

\item  We assume the regularity condition 
\begin{align}\label{A4}
	|H(z)|<1, \quad z \in \C \setminus \{ 0 \}, 
\end{align}
which implies that no two samples $F(z)$, $F(w)$ are deterministically correlated, as the determinant of their covariance matrix is
\begin{align}\label{A4p}
	1-|H(z-w)|^2 \not=0, \qquad z \neq w.
\end{align}

\item  We assume that\footnote{Separability is a technical condition, which is sometimes not even explicitly mentioned in the literature. Informally, it means that the process is determined by its values on a certain countable set. A process is separable as soon as it has continuous paths, and, on the other hand, a separable process with smooth covariance has smooth paths \cite[Section 1.1]{adler} \cite[Chapter 1]{level}.} 
\begin{align}\label{A5}
\text{$F$ is a separable process and $H$ is $C^6$ smooth in the real sense},
\end{align} 
which guarantees that $F$ is almost surely a $C^2$ function \cite[Chapter 1]{level}.

\item  We assume the decay condition
\begin{align}\label{A6}
\sup_{z \in \mathbb{C}} (1+|z|^2) |H(z)|, \,
\sup_{z \in \mathbb{C}} (1+|z|^2) |\mathcal{D}_iH(z)|, \,
\sup_{z \in \mathbb{C}} (1+|z|^2) |\mathcal{D}_i \overline{\mathcal{D}_j}H(z)| < \infty,
\qquad i,j=1,2.
\end{align}
\end{itemize}

\subsection{Charged zeros and hyperuniformity}
We  augment each zero $z$ of $F$ with the attribute of \emph{charge} $\pm 1$, according to whether $F$ preserves or reverses orientation around $z$. More precisely, we inspect the differential matrix $DF$ of $F$ considered as $F\colon \mathbb{R}^2 \to \mathbb{R}^2$ and define
\begin{align}\label{eq_charge}
\begin{aligned}
	\charge_z := \begin{cases}
		1 & \mbox{if } \det DF(z) >0 \\
		0 & \mbox{if } \det DF(z)  = 0 \\
		-1 & \mbox{if } \det DF(z) < 0
	\end{cases}.
\end{aligned}
\end{align}
As we show in Lemma \ref{lem_pi}, almost surely, $\{\charge_z : F(z)=0\} \subset \{-1,1\}$.

The following is our first main result.

\begin{theorem}\label{th1}
Let $F$ be the GWHF \eqref{eq_F} and assume \eqref{A1}, \eqref{A2}, \eqref{A3}, \eqref{A4}, \eqref{A5}, and \eqref{A6}. Then 
\begin{align*}
\sup_{w \in \mathbb{C}} \sup_{R \geq 1} \tfrac{1}{R} \mathrm{Var} \Big[ \sum_{|z-w| \leq R} \charge_z \Big] < \infty.
\end{align*}
\end{theorem}
To compare, under the assumptions of Theorem \ref{th1}, the \emph{expected charge} to be found in a ball $\mathbb{E} [ \sum_{|z-w| \leq R} \charge_z ]$
can be as large as $\approx R^2$. In fact, when $F$ has zero mean, \cite[Theorem 1.12]{hkr22} gives the exact expression
\begin{align}\label{hey}
\mathbb{E} \Big[ \sum_{|z-w| \leq R} \charge_z \Big] = \frac{1}{\pi} R^2.
\end{align}
Thus, Theorem \ref{th1} shows that the fluctuation of
charge at large scales is anomalously small in comparison to Poissonian statistics, where expectation and variance grow at the same asymptotic rate as functions of the observation radius. In the jargon of statistical mechanics, we say that the point process of charged zeros is \emph{hyperuniform}, or that it has \emph{non-extensive fluctuations} \cite{torquato2016hyperuniformity, MR3815253}. Stationary analogues of Theorem \ref{th1} go back to \cite{wilkinson2004screening}, albeit in a less mathematical formulation.

When $F_1=0$ and $H$ is radial, Theorem \ref{th1} follows from \cite[Theorem 1.14]{hkr22}, which also provides an asymptotic expression for the variance as $R \to \infty$. In this article, the variance estimate is extended to possibly non-radial twisted kernels and non-zero means. 
The proof in~\cite{hkr22} depends on lengthy explicit calculations and breaks completely in the presence of a mean or in the absence of radial symmetries. Thus, new methods are needed;
see Section \ref{sec_met}.

\subsection{Non-hyperuniformity of uncharged zeros}\label{sec_intro_th2}
When the GWHF $F$ is a (weighted) \emph{analytic} function, all its zeros have non-negative charge due to conformality, and Theorem \ref{th1} expresses the well-known fact that zeros of a GEF are hyperuniform \cite{MR1383056,MR1690355}---albeit in the novel setting of non-zero means; see also Section \ref{sec_gef}. 
Our second main result concerns a non-analytic GWHF and, in a remarkable contrast to Theorem \ref{th1}, disproves the hyperuniformity of the corresponding zero set when charges are neglected.
\begin{theorem}\label{th2}
Let $F$ be the GWHF \eqref{eq_F} with $F_1 \equiv 0$ and $H(z) = (1-|z|^2)e^{-\abs{z}^2/2}$. Then there exist constants $c,C>0$ such that
\begin{align}\label{eq_intro_cr}
c R^2 
\leq \mathrm{Var} \big[ \# \{z \in \mathbb{C}: F(z)=0, |z| \leq R \} \big] 
= \mathrm{Var} \Big[ \sum_{|z| \leq R} \lvert\charge_z\rvert \Big]
\leq C R^2, \qquad R \geq 1.
\end{align}
\end{theorem}
The technique that we shall develop to prove Theorem \ref{th2} is very general and applies to many other twisted kernels. For example, for kernels of the form $H(z)=P(|z|^2) e^{-|z|^2/2}$, with $P \in \mathbb{C}[z]$, our method gives a sufficient condition for \eqref{eq_intro_cr} in terms of a finite computation with the coefficients of $P$. We accompany the article with a symbolic software notebook \cite{jupy_gwhf_nonhyp} which performs these computations and delivers variations of Theorem \ref{th2}. With it, we have verified the analog of Theorem \ref{th2} for $H(z)=L_k(|z|^2) e^{-|z|^2/2}$ with $k \leq 5$, where $L_k$ denotes the Laguerre polynomial of degree $k$ -- while Theorem \ref{th2}
corresponds to the case $k=1$, since
$L_1(x)=1-x$. The random functions that result from choosing $H$ to be one of the Laguerre polynomials are important in mathematical physics, as they are Gaussian eigenfunctions of the so-called \emph{planar Landau equation} \cite{MR2593994, vas00}, or, alternatively, \emph{random polyanalytic functions of pure type} \cite{haimi2013polyanalytic,hkr22,
haimi2019central}. 

\subsection{Applications to Gaussian entire functions}\label{sec_gef}

Let $G_0$ be the zero-mean random analytic function given by \eqref{eq_g0}. It is well-known that the zero set of $G_0$ is hyperuniform:
\begin{align*}
\var \big[ \# \{z \in\C: G_0(z)=0, \, |z| \leq R \} \big] = O(R), \qquad R \to \infty,
\end{align*}
as follows from an explicit computation of the two-point correlation function \cite{MR1690355, MR1383056} which also delivers the asymptotic limit of the renormalized variance; see also \cite{MR2863379}. As a consequence of Theorem \ref{th1}, we derive a similar conclusion when $G_0$ is supplemented with a non-trivial analytic mean $G_1$
in the so-called Bargmann-Fock spaces of entire functions with quadratic exponential growth \cite{zhu}.
\begin{theorem}[Hyperuniformity of GEF zeros with mean]
\label{th_5}
Let $G=G_0+G_1$, where $G_0$ is the translation-invariant GEF \eqref{eq_g0} and $G_1\colon \C \to \C$ is entire with 
\begin{align}\label{eq_ie}
\sup_{z\in\C} |G_1(z)| e^{-\tfrac{1}{2}\abs{z}^2} < \infty.
\end{align}
Then 
\begin{align}\label{eq_id}
\var \big[ \# \{z \in\C: G(z)=0, \, |z| \leq R \} \big] \leq C R, \qquad R \geq 1,
\end{align}
for a constant $C>0$, which can be chosen to depend only on an upper bound for the left-hand side of \eqref{eq_ie}.
\end{theorem}
The strong version of the hyperuniformity of GEF zeros with mean that we present does not seem to follow easily from explicit computations as in \cite{MR1690355, MR1383056}, which rely on special symmetries destroyed by the presence of a mean.

A second set of applications concerns the \emph{covariant derivative}
\begin{align*}
\bar{\partial}^* G(z) = \bar{z} G(z) - \partial G(z)
\end{align*}
of a zero-mean translation-invariant GEF $G=G_0$. The operator $\bar{\partial}^*$ is the adjoint of the Wirtinger derivative $\bar{\partial}$ with respect to the $L^2$-inner product with Gaussian weight $e^{-|z|^2}$. In the language of complex geometry, $\bar{\partial}^* G$ is the derivative of a \emph{holomorphic section} $G$ of the standard line bundle on the plane with Gaussian metric (Hermitian Gaussian measure). The set of critical points $\{\bar{\partial}^* G=0\}$ corresponding to a random section $G$ is instrumental in the analysis of heuristic or approximate models in string theory \cite{MR2104882}. First order statistics for the critical points of $G$ are computed (with respect to more general metrics) in \cite{MR2104882}. We shall look into second order statistics.

A simple computation shows that $F(z)= e^{-|z|^2/2}\, \bar{\partial}^* G(z)$ is a GWHF with twisted kernel $H(z) = (1-|z|^2)e^{-\abs{z}^2/2}$ \cite{hkr22}. Thus Theorem \ref{th2} can be reformulated as follows:
\begin{theorem}[Non-hyperuniformity of critical points of GEF]\label{th_3}
The set of critical points of a zero-mean translation-invariant Gaussian entire function $G$ satisfies
\begin{align*}
c R^2 \leq
\var \big[ \# \{z \in\C: \bar{\partial}^* G(z)=0, \, |z| \leq R \} \big] \leq C R^2, \qquad R \geq 1,
\end{align*}
for adequate constants $c,C>0$.
\end{theorem}
 We emphasize that the covariant derivative $\bar{\partial}^* G$ is not analytic, and the study of the second order statistics of its zeros is different from the corresponding endeavor for $G$ \cite{MR1690355,MR2465693}.
 We illustrate the results of Theorems \ref{th_5} and \ref{th_3} in Figure~\ref{fig:iodzeros} by simulations of the index of dispersion (variance of the number of points divided by the mean number of points).
 \begin{figure}
    \centering
    \includegraphics{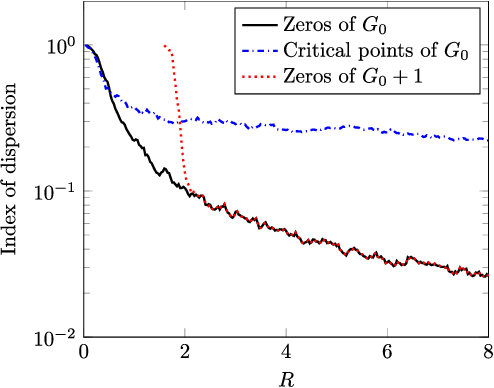}
    \caption{Simulated index of dispersion (variance of the number of points divided by the mean number of points) of the zeros of the GEF $G_0$, the GEF with non-zero mean $G_0+1$, and its critical points in a circle of radius $R$.  For zeros of the GEF (with or without mean), we see a decreasing index of dispersion, while for the critical points, we see that it stays approximately constant showing our proven asymptotic behavior already in a finite domain. The GEF with mean $G_0+1$ has an extremely low probability of having zeros close to the origin and thus no zeros were observed for small $R$ in the simulations.}
    \label{fig:iodzeros}
\end{figure}

The critical points of $G$ (which are almost surely non-degenerate) can be classified according to their topological index (which is almost surely $\pm 1$):
\begin{align*}
&\mathcal{N}_R^+ := \#\{ z \in\C: \bar{\partial}^* G(z)=0, \, |z| \leq R , \mathrm{Index}\big(\bar{\partial}^* G,z\big) = 1 \},
\\
&\mathcal{N}_R^- := \#\{ z \in\C: \bar{\partial}^* G(z)=0, \, |z| \leq R , \mathrm{Index}\big(\bar{\partial}^* G,z\big) = -1 \}.
\end{align*}
The corresponding first order statistics are given in \cite{MR2104882} in the context of general metrics (see also \cite[Section 6.8]{hkr22}). As for second order statistics, we note that the topological index of a critical point of $G$ is exactly the charge of the GWHF $F(z)= e^{-|z|^2/2}\, \bar{\partial}^* G(z)$
(see Section \ref{sec_gef_p}). As a consequence, Theorems \ref{th1} and \ref{th2} can be leveraged
to prove the following companion to Theorem \ref{th_3}.
\begin{theorem}[Non-hyperuniformity of critical points of GEF with given index]\label{th_4}
The set of critical points of a translation-invariant Gaussian entire function $G$ satisfies
\begin{align}\label{eq_ic}
c R^2 \leq
\var \big[ \mathcal{N}_R^\pm \big] \leq C R^2, \qquad R \geq 1,
\end{align}
for adequate constants $c,C>0$.
\end{theorem}
Theorems \ref{th_3} and \ref{th_4} can also be formulated in terms of the \emph{weighted magnitude}
\begin{align*}
A(z) = e^{-\tfrac{1}{2}\abs{z}^2} |G(z)|
\end{align*}
and the statistics of its local extrema (local maxima, local minima and saddle points); see Section \ref{sec_wgef}.

\subsection{Applications in statistical signal processing}\label{sec_intro_tf}

Our main motivation comes from certain recent developments in the field of signal processing. The goal is to analyze a distribution $f \in \mathcal{S}'(\mathbb{R})$---called signal---by means of its \emph{spectrogram}
\begin{align}
S f(x,\xi) := \left| 
\int_{\mathbb{R}} f(t) e^{-2\pi i t \xi} e^{-\pi (t-x)^2}\, dt
\right|^2, \qquad (x,\xi) \in \mathbb{R}^2.
\end{align}
The integral, which is to be interpreted distributionally, quantifies the influence of the frequency $\xi$ in $f(t)$ near $t=x$.

In realistic applications, the signal is \emph{non-stationary}, which means that all frequencies $\xi$ can potentially contribute to the spectrogram, and therefore the shape of the essential support of $Sf$ (that is, the region where most of the energy concentrates) is rather unpredictable. Nevertheless, in practice,
$S f$ is expected to be concentrated on a thin or small-measure set which carries most of the information (time-frequency sparsity), a fact that can be computationally exploited \cite{MR3078284}.

The estimation of the essential support of $Sf$ is traditionally done by looking into large values of $Sf$, a task that can be challenging if $f$ is contaminated with significant additive noise~\cite{flandrin2018explorations}. Remarkably, it has been recently discovered that in the additive noise regime the \emph{zeros} of $Sf$ provide a rich set of landmarks from which the essential support of $Sf$ can be effectively inferred. The intuition is that the zeros of the spectrogram of (white) noise $\mathcal{W}$ behave like charged particles, and thus form a rather rigid random pattern with predictable statistics \cite{gardner2006sparse}, from which the presence of an underlying signal can be recognized as a salient local perturbation \cite{7180335, 7869100}.
In practice, the number of zeros of $Sf$ is computed on reasonably sized test disks, which are classified as ``meaningful" if the zero count deviates significantly from what is expected for the spectrogram of noise $S\mathcal{W}$. The union of all meaningful disks furnishes an approximation of the essential support of $Sf$.\footnote{This is a very simplified description; practical algorithms are of course much more refined and combine various approaches \cite{MR4047541, bh, MIRAMONT2024109250, miramont2024benchmarking, pascal2024point}.}

The advantages of spectrogram zeros (``silent points") over other filtering landmarks such as local maxima (``loud points") has been the object of significant numerical investigations \cite{miramont:hal-04102094,miramont2024benchmarking,pascal2024point}. The success of zero-based spectrogram filtering depends crucially on the reliability of empirical statistics, that is, on the desirable property that zero counts computed with a signal impacted by a concrete realization of noise reflect ensemble averages if calculated on sufficiently large observation disks. As it turns out, the \emph{short-time Fourier transform} (STFT)
\begin{align}\label{eq_stft_gauss}
V f(x,\xi) = \int_{\mathbb{R}} f(t) e^{-2\pi i t \xi} e^{-\pi (t-x)^2}\, dt, \qquad (x,\xi) \in \mathbb{R}^2,
\end{align}
associated with standard complex white noise $f=\mathcal{W}$ as input can be identified with the translation-invariant Gaussian entire function $G_0$ \eqref{eq_g0}:
\begin{align}\label{eq_connection}
e^{-\tfrac{1}{2}\abs{z}^2} G_0(z) = 2^{1/4} e^{-i x \xi} \cdot V\, \mathcal{W} (x/\sqrt{\pi},-\xi/\sqrt{\pi}), 
\qquad z=x+i\xi,
\end{align}
see \cite{MR4047541, bh, MR1662451}. Hence, the zeros of the spectrogram of noise $S \mathcal{W}=|V \mathcal{W}|^2$ are hyperuniform (moderate variance in comparison to its expectation at large scales). This gives \emph{partial support} to the success of empirical zero statistics for $S (f+\mathcal{W})$: 
If the observed number of zeros deviates significantly from what is expected for the spectrogram of  noise $S (\mathcal{W})$, we can conclude with high probability that a signal is present. 
Furthermore, a formula for the expected number of zeros of $S (f+\mathcal{W})$ is known \cite[Proposition 3.4]{efkr24}, which suggests that in parts of the time-frequency plane  where the deterministic signal $f$ has low energy it does not significantly change the expected number of zeros. 
However, the effect of deterministic signals on the variance of the number of spectrogram zeros remained so far unknown.
Thus, the hypothesis that the disks that pass the statistical zero counting test (meaningful disks) are likely to be descriptive of the underlying signal $f$ has not been  well founded.

As an application of our main results, we contribute to the analysis of non-stationary signal processing as follows.

\begin{theorem}[Hyperuniformity of zeros of spectrograms, and lack thereof for local maxima]
\label{th_sp}
Let $f \in \mathcal{S}'(\mathbb{R})$ be a distribution with bounded spectrogram $Sf$
and let $\mathcal{W}$ be standard complex white noise. Then
\begin{itemize}
\item[(i)] The zero set $Z_{f+ \mathcal{W}}$ of $S(f+\mathcal{W})$ is hyperuniform. More precisely
\begin{align}\label{eq_vs}
\sup_{(x,\xi) \in \mathbb{R}^2}
\var [\# Z_{f+\mathcal{W}} \cap B_R(x,\xi)] \lesssim R, \qquad R \geq 1.
\end{align}
\item[(ii)] The set $M_{\mathcal{W}}$ of local maxima of $S\,\mathcal{W}$ is not hyperuniform. More precisely,
\begin{align}\label{eq_vs_2}
\var [\# M_{\mathcal{W}} \cap B_R(0,0)] \asymp R^2, \qquad \mbox{as }R \to \infty.
\end{align}
\end{itemize}
\end{theorem}
Part (i) of Theorem \ref{th_sp} concerns the success of empirically computed statistics for spectrogram zeros of an \emph{arbitrary signal} impacted by additive white noise and completes the heuristic explanation of the success of zero-based filtering by showing that disks that are not significant for the underlying clean signal are likely to fail the zero-counting statistical test (non-meaningful disks). 
Furthermore, novel tests based on spectrogram zeros that aim at specifically distinguishing between signals within a given class can be expected to perform very well due to the hyperuniformity of the zero count for each signal contaminated with noise.

Part (ii) of Theorem \ref{th_sp} provides, for the first time, statistical support for the superiority of zeros over local maxima as filtering landmarks \cite{7180335, 7869100,flandrin2018explorations}, because it shows that statistics computed with the latter suffer from much larger fluctuations in areas of the time-frequency plane dominated by noise. Theorem \ref{th_sp} also reveals a fundamental difference between the point process of spectrogram maxima and the heuristic proxy model used so far in signal processing, which approximately describes spectrogram maxima as a perturbed lattice \cite{flmax}: while perturbed lattices are hyperuniform \cite{MR4226523,MR4083986}, spectrogram maxima turn out not to be so, and, as a consequence, certain signal processing heuristics \cite{flmax} are inaccurate at large scales.

As for the class of signals $f$ covered by Theorem \ref{th_sp}, we mention that the assumption that their spectrograms are bounded is standard. Indeed, it means that $f$ belongs to the \emph{modulation space} $M^\infty(\mathbb{R})$ \cite{MR201959,benyimodulation}, which includes all distributions commonly used in signal processing. The bounded spectrogram assumption is also fair from the point of view of modeling, whereas, if we were to assume that $Sf$ decays, the proof of \eqref{eq_vs} could potentially be easier, as one could ignore the influence of the mean at very large scales. (In the same vein, one can most likely extend part (ii) to signals $f+\mathcal{W}$ where $Sf$ is assumed to decay at infinity, but in the end those details have not seemed interesting enough for this article.)

Finally, we comment on further applications to time-frequency analysis. The STFT of a signal $f$ is defined more generally with respect to a \emph{window function} $g \in \mathcal{S}(\mathbb{R})$:
\begin{align}
\label{eq_stft}
V_g f(x,y) = \int_{\mathbb{R}} f(t) \overline{g(t-x)} e^{-2\pi i t y} dt,
\qquad (x,y) \in \mathbb{R}^2.
\end{align}
The most common choice for $g$ is the Gaussian function, which corresponds to \eqref{eq_stft_gauss}. In practice, the same signal is often processed with multiple windows $g$ so as to average out the bias that they introduce (multi-tapering) \cite[Chapter 10]{flandrin2018explorations}. Typical choices for $g$ are Hermite functions $\{h_n: n \geq 0\}$ as these optimize several measures related to Heisenberg’s uncertainty principle.

For the Gaussian window $g=h_0$, the STFT $V_g f$ of a general signal $f$ is a weighted analytic function. While this is not the case for any other choice of $g$ (up to symmetries) \cite{MR2729877}, the STFT of a signal $f$ impacted by standard complex white noise $\mathcal{W}$ can still be identified with a GWHF as follows:
\begin{align*}
F(x+iy) := e^{-i xy} \cdot V_g \, (f + \mathcal{W}) \big(x/\sqrt{\pi}, -y/\sqrt{\pi}\big),
\end{align*}
see \cite{hkr22}.
Our work shows that the zero counting statistics for the STFT with respect to a general window $g$ may fail to share the statistical advantages found in the Gaussian window case, as it implies that the zeros of the STFT of complex white noise calculated with respect to the Hermite function $h_1$ are \emph{not} hyperuniform.

\begin{theorem}\label{th_sp_2}
Let $h_1(t) := t e^{-\pi t^2}$ and $\mathcal{W}$ standard complex white noise. Then the zero set $Z_{h_1, \mathcal{W}}$ of $V_{h_1} \mathcal{W}$ satisfies
\begin{align*}
\var [\# Z_{h_1, \mathcal{W}} \cap B_R(0,0)] \asymp R^2, \qquad \mbox{as }R \to \infty.
\end{align*}
\end{theorem}
On the bright side, we are able to show that, for a general window function $g$, \emph{charge statistics}, corresponding to sums of winding numbers of the STFT:
\begin{align}\label{eq_newcharge}
\newcharge_{z} f= \sgn \Big\{\Im  \Big[  \partial_x V_g f(x,y) \cdot \overline{ \partial_y V_g f(x,y)} \,\Big] \Big\}, \qquad z=x+iy,
\end{align}
do exhibit moderate fluctuations at large scales in the presence of noise, and thus offer an attractive novel alternative to zero-based filtering.
\begin{theorem}\label{th_sp_3}
Let $g \in \mathcal{S}(\mathbb{R})$ be non-zero, $f \in M^\infty(\mathbb{R})$ (i.e., $V_gf \in L^\infty(\mathbb{R})$), and $\mathcal{W}$ standard complex white noise. Then
	\begin{align*}
	\var  \Big[
	{\sum}_{|z| \leq R,\, V_g\, (f+\mathcal{W}) (z) = 0} \,\,\newcharge_z (f+\mathcal{W}) \Big]\leq C_g R, \qquad R \geq 1,
	\end{align*}
 for a constant $C_g$.
\end{theorem}

\subsection{Methods and related literature}\label{sec_met}
The model of twisted stationary Gaussian fields was introduced in \cite{hkr22}, where first and second order statistics for zero sets were derived by means of Kac-Rice formulae and laborious explicit computations, and under zero-mean and strong symmetry assumptions (such as radiality). The results that we present here seem to be out of the scope of such direct methods.

To prove Theorems \ref{th1} and \ref{th2}, the primary challenge lies in adapting techniques for stationary random fields to the case of twisted stationary random functions. A main insight that we systematically exploit is that in the twisted setting the differential operators $\mathcal{D}_1, \mathcal{D}_2$ play a role that is analogous to that of the Euclidean derivatives in the classical stationary context. Besides that basic common element, the proofs of Theorems \ref{th1} and \ref{th2} are based on rather independent techniques.

To prove Theorem \ref{th1}, we draw inspiration from \cite{buckley2017fluctuations, feldheim2018variance, buckley2018winding}, which study, respectively, the argument change of Gaussian entire functions along curves, the variance of analytic stationary random fields, and the winding of complex-valued stationary random fields. As a first step, we express the total charge of the zero set of a GWHF $F$ using \emph{Poincar\'e's index formula}
\[
\sum_{z \in B, F(z)=0} \kappa_z= \frac{1}{2\pi i} \int_{\partial B} \frac{dF}{F},
\]
which we suitably modify to make it covariant under twisted shifts (see Lemma \ref{lem_pi}). This leads us to the task of analyzing correlations among quotients of the form $\mathcal{D}_j F/F$. While in the stationary setting this can be done by means of rather explicit formulae that go back to Kahane \cite{MR0833073}, in the twisted setting we need to derive new estimates. These necessitate arguments that are subtler than those used for related purposes in \cite{feldheim2018variance}, which, for example, break down in the presence of a non-zero mean. Another challenge is that there is no simple description of the positive-definiteness of the twisted kernel \eqref{eq_cov} -- in contrast to the Euclidean stationary situation, where covariance kernels are characterized by the positivity of their Fourier transforms. This leads us to indirect arguments that exploit the positive-definiteness of \eqref{eq_cov} to analyze the Hessian of the squared twisted kernel $|H|^2$.

Our second main result, Theorem \ref{th2}, is proved by developing a \emph{chaos expansion} \cite{janson1997gaussian} of the zero counting statistic, a technique that goes back to \cite{MR1440400, MR1084335, MR1303648} in the stationary setting, and has been successfully applied in several contexts; see, e.g., \cite{dalmao2019phase, MR2257657,nourdin2019nodal}. In the twisted stationary setting, we develop an expansion adapted to the twisted derivatives $\mathcal{D}_1$ and $\mathcal{D}_2$. In Theorem \ref{th_chaos} below, we obtain a chaos expansion for the zero statistic of a GWHF with general radial twisted kernel. With the chaos expansion at hand, we then deduce the non-hyperuniformity of the zero set in question by explicitly estimating the projection of the zero count statistic into the so-called second order chaotic subspace. The corresponding computations rely on the so-called \emph{Feynman diagram method} \cite{janson1997gaussian}, which has also been used for related purposes in, e.g.,  ~\cite{buckley2022gaussian, buckley2017fluctuations, sodin2004random}. 

As a central technical step towards the chaos expansion, we prove a uniform upper bound for the two-point function at non-zero levels, $\{F=u\}$, (Proposition \ref{prop_r}) -- a conclusion that in the stationary case follows by simply inspecting the spectral measure,
see, e.g., \cite{MR0358956}. As a consequence of the uniform bound on the two-point function, we conclude in Theorem \ref{th_coro_r} that the first and second moments of $u$-level counting statistics are continuous in the level $u$, just as is the case in the stationary setting \cite{MR1440400}. We view this as a result of independent interest, and note that the proof technique is general enough to be applicable beyond the setting of Gaussian Weyl-Heisenberg functions.  
We expect that the chaos expansion will be useful in further applications, such as proving a CLT (see \cite{azais2023winding} in the stationary setting). Although we develop the chaos expansion for uncharged zeros, we expect the same to apply with charged zeros (and to be, in fact, technically easier). 

\subsection{Organization}
Section \ref{sec_pre} introduces the notation and relates properties of the covariance kernel \eqref{eq_cov} to pointwise properties of the twisted kernel $H$. Section \ref{sec_poincare} briefly discusses Poincar\'e's index formula and reformulates it in terms of twisted derivatives. Section \ref{sec_cov} obtains several correlation estimates related to the terms in Poincar\'e's formula, which are instrumental in studying winding numbers. Section \ref{sec_var}
contains a proof of Theorem \ref{th1}, in fact in a slightly more quantitative form (Theorem \ref{th1plus}).

Section \ref{sec_reg} investigates non-zero level crossings $\{F=u\}$ and their dependence on $u$. This lays the technical foundation for Section \ref{sec_chaos}, which derives the chaos expansion of the zero counting statistic (Theorem \ref{th_chaos}). This is done for general radial twisted kernels $H$. The approach can be extended to non-radial $H$ at the cost of additional technicalities, which in the end did not seem to merit inclusion in this article.
In Section \ref{sec_nonhyper}, we prove Theorem \ref{th2} by means of explicit computations with the so-called Feynman diagrams corresponding to the chaos expansion of the zero statistic and the particular twisted kernel $H(z)=(1-|z|^2) e^{-|z|^2/2}$. This requires certain algebraic computations with Laguerre polynomials, which are presented as an appendix (Section \ref{appendix}), and can also be followed with a symbolic software notebook publicly available at \cite{jupy_gwhf_nonhyp}. The notebook
also delivers the calculations that are relevant to establish analogues of Theorem \ref{th2} with respect to twisted kernels of the form $H(z) = P(|z|^2) e^{-|z|^2/2}$, $P \in \mathbb{C}[z]$, which are also of interest, cf.\ Section \ref{sec_intro_th2}.

Finally, Section \ref{sec_app} provides detailed arguments for the applications described in Section \ref{sec_gef} (Gaussian entire functions) and Section \ref{sec_intro_tf} (time-frequency analysis). In particular, Theorems \ref{th_5}, \ref{th_3}, and \ref{th_4} are proved in Section \ref{sec_gef_p}, while Section \ref{sec_wgef} discusses a reformulation of these in terms of the classification of critical points of weighted entire functions; and Theorem \ref{th_sp},  \ref{th_sp_2}, and \ref{th_sp_3} are proved in Section \ref{sec_app_tf}. Section \ref{sec_aux} contains auxiliary results.

\section{Preliminaries on Twisted Stationarity}\label{sec_pre}
\subsection{Notation}
The indicator function of a set $E$ is denoted $1_E$.
We adopt the usual complex-variable notation $dz= dx + i dy$,
$d\bar{z}= dx - i dy$, $|dz|= \sqrt{(dx)^2 + (dy)^2}$, while
the differential of the Lebesgue (area) measure on $\mathbb{C}$ is denoted $dA$. For short, we sometimes write $dA(z,w)$ for $dA(z)dA(w)$. We will write $\Delta=\partial \bar \partial$, which is one quarter of the standard Laplace operator.  

The \emph{Jacobian} of a function $F\colon \mathbb{C} \to \mathbb{C}$ at $z \in \mathbb{C}$ is the determinant of its differential matrix $DF$ considered as $F\colon \mathbb{R}^2 \to \mathbb{R}^2$:
\begin{align*}
	\jac F(z) := \det DF(z).
\end{align*}
Recall that the charge of $F$ at a zero $z$ is given by \eqref{eq_charge}. The \emph{total charge} of $F$ on a domain $\Omega \subset \mathbb{C}$ is defined as
\begin{align}
\label{eq:chargeomega}
	\charge_{\Omega} := \sum_{z \in \Omega, F(z)=0} \charge_z.
\end{align}
We will also make extensive use of the covariant differential operators \eqref{eq_D} and the twisted shifts $\mathcal{T}_{w}$ \eqref{eq_ts}. If the operators $\mathcal{D}_j$ are applied to a function of more than one variable, we specify the relevant variable in a second subscript, e.g., $\mathcal{D}_{j,z}F(w,z)$.

\subsection{Covariances}
In what follows, we will often need the covariance between $F(z)$ and $\mathcal{D}_jF(z)$, $j=1,2$, at different points $z$. The following lemma expresses these in terms of twisted shifts.
\begin{lemma}\label{lem_precov}
Let $F$ be the GWHF \eqref{eq_F} with twisted kernel $H$ and assume \eqref{A2} and \eqref{A5}. 
Then the following hold.
\begin{align}\label{eq_cov_spec}
\mathbb{E} \big[ F_0(z) \cdot \overline{F_0(w)} \big] &=
\mathcal{T}_{w}H(z) , \qquad z,w\in\mathbb{C}.
\\
\label{eq_cov_spec_diff}
\mathbb{E} \big[ \mathcal{D}_jF_0(z) \cdot \overline{F_0(w)} \big] &=
\mathcal{D}_j\mathcal{T}_{w}H(z) = \mathcal{T}_w\mathcal{D}_jH(z) , \qquad z,w\in\mathbb{C},\qquad j\in \{1,2\},
\\
\label{eq_cov_spec_diff2}
\mathbb{E} \big[ \mathcal{D}_jF_0(z) \cdot \overline{\mathcal{D}_kF_0(w)} \big] 
& = - \mathcal{T}_{w}\mathcal{D}_{j}\overline{\mathcal{D}_{k}}H(z) 
, \qquad z,w\in\mathbb{C}, \qquad j,k\in \{1,2\}.
\end{align}
\end{lemma}
\begin{proof}
The expression \eqref{eq_cov_spec} is clear from the definitions, while \eqref{eq_cov_spec_diff} follows from \eqref{eq_cov_spec}  and \eqref{eq_commute} after exchanging
expectation with the operators $\mathcal{D}_j$.
This is allowed because the covariance of $F$ is smooth and $F$ is separable \cite[Chapter 1]{level}. Similarly, to prove \eqref{eq_cov_spec_diff2} we can further rewrite 
\begin{align*}
\mathbb{E} \big[ \mathcal{D}_jF_0(z) \cdot \overline{\mathcal{D}_kF_0(w)} \big] 
& = \mathcal{D}_{j,z}\overline{\mathcal{D}_{k,w}}\mathcal{T}_{w}H(z) 
\notag  \\
& = - \mathcal{D}_{j,z}\mathcal{T}_{w}\overline{\mathcal{D}_{k,z}}H(z) 
\notag  \\
& = - \mathcal{T}_{w}\mathcal{D}_{j,z}\overline{\mathcal{D}_{k,z}}H(z) 
\notag  \\
& = - \mathcal{T}_{w}\mathcal{D}_{j}\overline{\mathcal{D}_{k}}H(z) 
, \qquad z,w\in\mathbb{C},
\end{align*}
where we used that $\overline{\mathcal{D}_{k,w}}\mathcal{T}_{w}H(z) = -\mathcal{T}_{w}\overline{\mathcal{D}_{k,z}}H(z) $.
\end{proof}

\subsection{Invertibility of the Hessian of the squared twisted kernel}
We will show that our assumptions on $H$ imply that there exists a positive constant $c$ such that  
\[
1-|H(z)|^2 \geq c|z|^2
\]
for $z$ in a neighborhood of the origin. The proof requires the following two lemmas. 
\begin{lemma} \label{lem_covariance}
Let $F$ be the GWHF \eqref{eq_F} and satisfy \eqref{A1}, \ldots, \eqref{A6} and $z \in \C$.
Then the 
covariance matrix of the Gaussian vector $\big(F(z), \mathcal{D}_1 F(z), \mathcal{D}_2F(z)\big)$ is 
given by 
\begin{equation} \label{eq:cov_matrix}
A=\begin{pmatrix} 
1 & -\bar \partial H(0) & - \partial H(0) \\
 \partial H(0) & -\Delta H(0)+\frac12 & -\partial^2 H(0)  \\
 \bar \partial H(0) & -\bar \partial^2 H(0) & -\Delta H(0)- \frac 12 
\end{pmatrix}.
\end{equation}
\end{lemma}
\begin{proof}
By Lemma \ref{lem_precov} and because $\mathcal{T}_{z}G(z)=G(0)$ for $G\in \{F, \mathcal{D}_j F, \mathcal{D}_j\overline{\mathcal{D}_k}F \}$, we can assume that $z=0$. We have by \eqref{eq_symm}, \eqref{A3},
\eqref{eq_cov_spec},  \eqref{eq_cov_spec_diff}, 
and \eqref{eq_cov_spec_diff2} 
that
\begin{align*}
\E \big[ F_0(0) \overline{F_0(0)}\big] &= 1,\\
\E \big[ \mathcal{D}_1F_0(0) \overline{F_0(0)}\big]&=\partial H(0), \\
\E \big[ \mathcal{D}_2F_0(0) \overline{F_0(0)}\big] &=\bar \partial H(0), \\
\E \big[ \mathcal{D}_1F_0(0) \overline{\mathcal{D}_1 F_0(0)}\big] =
-\mathcal{D}_1 \overline{\mathcal{D}_1}H(0) &=
-(\partial- \bar z/2)(\bar \partial - z/2) H(z)_{|z=0}= -\Delta H(0) + \frac12, 
\\
\E \big[ \mathcal{D}_2F_0(0) \overline{\mathcal{D}_2 F_0(0)}\big] 
=
-\mathcal{D}_2 \overline{\mathcal{D}_2}H(0) 
&=
-(\bar \partial+ z/2)(\partial + \bar z/2) H(z)_{|z=0}= -\Delta H(0)- \frac12, 
\\
\E \big[ \mathcal{D}_1F_0(0) \overline{\mathcal{D}_2 F_0(0)}\big] 
=
-\mathcal{D}_1 \overline{\mathcal{D}_2}H(0) 
&=
-(\partial - \bar z/2)( \partial + \bar z/2) H(z)_{|z=0}
= -\partial^2 H(0).
\end{align*}
The remaining entries of $A$ can be computed using Hermitian symmetry and \eqref{eq_symm}.
\end{proof}

\begin{lemma} \label{lem_hes}
Let $F$ be the GWHF \eqref{eq_F} and satisfy \eqref{A1}, \ldots, \eqref{A6} and let $A$ be given by \eqref{eq:cov_matrix}.
Then
$$
\det A = -\frac14 \det \begin{pmatrix} \partial^2 |H|^2(0) & \Delta |H|^2(0) \\
\Delta |H|^2(0) & \bar \partial^2 |H|^2(0)
\end{pmatrix} - \frac14.
$$
In addition, because $A$ is positive semi-definite, it follows that 
$$\det \begin{pmatrix} \partial^2 |H|^2(0) & \Delta |H|^2(0) \\
\Delta |H|^2(0) & \bar \partial^2 |H|^2(0)
\end{pmatrix} \leq -1.  $$
\end{lemma}
\begin{proof}
We will use repeatedly the relations $H(0)=1$, $H(-z)=\overline{H(z)}$ in the process. The second identity implies for example
\begin{align*}
\partial H(0)&= -\partial \bar H(0), \\
\partial^2 H(0)&= \partial^2 \bar H(0).
\end{align*}
Similar relations hold for other derivatives of first and second orders.

Gaussian elimination gives
\begin{align*}
\det A %
&= \det \begin{pmatrix} -\Delta H(0)+ \frac12 + \bar \partial H(0) \cdot \partial H(0) & - \partial^2 H(0) + (\partial H(0))^2 \\
-\bar \partial^2 H(0)+ (\bar \partial H(0))^2 & -\Delta H(0) -\frac12 + \partial H(0) \cdot \bar \partial H(0) \end{pmatrix} \\
&= \big(-\Delta H(0) + \bar \partial H(0) \cdot \partial H(0) \big)^2- \frac{1}{4} - 
\bigl|-\partial^2 H(0) + (\partial H(0))^2\bigr|^2.
\end{align*}
Next, we will compute the determinant of the $2 \times 2$ matrix in the statement of the lemma.   We have 
\begin{align*}
\partial^2 |H|^2(0) &= \partial (\partial H(z) \cdot \bar H(z)+ H(z) \cdot \partial \bar H(z))_{z=0} \\
&=\partial^2 H(0)-(\partial H(0))^2 - (\partial H(0))^2+ \partial^2 H(0)= 2(\partial^2 H(0)-(\partial H(0))^2 ),
\\
\bar \partial^2 |H|^2(0) &= 2(\bar \partial^2 H(0)-(\bar \partial H(0))^2 ),
\\
\Delta |H|^2(0) &= 
\bar \partial (\partial H(z) \cdot \bar H(z)+ H(z) \cdot \partial \bar H(z))_{z=0}= 2 (\Delta H(0)+ |\partial H(0)|^2),
\end{align*}
so 
\begin{align*}
\partial^2 |H|^2(0) \cdot \bar \partial^2 |H|^2(0) - (\Delta |H|^2(0))^2 %
&= 4|\partial^2 H(0)-(\partial H(0))^2|^2-
4(\Delta H(0)+|\partial H(0)|^2)^2. %
\end{align*}
We obtain the claim by comparing this expression with the expression for $\det A$.
\end{proof}

\begin{prop} \label{prop_hess}
Let $F$ be the GWHF \eqref{eq_F} and satisfy \eqref{A1}, \ldots, \eqref{A6}. 
There exists $c>0$ such that 
\begin{align*}
1-|H(z)|^2 \geq c |z|^2,
\end{align*}
for $z$ in an adequate neighborhood of the origin. 
\end{prop}
\begin{proof}
By \eqref{A4}, $|H|^2$ attains a maximum at $0$, and thus can be Taylor expanded near the origin as
$$
|H(z)|^2 = 1+ \frac12 \Big( 2 \mathrm{Re} \big[\partial^2 |H|^2(0) z^2 \big]+
2 \Delta |H|^2(0)|z|^2 \Big)+ O(|z|^3),
$$
with the quadratic form being necessarily negative semi-definite. It is enough to show that the form is in fact strictly negative definite. Suppose on the contrary that there exists $z \neq 0$ such that 
\begin{equation*} %
 2 \mathrm{Re} \left(\partial^2 |H|^2(0) z^2  +
 \Delta |H|^2(0)|z|^2 \right) = 0. 
\end{equation*}
This implies that there is a real number $r$ such that 
$$
\partial^2 |H|^2(0) z^2  +
 \Delta |H|^2(0) |z|^2 =  ir. 
$$
This means
$$
\big\lvert \partial^2 |H|^2(0)\big\rvert^2  \lvert z\rvert^4  
=
r^2 + \big\lvert \Delta |H|^2(0) \big\rvert^2 |z|^4,
$$
which implies that 
$$
\big\lvert \partial^2 |H|^2(0)\big\rvert^2- \big\lvert\Delta |H|^2(0)\big\rvert^2 - \frac{r^2}{|z|^4} = 0\,.
$$
However, the left-hand side is $<0$ by Lemma \ref{lem_hes}, a contradiction.
\end{proof}
Proposition \ref{prop_hess} and Assumptions \eqref{A4} and \eqref{A6} imply that 
\begin{align} \label{eq_H_lower_bound}
\inf_{z \in \C \setminus \{0\}} \frac{1-|H(z)|^2}{\min\{1, |z|^2\}} > 0.
\end{align}

\subsection{Model-dependent constants}\label{sec_mod_dep}
We say that $F$ is a GWHF satisfying the \emph{general assumptions} if $F$ is given by \eqref{eq_F} and satisfies \eqref{A1}, \ldots, \eqref{A6}. Many of the results below contain constants that depend on these assumptions. To be more quantitative, we say that a constant \emph{depends on the model} if it can be specified as a function of the following:
\begin{itemize}
\item An upper bound for the left-hand side of \eqref{A1}.
\item A lower bound for the left-hand side of \eqref{eq_H_lower_bound}. 
\item An upper bound for the left-hand side of \eqref{A6}.
\item A lower bound for the smallest \emph{positive} eigenvalue of the covariance matrices of \[(F(0), \mathcal{D}_j F(0)), \quad j=1,2.\]
\end{itemize}

Importantly, estimates for GWHF formulated in terms of model-dependent constants remain valid if $F$ is replaced with the twisted shift 
$\mathcal{T}_{\xi}F(z)= F(z-\xi)e^{i \mathrm{Im}(z \bar \xi)}$. The corresponding zero sets are related as follows.

\begin{rem}\label{rem_shift_charge}
If $F: \mathbb{C} \to \mathbb{C}$ vanishes at $z \in \mathbb{C}$ then a direct computation shows that $\mathcal{T}_{\xi} F$ vanishes at $z+\xi$ and
\[\charge_{z+\xi} (\mathcal{T}_{\xi} F) = \charge_z (F).\]
\end{rem}

\section{Poincare Index}\label{sec_poincare}
We consider a GWHF $F$ and the 
total charge $\charge_{\Omega}$ \eqref{eq:chargeomega} of $F$ on a domain
$\Omega \subset \C$.
As a first step towards the proof of the hyperuniformity of the charge statistics, we derive a variant of Poincare's index formula adapted to the twisted derivatives.
\begin{lemma}\label{lem_pi}
	Let $F$ be a GWHF satisfying the general assumptions, let $\Omega \subset \mathbb{C}$ be a  compact domain with smooth boundary, and let $u \in \mathbb{C}$. Then the following hold almost surely.
	\begin{itemize}
		\item[(i)] 
  The level-crossings of $F$ are non-degenerate, i.e., 
		\begin{align}\label{eq_nd}
			\jac F(z) \not=0, \qquad z \in \{F=u\}\cap \Omega.
		\end{align}		
		\item[(ii)] $\{F=0\} \cap \Omega$ is finite and does not intersect $\partial \Omega$.
		\item[(iii)] (Covariant Poincare's index formula)
		\begin{align}\label{eq_charge2}
			\charge_{\Omega} = \frac{1}{\pi} |\Omega| + \frac{1}{2\pi i} \int_{\partial \Omega} \left(\frac{\mathcal{D}_1 F }{F} \,dz + \frac{\mathcal{D}_2 F }{F} \,d\bar{z}\right).
		\end{align}
	\end{itemize}
\end{lemma}
\begin{proof}
	For part (i), we invoke \cite[Proposition 6.5]{level}.
	The required hypotheses are that $F$ be $C^2$ almost surely, as we assume in \eqref{A5}, and that the probability density of $F(z)$ be bounded near $u$ uniformly on $z$, which in our case holds because $F(z)$ is a circularly symmetric complex Gaussian vector with variance
	$\var[F(z)] = H(0)=1$ and bounded mean by \eqref{A1}. 
 
	For part (ii), we use Kac-Rice's formula to conclude that there is a measurable function $\rho_1\colon \C \to [0,\infty)$ such that for any Borel set $E \subset \C$: 
	\begin{align}\label{eq_kr1}
	\E \big[\#\{F=0\} \cap E\big] = \int_E \rho_1\, dA.
	\end{align}
	 Concretely, we invoke \cite[Theorem 6.2]{level}, a version of Kac-Rice's formula for Gaussian random fields that requires: (a) $F$ to be almost surely $C^1$, which is granted by \eqref{A5}; (b) $\var[F(z)]$ to be non-zero for all $z$, which is granted by \eqref{A3}; (c) the non-degeneracy condition \eqref{eq_nd}. The function $\rho_1$ can be expressed as a certain conditional expectation \cite[Theorem 6.2]{level}, but we shall not need this fact. Simply, since $\partial\Omega$ has null Lebesgue measure,
	 we use \eqref{eq_kr1} with $E=\partial\Omega$ to learn that
	 $\E \big[\#\{F=0\} \cap \partial \Omega \big]=0$, and thus $\#\{F=0\} \cap \partial \Omega=0$ almost surely.
	
	For part (iii), we note that \eqref{eq_nd} means that $0$ is a regular value of $F$. We invoke Poincare's index formula \cite{MR1487640, MR0209411} $\charge_\Omega = \frac{1}{2\pi i} \int_{\partial \Omega} \frac{dF}{F}$ and deduce \eqref{eq_charge2} as follows:
	\begin{align*}
		2\pi i \cdot \charge_\Omega &= \int_{\partial \Omega} \frac{dF}{F}
		= \int_{\partial \Omega} \left(\frac{\partial F }{F} \,dz + \frac{\bar{\partial} F }{F} \,d\bar{z}\right)
		\\
		&= \int_{\partial \Omega} \left(\frac{\mathcal{D}_1 F }{F} \,dz + \frac{\mathcal{D}_2 F }{F} \,d\bar{z}\right)
		+ \frac{1}{2} \int_{\partial \Omega} \bar{z} \,dz
		- \frac{1}{2} \int_{\partial \Omega} z \,d\bar{z}
		\\
		&=\int_{\partial \Omega} \left(\frac{\mathcal{D}_1 F }{F} \,dz + \frac{\mathcal{D}_2 F }{F} \,d\bar{z}\right) + i\cdot \Im\left[\int_{\partial \Omega} \bar{z} \,dz\right].
	\end{align*}
	Finally, by Green's theorem,
	\begin{align*}
		\Im\left[\int_{\partial \Omega} \bar{z} \,dz\right]=
		\int_{\partial \Omega} \left(x\,dy - y\,dx\right) = \int_\Omega 2 \,dxdy = 2 \abs{\Omega},
	\end{align*}
	which gives \eqref{eq_charge2}.
\end{proof}

\section{Covariance Estimates}\label{sec_cov}
Motivated by Lemma \ref{lem_pi}, we look into  correlations between the quotients $\mathcal{D}_j F/F$, and derive several technical estimates.
\begin{lemma}\label{lemma_c}
Let $F$ be a GWHF satisfying the general assumptions. Then the following hold.
\begin{itemize}
\item[(i)] For each $p \in (0,\infty)$ there exists a constant $C_p$ such that
\begin{align}\label{eq_ppp1}
\E \big[ \big| \mathcal{D}_j F(z) \big|^{p} \big]\leq C_p, \qquad z \in \mathbb{C}, j=1,2.
\end{align}
and
\begin{align}\label{eq_ppp2}
	\E \big[ \big| \mathcal{D}_j F(z) \cdot \mathcal{D}_k F(w) \big|^{p} \big] \leq C_p, \qquad z,w \in \mathbb{C}, j,k=1,2.
\end{align}
\item[(ii)] For each $p \in (1,2)$ there exists a constant $C_p$ such that
\begin{align}\label{eq_ma}
	\E \big[ \big| F(z) \big|^{-p} \big]\leq C_p, \qquad z \in \mathbb{C},
\end{align}
and
\begin{align}\label{eq_aaaa}
	\E \Big[ \big|{F(z)} \cdot
	F(w) \big|^{-p} \Big] \leq 
	\begin{cases}
	C_p & \mbox{when }|H(z-w)| < 1/2
	\\
	C_p \cdot (1-|H(z-w)|^2)^{1-p} & \mbox{when } |H(z-w)| \geq 1/2	
	\end{cases}.
\end{align}
\end{itemize}
In each case, the constant $C_p$ depends on the model and $p$.
\end{lemma}
\begin{proof}
\noindent {\bf Part (i)}. The mean of $\mathcal{D}_j F(z)$ is bounded by \eqref{A1} and its variance is bounded by Lemma \ref{lem_precov} and \eqref{A6}, which gives \eqref{eq_ppp1}, while \eqref{eq_ppp2} follows in turn by H\"older.

\noindent {\bf Part (ii)}.
The Gaussian random vector $(F(z), F(w))$ has mean $(F_1(z), F_1(w))$ and covariance
\begin{align}\label{q2}
\Gamma(z,w) = \begin{bmatrix}
	1 & H(z-w) e^{i \mathrm{Im}(z \bar w)}
	\\
	\overline{H(z-w)} \cdot e^{-i \mathrm{Im}(z \bar w)} & 1
	\end{bmatrix}.
\end{align}
To prove \eqref{eq_ma} we invoke the Hardy-Littlewood rearrangement inequality:
\begin{align*}
	\E \big[ \big| F(z) \big|^{-p} \big] = \frac{1}{\pi} \int_{\C} |u+F_1(z)|^{-p} e^{-|u|^2} \,dA(u)
	\leq  \frac{1}{\pi} \int_{\C} |u|^{-p} e^{-|u|^2} \,dA(u) =: C_p <\infty,
\end{align*}
because $p<2$. Suppose that $|H(z-w)| < 1/2$. Then $\Gamma(z,w)$ has determinant $\geq 1/2$ and there exists a model-dependent constant $a>0$ such that
\begin{align}\label{q1}
\Gamma^{-1}(z,w) \geq a I.
\end{align}
We use \eqref{q1} and the Hardy-Littlewood rearrangement inequality to estimate
\begin{align*}
\E \Big[ \big|{F(z)} \cdot
F(w) \big|^{-p} \Big] &= \frac{1}{\pi^2 \det \Gamma} \int_{\C^2} |u + F_1(z)|^{-p} |v + F_1(w)|^{-p} e^{-(u,v)^* \Gamma^{-1} (u,v)}\,dA(u) dA(v)
\\
&\lesssim \int_\C |u + F_1(z)|^{-p} e^{-a|u|^2} \,dA(u) \cdot 
\int_\C |v + F_1(w)|^{-p} e^{-a|v|^2} \,dA(v)
\\
&\leq \int_\C |u|^{-p} e^{-a|u|^2} \,dA(u) \cdot 
\int_\C |v|^{-p} e^{-a|v|^2} \,dA(v) =: C_p< \infty
\end{align*}
because $p<2$.

Suppose now that $|H(z-w)| \geq 1/2$ and let $\alpha := H(z-w) e^{i \mathrm{Im}(z \bar w)}$. Inspecting \eqref{q2}, we write
\begin{align*}
	&F(z) = F_1(z) + \xi_1,
	\\
	&F(w) = F_1(w) + \alpha \cdot \xi_1 + (1-|\alpha|^2)^{1/2} \cdot \xi_2,
\end{align*}
with $(\xi_1,\xi_2) \sim \mathcal{N}_\C(0,I)$. Then, 
\begin{align*}
	&\E \Big[ \big|{F(z)} \cdot
	F(w) \big|^{-p} \Big] 
	\\
	&\quad= \frac{1}{\pi^2} \int_{\C^2} 
	|u + F_1(z)|^{-p} \, |\alpha u + (1-|\alpha|^2)^{1/2} v + F_1(w)|^{-p}  \, e^{-|u|^2} e^{-|v|^2} \,dA(u) dA(v)
	\\
	&\quad= \frac{1}{\pi^2} \int_{\C^2} 
	|u + F_1(z)-\alpha^{-1}F_1(w)|^{-p} \, |\alpha u + (1-|\alpha|^2)^{1/2} v|^{-p}  \, e^{-|u-\alpha^{-1}F_1(w)|^2} e^{-|v|^2} \,dA(u) dA(v)
	\\
	&\quad\leq 
	\frac{1}{\pi^2} \int_{\C^2} 
	|u + F_1(z)-\alpha^{-1}F_1(w)|^{-p} \, |\alpha u + (1-|\alpha|^2)^{1/2} v|^{-p}  \,  e^{-|v|^2} \,dA(u) dA(v)
	\\
	&\quad=
	\frac{1}{\pi^2} \int_{\C} 
	|u + F_1(z)-\alpha^{-1}F_1(w)|^{-p} \,  \int_{\C} |\alpha u + (1-|\alpha|^2)^{1/2} v|^{-p}   e^{-|v|^2} \,dA(v) dA(u).
\end{align*}
Using Lemma \ref{lemma_conv}, we estimate the inner integral as
\begin{align*}
&\int_{\C} |\alpha u + (1-|\alpha|^2)^{1/2} v|^{-p}   e^{-|v|^2} \,dA(v)
=(1-|\alpha|^2)^{-p/2}
\int_{\C} |\alpha(1-|\alpha|^2)^{-1/2} u +  v|^{-p}   e^{-|v|^2} \,dA(v)
\\
&\qquad\leq C_p (1-|\alpha|^2)^{-p/2} (1+|\alpha(1-|\alpha|^2)^{-1/2}u|)^{-p},
\end{align*}
where $C_p$ is a constant depending only on $p$ that might change from line to line.
Therefore, using the Hardy-Littlewood rearrangement inequality
\begin{align*}
&\E \Big[ \big|{F(z)} \cdot
F(w) \big|^{-p} \Big]
\\
&\quad\leq 
	 C_{p}  (1-|\alpha|^2)^{-p/2}
\int_{\C} 
|u + F_1(z)-\alpha^{-1}F_1(w)|^{-p} \,  (1+|\alpha|(1-|\alpha|^2)^{-1/2}|u|)^{-p} dA(u)
\\
&\quad\leq C_{p} (1-|\alpha|^2)^{-p/2}
\int_{\C} 
|u|^{-p} \,  (1+|\alpha|(1-|\alpha|^2)^{-1/2}|u|)^{-p} dA(u)
\\
&\quad= C_{p} (1-|\alpha|^2)^{1-p/2} |\alpha|^{-2} 
\int_{\C} 
||\alpha|^{-1} (1-|\alpha|^2)^{1/2} u|^{-p} \,  (1+|u|)^{-p} dA(u)
\\
&\quad \leq C_p 
(1-|\alpha|^2)^{1-p} 
\int_{\C} |u|^{-p} \,  (1+|u|)^{-p} dA(u),
\end{align*}
since $|\alpha|^{p-2} \lesssim 1$. Finally, we note that $\int_{\C} |u|^{-p} \,  (1+|u|)^{-p} dA(u) < \infty$ because $1<p<2$, and obtain \eqref{eq_aaaa}.
\end{proof}

\begin{coro}\label{coro1}
Let $F$ be a GWHF satisfying the general assumptions. Then there exist model-dependent constants $C,L>0$ such that
\begin{align}\label{eq_p1}
	\E \Big[ \big| \tfrac{\mathcal{D}_j F(z) }{F(z)}\big|\Big] \leq C,
	\qquad z \in \C,
\end{align}
and
\begin{align}\label{eq_p1b}
	\E \Big[ \big| \tfrac{\mathcal{D}_j F(z) }{F(z)} \cdot
	{\tfrac{\mathcal{D}_k F(w) }{{F(w)}}} \big| \Big] \leq C, \qquad z,w \in \C, |z-w| > L.
\end{align}
\end{coro}
\begin{proof}
Let us choose $1<p<2$, say $p=3/2$, and let $q \in (2,\infty)$ be its H\"older conjugate, $1/p+1/q=1$. By H\"older and Lemma \ref{lemma_c},
\begin{align*}
\E \Big[ \big| \tfrac{\mathcal{D}_j F(z) }{F(z)}\big|\Big]
\leq \E \Big[ \big| \mathcal{D}_j F(z) \big|^q\Big]^{1/q} \cdot \E \Big[ \big| F(z)\big|^{-p}\Big]^{1/p} \leq C.
\end{align*}
Similarly, by \eqref{A2} and \eqref{A6}, there exists a constant $L>0$ such that $|H(z-w)|<1/2$ for $|z-w|>L$ and Lemma \ref{lemma_c} implies
\begin{align*}
	\E \Big[ \big| \tfrac{\mathcal{D}_j F(z) }{F(z)} \cdot
	{\tfrac{\mathcal{D}_k F(w) }{{F(w)}}} \big| \Big]
	\leq
	\E \Big[ \big| {\mathcal{D}_j F(z) } \cdot
	{\mathcal{D}_k F(w) }\big|^q \Big]^{1/q} \cdot 
	\E \Big[ \big| {F(z) } \cdot
	{F(w) }\big|^{-p} \Big]^{1/p} \leq C'.
\end{align*}
\end{proof}

\begin{prop}\label{prop_a}
Let $F$ be a GWHF satisfying the general assumptions. Then there exist model-dependent constants $L,C>0$ such that the following bound holds for all $j,k\in\{1,2\}$:
\begin{align}\label{eq_a}
\Big|
\cov \Big[ \tfrac{\mathcal{D}_j F(z) }{F(z)} ,
{\tfrac{\mathcal{D}_k F(w) }{{F(w)}}} \Big]\Big| \leq C (1+|z-w|)^{-2}, \qquad
|z-w| >L.
\end{align}
\end{prop}
\begin{proof}
Throughout the proof, we fix $j,k\in\{1,2\}$. \\
\noindent {\bf Step 1}. Assume first that the random vectors $(F(z), \mathcal{D}_j F(z))$, $(F(w), \mathcal{D}_k F(w))$ have non-singular covariance matrices, and denote them by $A_j, A_k \in \C^{2\times2}$. These matrices are independent of $(z,w)$ by Lemma \ref{lem_precov} (see also Lemma \ref{lem_covariance}).

We consider the random vector 
\begin{align}\label{V}
(F(z), \mathcal{D}_j F(z), F(w), \mathcal{D}_k F(w))
\end{align}
and decompose its covariance into $2\times2$ blocks 
\begin{align}\label{eq_AB}
\Gamma(z,w) := 
\begin{bmatrix}
A_j & B(z,w) \\
B^*(z,w) & A_k 
\end{bmatrix},
\end{align}
The entries in these blocks are the covariances derived in \eqref{eq_cov_spec}--\eqref{eq_cov_spec_diff2}.
Thus, by Assumption \eqref{A6}, $\| A_j \|, \|A_k\| \lesssim 1$, while
$\| B(z,w) \| \lesssim (1+|z-w|)^{-2}$. As the \emph{approximate covariance matrix}
\begin{align}\label{eq_AB2}
\widetilde{\Gamma}(z,w) := 
\begin{bmatrix}
	A_j & 0 \\
	0 & A_k 
\end{bmatrix}
\end{align}
is invertible, there exist constants $L,a>0$ such that if $|z-w|>L$, $\Gamma(z,w)$ is also invertible, and, moreover,
\begin{align}\label{eq_p2}
	\|\Gamma(z,w) - \widetilde{\Gamma}(z,w)\|, \ 
\|\Gamma(z,w)^{-1} - \widetilde{\Gamma}(z,w)^{-1}\| \lesssim (1+|z-w|)^{-2},
\end{align}
\begin{align}
\label{eq_p2b}
\Gamma(z,w) \geq a I, \qquad
\Gamma(z,w)^{-1} \geq a I, \qquad
\widetilde{\Gamma}(z,w) \geq aI, \qquad
\widetilde{\Gamma}(z,w)^{-1} \geq aI,
\end{align}
where $a$ is model dependent (depending on the smallest eigenvalue of $\widetilde{\Gamma}$).
We also assume that $L$ is larger than the corresponding constant in Corollary \ref{coro1}, so that \eqref{eq_p1} and \eqref{eq_p1b} hold.

\noindent {\bf Step 2}. Let $z,w\in\C$ with $|z-w|>L$. 

Using \eqref{eq_p2} and a  relative perturbation bound for determinants, we 
get
\begin{equation*}
    \bigg\lvert\frac{\det \Gamma(z,w) -\det \widetilde{\Gamma}(z,w)}{\det \widetilde{\Gamma}(z,w)}\bigg\rvert
    \leq O((1+|z-w|)^{-2}).
\end{equation*}
Thus, by \eqref{eq_p1b}, we can rewrite
\begin{align}\label{eq_step2}
\E \Big[ \tfrac{\mathcal{D}_j F(z) }{F(z)} 
\overline{\left({\tfrac{\mathcal{D}_k F(w) }{{F(w)}}}\right)} \Big]
= 
\frac{\det \Gamma(z,w) }{\det \widetilde{\Gamma}(z,w)} \cdot
\E \Big[ \tfrac{\mathcal{D}_j F(z) }{F(z)} 
\overline{\left({\tfrac{\mathcal{D}_k F(w) }{{F(w)}}}\right)} \Big]
+ O((1+|z-w|)^{-2}).
\end{align}
Hence,
\begin{align}\label{eq_p3}
\begin{aligned}
&\cov \Big[ \tfrac{\mathcal{D}_j F(z) }{F(z)} ,
{\tfrac{\mathcal{D}_k F(w) }{{F(w)}}} \Big]
\\
&\qquad=
\E \Big[ \tfrac{\mathcal{D}_j F(z) }{F(z)} 
\overline{\left({\tfrac{\mathcal{D}_k F(w) }{{F(w)}}}\right)} \Big]
- \E \Big[ \tfrac{\mathcal{D}_j F(z) }{F(z)} \Big]
\E \Big[ \overline{\Big({\tfrac{\mathcal{D}_k F(w) }{{F(w)}}}\Big)} \Big]
\\
&\qquad=
\frac{\det \Gamma(z,w) }{\det \widetilde{\Gamma}(z,w)} \cdot
\E \Big[ \tfrac{\mathcal{D}_j F(z) }{F(z)} 
\overline{\left({\tfrac{\mathcal{D}_k F(w) }{{F(w)}}}\right)} \Big]
- \E \Big[ \tfrac{\mathcal{D}_j F(z) }{F(z)} \Big]
\E \Big[ \overline{\Big({\tfrac{\mathcal{D}_k F(w) }{{F(w)}}}\Big)} \Big]
+ O((1+|z-w|)^{-2})
\\
&\qquad=
\frac{1}{\det \widetilde{\Gamma}(z,w)}
\int_{\C^4} \frac{u_2 + \mathcal{D}_j F_1(z)}{u_1 + F_1(z)} \cdot 
\overline{\left(\frac{u_4 + \mathcal{D}_k F_1(w)}{u_3 + F_1(w)}\right)} 
\left[ e^{-u \Gamma(z,w)^{-1} u^*} - e^{-u \widetilde{\Gamma}(z,w)^{-1} u^*}\right]
\, dA(u)
\\
&\qquad\quad +
 O((1+|z-w|)^{-2}).
\end{aligned}
\end{align}
We use the estimate $|e^{-t} - e^{-s}| \leq |t-s| (e^{-t}+e^{-s})$, $s,t>0$, 
and the fact that $\Gamma(z,w)$ and $\widetilde{\Gamma}(z,w)$ are positive matrices to bound
\begin{align*}
&\big| e^{-u \Gamma(z,w)^{-1} u^*} - e^{-u \widetilde{\Gamma}(z,w)^{-1} u^*}\big|
\leq | u (\Gamma^{-1}(z,w) - \widetilde{\Gamma}^{-1}(z,w)) u^*|
\big( e^{-u \Gamma(z,w)^{-1} u^*} + e^{-u \widetilde{\Gamma}(z,w)^{-1} u^*}\big)
\\
&\qquad\lesssim (1+|z-w|)^{-2} \big( e^{-u \Gamma(z,w)^{-1} u^*} + e^{-u \widetilde{\Gamma}(z,w)^{-1} u^*}\big) (|u_1|^2+|u_2|^2+|u_3|^2+|u_4|^2)
\\
&\qquad\lesssim (1+|z-w|)^{-2} \prod_{h=1}^4 e^{-a |u_j|^2} (1+|u_j|^2).
\end{align*}
Combining this estimate with \eqref{eq_p3} we conclude that
\begin{align}\label{kl5}
\Big|\cov \Big[ \tfrac{\mathcal{D}_j F(z) }{F(z)} ,
{\tfrac{\mathcal{D}_k F(w) }{{F(w)}}} \Big]\Big|
\lesssim (1+|z-w|)^{-2} \big(
1+I_1 \cdot I_2 \cdot I_3 \cdot I_4 \big),
\end{align}
where
\begin{align}\label{kl1}
\begin{aligned}
	I_1 = \int_\C |u_1+F_1(z)|^{-1} (1+|u_1|^2) e^{-a|u_1|^2} \, dA(u_1)
	&\lesssim \int_\C |u_1+F_1(z)|^{-1} e^{-\tfrac{a}{2}|u_1|^2} \, dA(u_1)
	\\
	& \leq 
	\int_\C |u_1|^{-1} e^{-\tfrac{a}{2}|u_1|^2} \, dA(u_1) \lesssim 1,
\end{aligned}
\end{align}
by the Hardy-Littlewood rearrangement inequality;
\begin{align}\label{kl2}
I_2 = \int_\C |u_2+\mathcal{D}_j F_1(z)| (1+|u_2|^2) e^{-a|u_2|^2} \, dA(u_2)
\lesssim \int_\C (1+|u_2|) \cdot (1+|u_2|^2) e^{-a|u_2|^2} \, dA(u_2) \lesssim 1,
\end{align}
by \eqref{A1}; and
\begin{align}\label{kl3}
I_3 = \int_\C |u_3+F_1(w)|^{-1} (1+|u_3|^2) e^{-a|u_3|^2} \, dA(u_3) \lesssim 1,
\\\label{kl4}
I_4 = \int_\C |u_4+\mathcal{D}_k F_1(w)| (1+|u_4|^2) e^{-a|u_4|^2} \, dA(u_4) \lesssim 1,
\end{align}
as in \eqref{kl1}, \eqref{kl2}.
We now combine \eqref{kl1}, \eqref{kl2}, \eqref{kl3}, \eqref{kl4} to obtain the desired conclusion.

\noindent {\bf Step 3}. Finally, we consider the case in which the covariance matrix of either \[V_j(z) := (F(z), \mathcal{D}_j F(z)) \qquad \mbox{or} \qquad V_k(w) := (F(w), \mathcal{D}_k F(w))\] is singular (possibly both).

If the covariance matrix of $V_j(z)$ is non-singular, let $V_j^*(z) := V_j(z)$. If it is singular, let $V^*_j(z) := F(z)$. Note that in this latter case, since $\var[F(z)] = |H(0)|^2=1$, there exists a deterministic constant $\lambda_j \in \C$ such that $\mathcal{D}_j F_0(z) =  \lambda_j F_0(z)$ and, thus,
$\mathcal{D}_j F(z) = \mathcal{D}_j F_1(z) + \lambda_j F_0(z)$. Moreover, $|\lambda_j| \lesssim 1$ by \eqref{A6}. We define $V_k^*$ similarly, and concatenate $V_j^*$ and $V_k^*$ to form a random vector $V^* =(V^*_j, V^*_k)$ of length $n \in \{2,3\}$. This vector is thus obtained from \eqref{V} by eliminating one or two components.

Let $A_j, A_k$ and $\Gamma(z,w)$ be the covariance matrices of $V^*_j$, $V^*_k$ and $V^*$ respectively, and consider again the decomposition \eqref{eq_AB} and \eqref{eq_AB2}. We argue as in the previous case to obtain constants $a,L>0$ such that \eqref{eq_p2} and \eqref{eq_p2b} hold. As a consequence \eqref{eq_step2} remains valid. The expansion \eqref{eq_p3} remains valid with integration over $\C^n$ and by setting $u_2=\lambda_j u_1$ and/or $u_4=\lambda_k u_3$. To adapt the subsequent bounds, one needs to consider one additional estimate of the form
\begin{align*}
I_5 &= \int_\C |u+F_1(z)|^{-1} |\lambda_j u+\mathcal{D}_j F_1(z)| (1+|u|^2) e^{-a|u|^2} \, dA(u)
\\
&\lesssim \int_\C |u+F_1(z)|^{-1} (1+|u|)(1+|u|^2) e^{-a|u|^2} \, dA(u)
\\
&\lesssim \int_\C |u+F_1(z)|^{-1} e^{-\tfrac{a}{2}|u|^2} \, dA(u)
\\
&\leq \int_\C |u|^{-1} e^{-\tfrac{a}{2}|u|^2} \, dA(u) \lesssim 1,
\end{align*}
and an analogous quantity with $(F_1(w), \mathcal{D}_k F_1(w))$ in lieu of $(F_1(z), \mathcal{D}_j F_1(z))$. This completes the proof.
\end{proof}

\section{Variance of Aggregated Charge}\label{sec_var}
We now state and prove the following more precise version of Theorem \ref{th1}.
\begin{theorem}\label{th1plus}
Let $F$ be the GWHF \eqref{eq_F} and assume \eqref{A1}, \eqref{A2}, \eqref{A3}, \eqref{A4}, \eqref{A5}, and \eqref{A6}. Then there exists a model-dependent constant $c$ such that
\begin{align*}%
\sup_{w \in \mathbb{C}} \sup_{R \geq 1} \tfrac{1}{R} \mathrm{Var} \Big[ \sum_{|z-w| \leq R} \charge_z \Big] \leq c < \infty.
\end{align*}
\end{theorem}
\begin{proof}
We restrict attention to the ball $B=B_R(0)$. The conclusions then extend to balls centered at any point because the distribution of $F_0$ is invariant under twisted shifts \eqref{eq_ts}, while all assumptions on $F_1$ also hold for $\mathcal{T}_\xi F_1$ and the constants we rely on are model dependent, cf. Remark \ref{rem_shift_charge}.
	
Using \eqref{eq_charge2} we \emph{formally} expand the variance as
\begin{align}\label{eq_form}
	\begin{aligned}
		4 \pi^2 \cdot
		\var \big[ \charge_{B} \big] &=
		\int_{\partial B} \int_{\partial B}
		\Big(\cov \left[ \tfrac{\mathcal{D}_1 F(z) }{F(z)} ,
		{\tfrac{\mathcal{D}_1 F(w) }{F(w)}} \right] \,dz \,d\bar{w}
		+
		\cov \left[ \tfrac{\mathcal{D}_1 F(z) }{F(z)} ,
		{\tfrac{\mathcal{D}_2 F(w) }{{F(w)}}} \right] \,dz \,dw 
		\\
		&\qquad +
		\cov \left[ \tfrac{\mathcal{D}_2 F(z) }{F(z)} ,
		{\tfrac{\mathcal{D}_1 F(w) }{{F(w)}}} \right] \,d\bar{z} \,d\bar{w}
		+
		\cov \left[ \tfrac{\mathcal{D}_2 F(z) }{F(z)} ,
		{\tfrac{\mathcal{D}_2 F(w) }{{F(w)}}} \right] \,d\bar{z} \,dw
		\Big).
	\end{aligned}
\end{align}
We now derive estimates that justify the expansion \eqref{eq_form} and bound each of the terms in it. We make use of some estimates for convolutions proved in Section \ref{sec_ce}.

Let $L>0$ be larger than the corresponding constants in Corollary \ref{coro1} and Proposition \ref{prop_a}. Let
\begin{align}\label{eq_eps}
\varepsilon := \inf_{0<|z| \leq L} \frac{1-|H(z)|^2}{|z|^2}.
\end{align}
Proposition \ref{prop_hess} and
Assumption \eqref{A4} imply that $\varepsilon>0$, cf. \eqref{eq_H_lower_bound}.

Fix $p \in (1,2)$ and let $q \in (2,\infty)$ be the H\"older conjugate $1/p+1/q=1$. Fix initially $j,k\in\{1,2\}$.

\noindent {\bf Step 1}. Recall that, by Corollary \ref{coro1}, \eqref{eq_p1} holds. Hence,
\begin{align}\label{zz1}
\int_{z \in \partial B} \int_{w \in \partial B} \E \big[ \big|\tfrac{\mathcal{D}_j F(z) }{F(z)} \big|\big] \cdot \E \big[ \big|\tfrac{\mathcal{D}_k F(w) }{F(w)} \big|\big] \, |dz| |dw| \lesssim R^2,
\end{align}
while, by Lemma \ref{lemma_L}  (with $\delta=0$),
\begin{align}\label{eq_y3}
\int_{z \in \partial B} \int_{w \in \partial B, |z-w| \leq L} \E \big[ \big|\tfrac{\mathcal{D}_j F(z) }{F(z)} \big|\big] \cdot \E \big[ \big|\tfrac{\mathcal{D}_k F(w) }{F(w)} \big|\big] \, |dz| |dw| \leq C' L R \lesssim R.
\end{align}

\noindent {\bf Step 2}. We use H\"older's inequality and Lemma \ref{lemma_c} to estimate
\begin{align}\label{y4}
\begin{aligned}
	\E \Big[ \big|\frac{\mathcal{D}_j F(z) }{F(z)} \cdot
	{\frac{\mathcal{D}_k F(w) }{{F(w)}}} \big| \Big]
	&\leq \E \Big[ \big|{\mathcal{D}_j F(z)} \cdot
	\mathcal{D}_k F(w) \big|^q \Big]^{1/q}
	\cdot
	\E \Big[ \big|{F(z)} \cdot
	F(w) \big|^{-p} \Big]^{1/p}
	\\
	&\lesssim \E\Big[ \big|{F(z)} \cdot
	F(w) \big|^{-p} \Big]^{1/p}.
\end{aligned}
\end{align}

Set
\begin{align*}
E_1 &:= \{(z,w) \in \partial B \times \partial B: |z-w| \leq L, |H(z-w)| \geq 1/2\},
\\
E_2 &:= \{(z,w) \in \partial B \times \partial B: |z-w| \leq L, |H(z-w)| < 1/2\},
\\
E_3 &:= \{(z,w) \in \partial B \times \partial B: |z-w| >L\}.
\end{align*}
	
\noindent {\bf Step 3}. Let $(z,w) \in E_1$. We recall \eqref{eq_aaaa} and use \eqref{y4} to further estimate
\begin{align*}
	\E \Big[ \big|\frac{\mathcal{D}_j F(z) }{F(z)} \cdot
	{\frac{\mathcal{D}_k F(w) }{{F(w)}}} \big| \Big]
	&\lesssim \E \Big[ \big|{F(z)} \cdot
	F(w) \big|^{-p} \Big]^{1/p}
	\\
	&\lesssim (1-|H(z-w)|^2)^{1/p-1}
	\leq \varepsilon^{1/p-1} |z-w|^{2(1/p-1)},
\end{align*}
where we used the definition \eqref{eq_eps}. Since $2(1/p-1) > -1$, Lemma \ref{lemma_L} (with $\delta=2(1-1/p)$) gives
\begin{align}\label{zz2}
 \int_{(z,w) \in E_1} \E \Big[ \big|\tfrac{\mathcal{D}_j F(z) }{F(z)} \cdot
 {\tfrac{\mathcal{D}_k F(w) }{{F(w)}}} \big| \Big] \, |dz| |dw|
 \lesssim \int_{z\in\partial B} \int_{\substack{w\in\partial B,\\|z-w| \leq L}} |z-w|^{2(1/p-1)} \, |dw| |dz| \leq CR.
\end{align}
This estimate, together with \eqref{eq_y3}, shows that
\begin{align}\label{yy1}
 \int_{(z,w) \in E_1}
\big|\cov \Big[ \tfrac{\mathcal{D}_j F(z) }{F(z)} ,
{\tfrac{\mathcal{D}_k F(w) }{{F(w)}}} \Big] \big|
\,|dz| \,|dw| \leq C R.
\end{align}

\noindent {\bf Step 4}. For $(z,w) \in E_2$, Lemma \ref{lemma_c} gives $\E\Big[ \big|{F(z)} \cdot F(w) \big|^{-p} \Big]^{1/p} \lesssim 1$. Thus, we combine \eqref{y4} with Lemma \ref{lemma_L} to conclude
\begin{align}\label{zz3}
	\int_{(z,w) \in E_2}
	\E \big[ \big|\tfrac{\mathcal{D}_j F(z) }{F(z)} \cdot
	{\tfrac{\mathcal{D}_k F(w) }{{F(w)}}} \big| \big]
	\,|dz| \,|dw| \lesssim 
	\int_{z \in \partial B} \int_{\substack{w \in \partial B,\\ |z-w| \leq L}} \E\big[ \big|{F(z)} \cdot F(w) \big|^{-p} \big]^{1/p} \, |dz| |dw| \lesssim R.
\end{align}
Together with \eqref{eq_y3}, this shows that
\begin{align}\label{yy2}
	\int_{(z,w) \in E_2}
	\big|\cov \big[ \tfrac{\mathcal{D}_j F(z) }{F(z)} ,
	{\tfrac{\mathcal{D}_k F(w) }{{F(w)}}} \big] \big|
	\,|dz| \,|dw| \leq C R.
\end{align}

\noindent {\bf Step 5}. We combine Proposition \ref{prop_a} and Lemma \ref{lemma_b}  to estimate
\begin{align}\label{yy3}
\begin{aligned}
	\int_{(z,w) \in E_3}
	\big|\cov \Big[ \tfrac{\mathcal{D}_j F(z) }{F(z)} ,
	{\tfrac{\mathcal{D}_k F(w) }{{F(w)}}} \Big] \big|
	\,|dz| \,|dw| &\lesssim \int_{\partial B} \int_{\partial B} (1+|z-w|)^{-2} \, |dz| |dw|
	\\
	&\lesssim \int_{\partial B} \, |dz| \leq CR.
\end{aligned}
\end{align}
In addition, by Corollary \ref{coro1}, \eqref{eq_p1b} holds and
\begin{align}\label{zz4}
\int_{(z,w) \in E_3}	\E \big[ \big|\tfrac{\mathcal{D}_j F(z) }{F(z)} \cdot
{\tfrac{\mathcal{D}_k F(w) }{{F(w)}}} \big| \big]
\,|dz| \,|dw| \lesssim R^2.
\end{align}

\noindent {\bf Step 6}.	
Combining \eqref{yy1}, \eqref{yy2}, and \eqref{yy3} we conclude that
\begin{align}\label{f1}
	\int_{\partial B} \int_{\partial B}
	\big|\cov \Big[ \tfrac{\mathcal{D}_j F(z) }{F(z)} ,
	{\tfrac{\mathcal{D}_k F(w) }{{F(w)}}} \Big] \big|
	\,|dz| \,|dw| \lesssim R, \qquad j,k=1,2.
\end{align}
In addition, combining \eqref{zz2}, \eqref{zz3}, and \eqref{zz4} we conclude that
\begin{align}\label{zz5}
	\int_{\partial B} \int_{\partial B}	\E \big[ \big|\tfrac{\mathcal{D}_j F(z) }{F(z)} \cdot
	{\tfrac{\mathcal{D}_k F(w) }{{F(w)}}} \big| \big]
	\,|dz| \,|dw| \lesssim R^2.
\end{align}
	
Finally, we can justify the expansion \eqref{eq_form}: the interchange of integration and covariance is justified by \eqref{zz1} and \eqref{zz5}. We now bound each of the terms in \eqref{eq_form} with \eqref{f1} to conclude that $\var \big[ \charge_{B} \big] \lesssim R$, as desired.
\end{proof}

 \section{Regularity of Level-Crossing Statistics}\label{sec_reg}

Let $F$ be a \emph{zero-mean} GWHF satisfying the general assumptions. The goal of this section is to show that level-crossing statistics $\#\{z \in K: F(z)=u\}$
depend continuously on the level value $u$ in quadratic mean. For stationary processes, this follows by simply inspecting explicit formulae; see, e.g., \cite{MR0358956}. In our case, we will need to develop estimates for the densities of the level-crossing statistics.

\subsection{Improved covariance bounds in diagonal directions}\label{sec_imp}
Under the general assumptions, the vector $(F(z), F(w))$ is non-degenerate for $z\not=w$ and we will denote its covariance matrix by
\begin{equation}\label{eq_Gamma}
    \Gamma(z,w) = \cov[(F(z),F(w))].
\end{equation} 
By  \eqref{eq_H_lower_bound}, the inverse covariance matrix $\Gamma^{-1} (z,w)$ (henceforth, we will write $\Gamma^{-1}$ and omit the dependency on $z$ and $w$ if the arguments are clear from context) has norm $\lesssim O(|z-w|^{-2})$ (see \eqref{eq_lex} below). We now make the crucial observation that $\Gamma^{-1}$ is much better conditioned when acting on diagonal vectors.
\begin{lemma}\label{lemma_K}
Let $F$ be a \emph{zero-mean} GWHF satisfying the general assumptions. Then for each $R>0$ there exists a model-dependent constant $C_R<\infty$, such that for distinct $z,w \in B_R(0) \subset \mathbb{C}$:
\begin{align}\label{eq_le1}
\var\big[F(z)-F(w)\big] \leq C_R |z-w|^2,
\\\label{eq_le2}
\left| \Gamma^{-1} \cdot (1,1)^\intercal \right| \leq \frac{C_R}{\min\{1,|z-w|\}}.
\\\label{eq_le3}
\left| (1,1) \cdot \Gamma^{-1} \cdot (1,1)^\intercal \right| \leq C_R.
\end{align}
(Here, $\Gamma=\Gamma(z,w)$ is the covariance matrix \eqref{eq_Gamma}.)
\end{lemma}
\begin{proof}
Since $z\not= w$, $(F(z),F(w))$ is a non-degenerate jointly Gaussian vector. Specifically, writing 
$K(z,w)=H(z-w)e^{i \Im(z\bar{w})}$, its covariance is
\begin{align}\label{eq_lep} 
\Gamma(z,w) = \begin{bmatrix}
1 & K(z,w)\\
\overline{K(z,w)} & 1
\end{bmatrix}
\end{align}
and has inverse
\begin{align}
\Gamma^{-1}(z,w) = 
\frac{1}{1-|H(z-w)|^2}
\begin{bmatrix}
1 & -K(z,w)\\
-\overline{K(z,w)}  & 1
\end{bmatrix}.
\end{align}
A Taylor expansion in $z$ around $w$ gives 
\begin{align*}
K(z,w)&= 1+ \partial_1K(w,w)(z-w)+ \bar \partial_1K(w,w)\overline{(z-w)}+ O(|z-w|^2)
\end{align*}
where $\partial_1$ and $\bar \partial_1$ denote derivatives with respect to the first variable,
and the implied constants may depend on $R$.
Hence, 
\begin{align*}
2 \Re K(z,w) 
= K(z,w)+K(w,z)
& = 2+ (\partial_1K(w,w)-\partial_1K(z,z))(z-w)
 \\
 & \quad +(\bar \partial_1K(w,w)-\bar \partial_1K(z,z))\overline{(z-w)}+O(|z-w|^2).
\end{align*}
Since $\partial_1 K(w,w)- \partial_1 K(z,z) = \frac{\overline{w-z}}{2}$ and $\bar \partial_1K(w,w)-\bar \partial_1K(z,z) = \frac{z-w}{2}$ are $O(|z-w|)$, we see that
\begin{align}\label{eq_lea}
1- \Re K(z,w) = O(|z-w|^2),
\qquad
1-K(z,w) = O(|z-w|).
\end{align}

Now \eqref{eq_le1} follows
since \[\var\big[F(z)-F(w)\big] = \mathbb{E}\left[|F(z)-F(w)|^2\right] = 2(1- \Re K(z,w)).\]

Second, we compute
\begin{align*}
\Gamma^{-1}(z,w) \cdot (1,1)^\intercal 
=\left(1-|H(z-w)|^2\right)^{-1} \big(1 - K(z,w), 1 - \overline{K(z,w)}\big)^\intercal,
\end{align*}
and
\begin{align*}
(1,1) \cdot \Gamma^{-1}(z,w) \cdot (1,1)^\intercal 
=2 \left(1-|H(z-w)|^2\right)^{-1} \left(1 - \Re K(z,w)
\right).
\end{align*}
By  \eqref{eq_H_lower_bound}, 
\eqref{eq_lea} readily implies \eqref{eq_le2} and \eqref{eq_le3}.
\end{proof}

\subsection{Boundedness of the intensity functions}
We define the \emph{one and two point functions} 
$\rho\colon \mathbb{C} \to [0,\infty]$,
$\tau\colon \mathbb{C}^2 \to [0,\infty]$ (associated with the level $u \in \mathbb{C}$)  by
\begin{align}\label{eq_sp1}
\rho(z,u) &=   \E \left[ |\det DF(z)| \,\big|\, F(z)=u \right] p_{F(z)}(u),
\\
\label{eq_sp2}
\tau(z,w,u)&= \begin{cases}
\E \left[ |\det DF(z)| \cdot |
\det DF(w)|\,\big|\, F(z)=F(w)=u \right] p_{F(z), F(w)}(u,u), &z\not= w
\\
0, &z=w
\end{cases},
\end{align}
where $p_{F(z), F(w)}$ is the joint probability density of $(F(z), F(w))$ and the conditional expectation is defined by Gaussian regression; see, e.g., \cite[Proposition 1.2]{level}. The two point function is well-defined because \eqref{A4} implies that $(F(z),F(w))$ is non-singular for $z\not=w$.

According to the Kac-Rice formula \cite{level, adler}, the one and two point functions provide densities for the first and second factorial moments of the number of level crossings $\{F=u\}$ within a test set. The following key result estimates the two-point function associated with general level crossings.

\begin{prop}\label{prop_r}
Let $F$ be a \emph{zero-mean} GWHF satisfying the general assumptions. Then for each compact set $K \subset \C$ there exists a model-dependent constant $C_K<\infty$, such that 
\begin{align}\label{eq_sg}
\tau(z,w,u) \leq C_K, \qquad z,w,u \in K.
\end{align}
Similarly, for each compact set $K \subset \mathbb{C}$ there exists a model-dependent constant $C_K<\infty$, such that $\sup_{z,u\in K} \rho(z,u) \leq C_K$.
\end{prop}

\begin{proof}
We start by considering the bound \eqref{eq_sg} on the second intensity. 
Let us fix $(z,w,u) \in K$ with $z\not=w$. Without loss of generality we assume that $K$ is a ball. Throughout the proof we let $C_K$ denote a finite constant that may depend on $K$ and the model. Other dependencies are noted with further subscripts. The particular values of $C_K$ may change from line to line.

First we consider the factor $p_{F(z), F(w)}(u,u)$ in \eqref{eq_sp2} and show that
\begin{align}\label{eq_s1}
p_{F(z), F(w)}(u,u) \leq C_K |z-w|^{-2}.
\end{align}
Since $z\not= w$, $(F(z),F(w))$ is a non-degenerate jointly Gaussian vector. Specifically, its covariance
is given by \eqref{eq_lep} and its determinant satisfies
\begin{align} \label{eq_lex}
  \det \Gamma(z,w) = 1-|H(z-w)|^2 \geq c \min\{1,|z-w|^2\}.
\end{align}
As a consequence, the probability density 
\begin{align*}
p_{F(z), F(w)}(u,v)  &= \frac{1}{\pi^2\det \Gamma(z,w)}
\exp\Big({-\big(u,v\big)^* \,\Gamma(z,w)^{-1} \, \big(u,v\big)}\Big)
\end{align*}
satisfies \eqref{eq_s1}.

Thus, to prove  \eqref{eq_sg}, it remains to show that the other factor in \eqref{eq_sp2} satisfies
$$
\E\left( |\det(DF(z))| \cdot |\det(DF(w)) |\,\big|\, F(z)=F(w)=u \right)
\leq
C_K |z-w|^{2}.
$$ 
This will be done in several steps.

\noindent {\bf Step 1}. By Cauchy-Schwarz,
\begin{align*}
&\E\left( |\det(DF(z))| \cdot |\det(DF(w)) |\,\big|\, F(z)=F(w)=u \right)
\\
&\qquad \leq \left[
\E\left( |\det(DF(z))|^2\,\big|\, F(z)=F(w)=u \right)\right]^{1/2}
\cdot
\left[
\E\left( |\det(DF(w))|^2\,\big|\, F(z)=F(w)=u \right)\right]^{1/2}.
\end{align*}
Hence, by symmetry in $z$ and $w$, it remains to show that
\begin{align}\label{eq_detsqcond}
    \E\left( |\det(DF(z))|^2\,\big|\, F(z)=F(w)=u \right) \leq C_K |z-w|^2.
\end{align}

\smallskip

\noindent {\bf Step 2}. 
We expand $F$ around $z$ to obtain 
\begin{align}\label{eq_taylor}
F(w)
& =
F(z) + (w-z) \partial F(z) + \overline{(w-z)} \bar \partial F(z)
\\*
& \quad + (w-z)^2 R_1(z,w) + |w-z|^2 R_2(z,w) + \overline{(w-z)}^2  R_3(z,w),
\notag 
\end{align}
where $R_k$ are zero-mean Gaussian random functions resulting from Taylor's theorem with integral remainder:
\begin{align*}
R_1(z,w)&= \int_0^1 \partial^2 F(z+t(w-z)) (1-t) dt, \\
R_2(z,w)&= 2 \int_0^1 \partial \bar \partial F(z+t(w-z)) (1-t) dt, \\
R_3(z,w)&= \int_0^1 \bar \partial^2 F(z+t(w-z)) (1-t) dt. 
\end{align*}
We introduce, for $v \in \C$ with $|v|=1$, the notation 
\begin{align}\label{eq_nota}
\partial_v F(z)= v \partial F(z) + \bar v \bar \partial F(z)
\end{align}
for the corresponding directional derivative.
Considering the unit vector $v := (w-z)/|z-w|$, this allows us to rewrite 
\eqref{eq_taylor} as
\begin{align}\label{eq_taylorrw}
\partial_v F(z)
& =
\frac{F(w) - F(z)}{|z-w|}
 - \frac{(w-z)^2}{|z-w|} R_1(z,w) - |z-w| R_2(z,w) - \frac{\overline{(w-z)}^2}{|z-w|}  R_3(z,w).
\end{align}
Furthermore, in the orthogonal coordinate system $\{v, -iv\}$, the Jacobian of $F$ can be rewritten as:
\begin{align}\label{eq_jacv}
\det DF(z) = |\partial F(z)|^2 -|\bar \partial F(z)|^2 = \mathrm{Im} \Big( \partial_v F(z) \overline{\partial_{-iv}F(z)}\Big),
\quad z \in \C, |v|=1.
\end{align}
Noting that we condition on the event $F(z)=F(w)=u$, we can thus bound the left-hand side in \eqref{eq_detsqcond} as
{\small
\begin{align}\label{eq_detsqcondbound1}
    & \E\left( |\det(DF(z))|^2\,\big|\, F(z)=F(w)=u \right) 
    \\
    & \, \leq 
     \E\Big( \big(\big\lvert R_1(z,w)\big\rvert + \big\lvert R_2(z,w)\big\rvert + \big\lvert R_3(z,w)\big\rvert\big)^2
    \big(\lvert\partial F(z) \rvert + \lvert\bar \partial F(z)\rvert\big)^2\,\big|\, F(z)=F(w)=u \Big) \cdot |z-w|^2.
    \notag 
\end{align}}
It remains to bound the conditional expectation on the right-hand side of \eqref{eq_detsqcondbound1} by a constant that depends only on $K$ and the model.

\smallskip

\noindent {\bf Step 3}. For a multi-index $\alpha=(\alpha_1, \alpha_2)$, we denote $\partial_{(z,\bar{z})}^{\alpha} = \partial^{\alpha_1}  \bar\partial^{\alpha_2}$ and let
\begin{align}
E := \sum_{\alpha: |\alpha| \leq 2}
\sup_{\xi \in K} | \partial_{(z,\bar{z})}^{\alpha} F(\xi)|.
\end{align}
Note that the functions $R_j$ in \eqref{eq_taylor}, as well as $\partial F(z)$ and $\bar \partial F(z)$, satisfy
\begin{align}\label{eq_s6}
|X| \lesssim E, \qquad X \in \{\partial F(z), \bar \partial F(z), R_1(z,w), R_2(z,w), R_3(z,w)\}.
\end{align}
We will show that for $p \geq 1$ 
\begin{align}\label{eq_s5b}
\mathbb{E}[E^p] \leq C_{K,p}.
\end{align}
Let us first fix a multi-index $\alpha$ with $|\alpha| \leq 2$ and set \[L^\alpha(\xi,\chi) := \mathbb{E}\Big[\,\partial_{(z,\bar{z})}^{\alpha} F(\xi) \cdot \overline{\partial_{(z,\bar{z})}^{\alpha} F(\chi)}\,\Big], \qquad \xi,\chi \in K.\]
While we do not need an explicit expression for $L^\alpha$, we see as in Lemma \ref{lem_precov}, that we can exchange expectation and differentiation and the regularity assumption \eqref{A5} on $H$ then implies that $L^\alpha$ is $C^2$ in the real sense.
Thus, for $\xi,\chi$ in the compact domain $K$,
\begin{align}\label{eq_bo1}
&\var[\partial_{(z,\bar{z})}^{\alpha} F(\xi)] = L^\alpha(\xi,\xi) \leq C_K,
\\\label{eq_bo2}
&\var[\partial_{(z,\bar{z})}^{\alpha} F(\xi) -  \partial_{(z,\bar{z})}^{\alpha} F(\chi)] = 
L^\alpha(\xi,\xi) + L^\alpha(\chi,\chi) - 2 \Re [L^\alpha(\xi,\chi)] \leq C_K |\xi-\chi|.
\end{align}
(The second estimate can be improved but this is not important for us).

We shall invoke Dudley's inequality \cite[Theorem 2.10]{level} \cite[Theorem 1.3.3]{adler}, which estimates $\sup_{\xi \in K}| \partial_{(z,\bar{z})}^{\alpha} F(\xi)|$ in terms of the covering number of $K$ with respect to the so-called canonical distance $\big(\var[\partial_{(z,\bar{z})}^{\alpha} F(\xi) -  \partial^\alpha F(\chi)]\big)^{1/2}$. By \eqref{eq_bo2}, we can bound the logarithm of the covering number in question by a constant times the logarithm of the covering number associated with the Euclidean distance. 
Hence, Dudley's inequality implies that
\begin{align*}
\mathbb{E}\big[\sup_{\xi \in K} |\partial_{(z,\bar{z})}^{\alpha} F(\xi)|\big] \leq C_K.
\end{align*}
Due to the Borell-TIS inequality
\cite[Theorem 2.9]{level} \cite[Theorem 2.1.1]{adler}, the previous estimate, together with \eqref{eq_bo1} yields
\begin{align*}
\mathbb{P}\big[ \sup_{\xi \in K}  |\partial_{(z,\bar{z})}^{\alpha} F(\xi)| >t\big] \leq C_K \exp[-c_K t^2], \qquad t \geq 0,
\end{align*}
 which readily gives \eqref{eq_s5b}. 
\smallskip

\noindent {\bf Step 4}. We show that
for $X \in \{\partial F(z), \bar \partial F(z), R_1(z,w), R_2(z,w), R_3(z,w)\}$ the following estimate holds
\begin{align}\label{eq_s5}
\big|\mathbb{E}\big[X\,|\, F(z)=F(w)=u] \big| \leq C_K.
\end{align}
The conditional expectation in question is the expectation of a certain Gaussian vector $Z_X$---defined by Gaussian regression applied to the (zero-mean, circularly symmetric) Gaussian vector $(X,F(z), F(w))$ \cite[Proposition 1.2]{level}. Specifically, the conditional mean is
\begin{align}\label{eq_EZX}
\mathbb{E}[Z_X] = \big(\cov[X,F(z)], \cov[X,F(w)]\big) \cdot \Gamma^{-1}(z,w) \cdot (u,u)^\intercal,
\end{align}
where we identify $\mathbb{C}^2 \sim \mathbb{C}^{1\times 2}$. 
We split the last expression as
\begin{align}\label{eq_expzx}
\begin{aligned}
\mathbb{E}[Z_X] &= \big(\cov[X,F(w)], \cov[X,F(w)]\big) \cdot \Gamma^{-1}(z,w) \cdot (u,u)^\intercal \, + 
\\
&\qquad\qquad
\big(\cov[X,F(z)-F(w)], 0\big) \cdot \Gamma^{-1}(z,w) \cdot (u,u)^\intercal.
\end{aligned}
\end{align}
We first recall that by \eqref{eq_s5b} and \eqref{eq_s6}, $\var[X], \var[F(z)], \var[F(w)] \leq C_K$, and consequently also
$|\cov[X,F(z)]|, |\cov[X,F(w)]| \leq C_K$.
Thus, the first term on the right-hand side of \eqref{eq_expzx} can be estimated with Lemma \ref{lemma_K} as
\begin{align*}
&\left|\big(\cov[X,F(w)], \cov[X,F(w)]\big) \cdot \Gamma^{-1}(z,w) \cdot (u,u)^\intercal \right|
\\
&\qquad\leq \left| \cov[X,F(w)] \right | \cdot |u| \cdot \left| (1,1) \cdot \Gamma^{-1}(z,w) \cdot (1,1)^\intercal \right|
\leq C_K.
\end{align*}
Similarly, we use Lemma \ref{lemma_K} to estimate
the second term on the right-hand side of \eqref{eq_expzx} as
\begin{align*}
&\left|\big(\cov[X,F(z)-F(w)], 0\big) \cdot \Gamma^{-1}(z,w) \cdot (u,u)^\intercal\right|
\\
&\qquad \leq
\left|\cov[X,F(z)-F(w)]\right| \cdot |u| \cdot \left| \Gamma^{-1}(z,w) \cdot (1,1)^\intercal\right|
\\ 
&\qquad \leq \var[X]^{1/2} \cdot \var \big[F(z)-F(w)\big]^{1/2} \cdot |u| \cdot \left|  \Gamma^{-1}(z,w)  \cdot (1,1)^\intercal \right| \leq C_K,
\end{align*}
which finishes the proof of \eqref{eq_s5}.

\smallskip

\noindent {\bf Step 5}. 
We finally show that for every $p \geq 1$ and \[X \in \{\partial F(z), \bar \partial F(z), R_1(z,w), R_2(z,w), R_3(z,w)\}\] the following estimate holds
\begin{align}\label{eq_s2}
\mathbb{E}\big[|X|^p\,|\, F(z)=F(w)=u\big]  \leq C_{p,K}.
\end{align}
This then can be combined with \eqref{eq_detsqcondbound1} and concludes the proof of \eqref{eq_sg}.

First note that we want to bound the $p$-th absolute moment of the vector $Z_X$ described in Step 4.
Its expectation is bounded as in \eqref{eq_s5}
while its variance satisfies
\begin{align*}
\mathrm{Var}[Z_X] \leq \mathrm{Var}[X],
\end{align*}
because the conditional expectation map is an orthogonal projection. On the other hand, by 
\eqref{eq_s6} and \eqref{eq_s5},
\[\mathrm{Var}[X] \lesssim \mathbb{E}[E^2] \leq C_K.\]
Since $Z_X$ is normally distributed, it follows that
\begin{align*}
\mathbb{E}\big[|X|^p\,|\, F(z)=F(w)=u]
= \mathbb{E}\big[|Z_X|^p]
\lesssim \left|\mathbb{E}\big[Z_X]\right|^p + 
\left[\mathrm{Var}\big[Z_X]\right]^{p/2}
\leq C_{K,p}.
\end{align*}
This proves \eqref{eq_s2}.

\smallskip

\noindent {\bf Step 6}. 
Finally, we prove the bound for the one-point function. 
This is significantly easier, as the probability density function $p_{F(z)}(u)$ can be bounded by a constant and, thus, we only have to bound the conditional expectation in \eqref{eq_sp1} by a constant $C_K$. 
Let $z \in K$, write $\det DF(z) = |\partial F(z)|^2 - |\bar\partial F(z)|^2$ and estimate
\begin{align*}
&\E \left[ |\det DF(z)| \,\big|\, F(z)=u \right] \leq 
\E \left[ |\partial F(z)|^2 \,\big|\, F(z)=u \right]
+\E \left[ |\bar \partial F(z)|^2 \,\big|\, F(z)=u \right]
\\
&\quad\lesssim
\left|\E \left[ \partial F(z)\,\big|\, F(z)=u \right]\right|^2
+
\left|\E \left[ \bar\partial F(z)\,\big|\, F(z)=u \right]\right|^2
\\*
&\qquad\qquad
+ \var[\partial F(z)\,|\,F(z)=u]
+ \var[\bar\partial F(z)\,|\,F(z)=u]
\\
&\quad\leq
\left|\E \left[ \partial F(z)\,\big|\, F(z)=u \right]\right|^2
+
\left|\E \left[ \bar\partial F(z)\,\big|\, F(z)=u \right]\right|^2 +
\var[\partial F(z)]
+ \var[\bar\partial F(z)].
\end{align*}
As in Step 2, $\var[\partial F(z)]
+ \var[\bar\partial F(z)] \leq C_K$. In addition,
\begin{align*}
\E \left[ \partial F(z)\,\big|\, F(z)=u \right] = \var[F(z)]^{-1} \cov[F(z), \partial F(z)] u,
\\
\E \left[ \bar\partial F(z)\,\big|\, F(z)=u \right] = \var[F(z)]^{-1} \cov[F(z), \bar\partial F(z)] u.
\end{align*}
Finally, note that $\var[F(z)] = 1$
while $|\cov[F(z), \partial F(z)]| \leq \var[\partial F(z)]^{1/2} \leq C_K$ and similarly 
$|\cov[F(z), \bar \partial F(z)]| \leq C_K$. This shows that $\rho(z,u) \leq C_K$.
\end{proof}
\begin{rem}
Step 5 in the previous proof is somehow analogous to the estimate on Taylor expansions for stationary processes in \cite[Proposition 1]{MR4648512}; see also \cite[Appendix A]{MR3572325}.
\end{rem}

We can now prove that moments of level crossing statistics are continuous at the origin.

\begin{theorem}
\label{th_coro_r}
Under the hypothesis of Proposition \ref{prop_r}, let $N_u(K) := \#\{z\in K: F(z)=u\}$. Then
\begin{align}\label{eq_coro1}
&\lim_{u\to 0} \mathbb{E}\big[N_u(K)\big] = \mathbb{E}\big[N_0(K)\big],
\\\label{eq_coro2}
&\lim_{u\to 0} \mathbb{E}\big[(N_u(K))^2\big] = \mathbb{E}\big[(N_0(K))^2\big],
\\\label{eq_corox}
&\sup_{|u|\leq 1} \mathbb{E} \big[(N_{u}(K))^2\big] < \infty.
\end{align}
\end{theorem}
\begin{proof}
We consider the intensity functions \eqref{eq_sp1}, \eqref{eq_sp2}. The conditional expectations defining these functions have explicit expressions \cite[Proposition 1.2]{level} in terms of the $C^2$-smooth vector $(F(z),F(w))$ and the $C^6$-smooth matrix \[\cov(F(z),F(w),\partial F(z), \partial F(w), \bar\partial F(z),\bar\partial F(w)),\]
which show that $\lim_{u\to 0}\rho(z,u)=\rho(z,0)$ and, for $z\not= w$, $\lim_{u\to 0}\tau(z,w,u)=\tau(z,w,0)$.

On the other hand, the Kac-Rice formulae \cite[Theorems 6.2 and 6.3]{level} provide the representations
\begin{align}\label{eq_coro3}
&\mathbb{E}\big[N_u(K)\big] = \int_K \rho(z,u) \, dA(z),
\\
\label{eq_coro4}
&\mathbb{E}\big[N_u(K)^2-N_u(K)\big]
=
\int_{K\times K} \tau(z,w,u) \, dA(z,w).
\end{align}
(The formulae are applicable because $F$ has $C^2$ paths; each $F(z)$ is a standard normal variable; for $z\not= w$ the vector $(F(z),F(w))$ has non-singular covariance; while Lemma \ref{lem_pi} shows that zeros are almost surely non-degenerate.)

The uniform bound in Proposition \ref{prop_r} then allows us to exchange $\mathbb{E}$ and $\lim_{u\to 0}$ in \eqref{eq_coro3} and \eqref{eq_coro4}
to obtain \eqref{eq_coro1} and \eqref{eq_coro2}. Similarly, Proposition \ref{prop_r} provides bounds for \eqref{eq_coro3} and \eqref{eq_coro4} which readily implies \eqref{eq_corox}.
\end{proof}
\begin{rem}
The analog of Theorem \ref{th_coro_r} for stationary processes follows by direct inspection of explicit formulae for the point intensities; see, e.g., \cite{MR0358956}.
\end{rem}

\subsection{Quadratic convergence of level crossing statistics}
We can now prove that level crossing statistics are continuous in quadratic mean.
\begin{lemma}\label{lem_ae}
Under the hypothesis of Proposition \ref{prop_r}, let $N_u(K) := \#\{z\in K: F(z)=u\}$. Then $\mathbb{E} |N_{u}(K) - N_{0}(K)|^2 \to 0$, as $u \to 0$.
\end{lemma}
\begin{proof}
Let $\{u_k: k \geq 1\} \subset \mathbb{C}$ be an arbitrary sequence with $u_k \to 0$. Let us show that $\mathbb{E} \big[|N_{u_k}(K) - N_{0}(K)|^2\big] \to 0$ as $k \to \infty$. By Theorem \ref{th_coro_r},
$\mathbb{E}\big[(N_{u_k}(K))^2\big] \longrightarrow \mathbb{E}\big[(N_0(K))^2\big]$ as $k \to \infty$. Hence, by the Br\'ezis-Lieb Lemma \cite{MR0699419, MR1817225}, it is enough to show that $N_{u_k}(K) \to N_0(K)$ almost surely. This is a standard argument \cite{adler,level} and we just briefly sketch it.

By Lemma \ref{lem_pi}, the zeros of $F$ are almost surely non-degenerate, cf.\ \eqref{eq_nd}. Fix one such realization of $F$. Then the set $\{F=0\}\cap K$ must be finite. In addition, by the inverse function theorem, around each zero $z$ there is a neighborhood $V(z)$ such that $F\colon V(z) \to F(V(z))$ is a homeomorphism. These neighborhoods can be assumed to be disjoint, and their union is denoted $U$. For all sufficiently large $k$, the equation $F(z)=u_k$ has a unique solution on each $V(z)$ and therefore $N_{u_k}(K) \geq N_0(K)$. On the other hand, suppose that, after passing to a subsequence of $u_k$, one can find solutions $F(z_k)=u_k$ with $z_k \notin U$. After passing to a further subsequence, $z_k \to z_* \in K \setminus U$, while, by continuity $F(z_*)=0$ and therefore $z_* \in U$. 
This contradiction concludes the proof.
\end{proof}

\section{Chaos Expansion on the Complex Plane}\label{sec_chaos}
\subsection{Complex Hermite polynomials}
We will employ the \emph{complex Hermite polynomials} $H_{j,k}$, defined via the generating function identity as
\begin{align*}
e^{u z + v \bar z -uv}
= \sum_{j,k \geq0 } \frac{u^j v^k}{j!k!}H_{j,k}(z, \bar z), \qquad z \in \mathbb{C},
\end{align*}
or, more explicitly, by
\begin{align}\label{eq_hermite}
H_{j,k}(z,\bar z)= \sum_{r=0}^{\min(j,k)}(-1)^r r! \binom{j}{r}\binom{k}{r}z^{j-r}\bar z^{k-r}, \qquad j,k \geq0,
\end{align}
see, e.g., \cite{is16}.
Using the explicit expression, complex Hermite polynomials can be written in terms of associated Laguerre polynomials \cite{is16}. We shall only be interested in the following expression for diagonal index pairs:
\begin{align}\label{eq_lh}
H_{k,k}(z, \bar z) = (-1)^k k! L_k(|z|^2),
\end{align}
where $L_k$ is the standard Laguerre polynomial
\begin{align}\label{eq_lagn}
L_{k}(t)=\sum_{j=0}^{k} (-1)^j \binom{k}{j} \frac{t^j}{j!}.
\end{align} 
Complex Hermite polynomials satisfy the orthogonality relation
\begin{align*}
\int_\C H_{j,k}(z,\bar z) \overline{H_{l,m}(z,\bar z)}e^{-|z|^2} \,\frac{dA(z)}{\pi} = j! k! \delta_{j=l} \delta_{k=m},
\end{align*}
and they form an orthogonal basis of $L^2\big(\mathbb{C},e^{-|z|^2}\,dA(z)/\pi\big)$. By \eqref{eq_lh}, for any radial function $f \in L^2\big(\mathbb{C},e^{-|z|^2}\,dA(z)/\pi\big)$ we have 
\begin{align*}
f(z) = \sum_{k \geq 0} \tfrac{(-1)^k} {k!} a_k H_{k,k}(z,\bar z)
= \sum_{k \geq 0} a_k L_k(|z|^2), \qquad  z\in\mathbb{C} %
\end{align*}
for appropriate coefficients $a_k$.
 
Complex Hermite polynomials were discovered by \cite{ito1952complex} in the context of complex multiple Wiener integrals. They share many properties of standard Hermite polynomials \cite{chen2019complex} \cite{is16} \cite{gh13}. To prove Theorem \ref{th2}, we shall employ a so-called Wiener chaos expansion in complex Hermite polynomials, which is in many ways analogous to the more standard chaos expansion in real Hermite polynomials. 

\subsection{The planar chaos decomposition}\label{sec_plch}
Our presentation is based on Gaussian Hilbert spaces following Janson's book \cite{janson1997gaussian}, although we use somewhat different notation. He presents chaos decompositions in terms of Wick products rather than complex Hermite polynomials, which is just a notational difference. 

Let $F\colon \mathbb{C} \to \mathbb{C}$ be a circularly symmetric Gaussian random function with underlying probability space $(\Omega,\mathbb{P})$, and consider the set of Gaussian variables
\begin{align}\label{eq_defs}
S:=\{ F(z), \mathcal{D}_1F(z), \mathcal{D}_2F(z) \,:\,z \in \C\}.
\end{align}
Let $\mathcal{F}_S$ 
be the sigma-algebra generated by $S$ and $G$ the completion of the linear span of $S$ as a subspace of $L^2(\Omega, \mathcal{F}_S)$. The space $G$, called \emph{the Gaussian space induced by $S$}, is a separable Hilbert space consisting of circularly symmetric complex Gaussian variables.  
Let $\{\xi_n \}_{n=0,1,2,\ldots}$ be some orthonormal basis of $G$. 
Translated to our notation involving complex Hermite polynomials, Proposition 1.34 and Example 3.32 in \cite{janson1997gaussian} tell us that 
\begin{align}\label{eq_basis}
\bigg\{ \prod_{k=0}^\infty \frac{1}{\sqrt{\alpha_k! \beta_k!}} H_{\alpha_k, \beta_k}( \xi_k, \overline{\xi_k}): |\alpha|, |\beta| < \infty \bigg\}
\end{align}
is an orthonormal basis of 
$L^2(\Omega, \mathcal{F}_S)$. Here  $\alpha=(\alpha_k)_{k=0,1,2,\ldots}, \beta=(\beta_k)_{k=0,1,2,\ldots}$ are multi-indices, and we are using the notation $|\alpha|= \sum_k \alpha_k$.

The \emph{$(M,N)$'th chaotic subspace of $L^2(\Omega, \mathcal{F}_S)$} is defined as
\begin{align}\label{eq_ch2}
C_{M,N}= \overline{ \mathrm{span}} \bigg\{
\prod_{k=0}^\infty H_{\alpha_k, \beta_k}( \xi_k, \overline{\xi_k}):
|\alpha| =M, |\beta| = N \bigg\} , \quad N,M \geq 0
\end{align}
and we will denote the orthogonal projection onto $C_{M,N}$ by $Q_{M,N}$.
We therefore have the following \emph{chaos decomposition}:
\begin{align}\label{eq_ch}
L^2(\Omega, \mathcal{F}_S)= \bigoplus_{M,N \geq 0} C_{M,N}.
\end{align}
The decomposition is independent of the choice of the orthonormal basis of $G$; for completeness, we provide a short argument for this fact in Section \ref{sec_chin}.

\subsection{Chaos decomposition of number statistics}\label{sec: chaos decomp}
We now look into the chaotic components of the number statistic
\begin{align*}
N(B) := \# \{z \in B: F(z)=0\}
\end{align*}
associated with a GWHF $F$ and a test set $B \subset \mathbb{C}$. 

This will be done by first considering certain regularized versions of $N(B)$. Let $\chi\colon [0,+\infty) \to [0,+\infty)$ be smooth with $\mathrm{supp}(\chi) \subset [0,1]$ and $\int_0^\infty \chi(t) \, dt =1$. 
For $\epsilon>0$, let 
\begin{align}\label{eq_m1}
\chi_\epsilon(z)= \frac{1}{\pi \epsilon^2} \chi\left(|z|^2/\epsilon^2\right), \qquad z \in \mathbb{C}.
\end{align}
Then $\chi_\epsilon$ is smooth, $\mathrm{supp}(\chi_\epsilon) \subset \bar{B}_\epsilon(0)$ and $\int \chi_\epsilon\,dA=1$.
We define the regularized variables $N^\epsilon(B)$ by 
\begin{align}\label{eq_m2}
N^\epsilon(B) =  \int_B \chi_\epsilon( F(z)) \big| |\mathcal{D}_1 F(z)|^2- |\mathcal{D}_2 F(z)|^2 \big| \,dA(z). 
\end{align}
Note that the expression involves the \emph{covariant Jacobian determinant} 
$|\mathcal{D}_1 F(z)|^2- |\mathcal{D}_2F(z)|^2$
instead of the usual Euclidean Jacobian
$|\partial F(z)|^2- |\bar \partial F(z)|^2$.

The following lemma, proved in Section \ref{sec_plem_lip}, shows that $N^\epsilon(B) \in L^2(\mathcal{F}_S)$.
\begin{lemma}\label{lem_lip}
Let $F$ be a zero-mean GWHF satisfying the general assumptions. Let $\phi\colon \mathbb{C}^3\to\mathbb{C}$ satisfy
\begin{align*}
|\phi(\zeta) - \phi(\zeta')| \leq C |\zeta-\zeta'|(1+|\zeta|^s+|\zeta'|^s), \qquad \zeta,\zeta' \in \mathbb{C}^3, 
\end{align*}
for some constants $C>0$ and $s\geq 1$. Let $B \subset \mathbb{C}$ be a bounded Borel set. For each $n \in \mathbb{N}$ consider a finite cover of $B$ by almost disjoint cubes $Q_1, \dots Q_{L_n}$ with centers $z_1, \ldots, z_{L_n}$ and diameter $1/n$. Then
\begin{align*}
\sum_{k=1}^{L_n} \phi\big(F(z_k),\mathcal{D}_1 F(z_k), \mathcal{D}_2 F(z_k)\big)
\,\big|B \cap Q_k \big|
\longrightarrow \int_B \phi\big(F(z),\mathcal{D}_1 F(z), \mathcal{D}_2 F(z)\big) \, dA(z)
\end{align*}
in $L^2(\mathrm{d}\mathbb{P})$ as $n \to \infty$, where the right-hand side is defined realization-wise. 
\end{lemma}
We now show that the regularized number statistic converges in quadratic mean.
\begin{prop}\label{prop_reg}
Let $F$ be a \emph{zero-mean} GWHF satisfying the general assumptions and let $B \subset \mathbb{C}$ be a bounded Borel set. Then $\var N(B) < \infty$ and
\begin{align}\label{eq_r0}
N^\epsilon(B) \to N(B) \mbox{ in } L^2\mbox{ as }\epsilon\to0^+.
\end{align}
\end{prop}
\begin{proof}
\noindent {\bf Step 1}. Consider the \emph{Euclidean regularized statistic}
\begin{align*}
\tilde{N}^\epsilon(B) =  \int_B \chi_\epsilon( F(z)) \big| |\partial F(z)|^2- |\bar\partial F(z)|^2 \big| \,dA(z),
\end{align*}
where, in contrast to \eqref{eq_m2} we use the Euclidean Jacobian. Let us show that
\begin{align}\label{eq_r1}
\mathbb{E} \big[\big|\tilde{N}^\epsilon(B) - N(B) \big|^2\big] \longrightarrow 0, \mbox{ as }\epsilon\to0^+.
\end{align}
To prove this, we take an arbitrary sequence $\epsilon_k \to 0^+$ and consider the corresponding limit.

We shall invoke Lemma \ref{lem_ae} and adopt its notation. By the area formula, 
\begin{align*}
\tilde N^{\epsilon_k}(B)= \int_{\C} \chi_{\epsilon_k}(u) N_u(B) \,dA(u)
= \int_{B_1(0)} \chi_1(u) N_{\epsilon_k u}(B) \,dA(u), 
\end{align*}
see, e.g., \cite[Proposition 6.1.]{level}. 
Consequently,
\begin{align*}
\mathbb{E} \Big[\big|
\tilde N^{\epsilon_k}(B) - N(B)\big|^2 \Big]&=
\mathbb{E}\bigg[ \Big|\int_{B_1(0)}  \chi_1(u) \big( N_{\epsilon_k u}(B) - N_0(B) \big) \,dA(u)\Big|^2 \bigg]
\\
&\leq \|\chi_1\|_2^2 \cdot 
\int_{B_1(0)}  \mathbb{E} \Big[ \big| N_{\epsilon_k u}(B) - N_0(B) \big|^2 \Big]\,dA(u).
\end{align*}
The integrand in the last expression is bounded due to Theorem \ref{th_coro_r}, and converges to $0$ pointwise due to Lemma \ref{lem_ae}. Thus, the dominated convergence theorem yields \eqref{eq_r1}. 
\smallskip

\noindent {\bf Step 2}. Let us show that
\begin{align}\label{eq_r2}
\mathbb{E} \big[\big|\tilde{N}^\epsilon(B) - N^\epsilon(B) \big|^2\big] \longrightarrow 0, \mbox{ as }\epsilon\to0^+.
\end{align}
Let us first estimate
\begin{align*}
&\big|\tilde{N}^\epsilon(B) - N^\epsilon(B) \big|
\\
&\quad
\leq \int_B \chi_{\epsilon}(F(z)) \Big|
\big|
|\partial F(z)|^2 - |\bar\partial F(z)|^2\big|
-
\big||\partial F(z) - \tfrac{\bar z}{2} F(z)|^2 - |\bar\partial F(z) + \tfrac{z}{2} F(z)|^2\big|
\Big|\, dA(z)
\\
&\quad
\leq \int_B \chi_{\epsilon}(F(z))
\Big(
\big|
|\partial F(z)|^2 - |\partial F(z) - \tfrac{\bar z}{2} F(z)|^2 \big|
+
\big|
|\bar\partial F(z)|^2 - |\bar\partial F(z) + \tfrac{z}{2} F(z)|^2 \big|
\Big)
\, dA(z)
\\
&\quad 
\lesssim \int_B \chi_{\epsilon}(F(z))
|z| |F(z)| \big( |\partial F(z)| + |\bar\partial F(z)| + |z| |F(z)| \big)
\, dA(z)
\\
&\quad 
\leq C_B \cdot \epsilon \cdot
\int_B \chi_{\epsilon}(F(z)) \Phi(z)
\, dA(z),
\end{align*}
where $\Phi(z) = \big( |\partial F(z)| + |\bar\partial F(z)| + 1 \big)$.

Let $p \in (1,2)$, $p' \in (2,\infty)$ its H\"older conjugate, $1/p+1/p'=1$,
and select $q \in (1,2)$ with $pq <2$.
Our assumptions on $F$ imply that
\begin{align*}
\sup_{z,w \in B}
\big(\mathbb{E} |\Phi(z) \Phi(w)|^{p'}\big)^{1/p'} \leq C'_{B,p}.
\end{align*}
In addition, since $1<pq<2$, by Lemma \ref{lemma_c} and Proposition \ref{prop_hess},
\begin{align*}
\mathbb{E} \big[|F(z) F(w)|^{-pq}\big] \leq C_{pq} \max\big\{1,|z-w|^{2(1-pq)}\big\}, \qquad z,w \in \mathbb{C}.
\end{align*}
We also note that $\chi_{\epsilon}(F(z)) \lesssim \epsilon^{q-2} \lvert F(z)\rvert^{-q}$ and further estimate
\begin{align*}
\mathbb{E} \big[\big|\tilde{N}^\epsilon(B) - N^\epsilon(B) \big|^2\big]
&\leq C^2_B \cdot \epsilon^2 \cdot
\int_{B\times B} \mathbb{E} \big[ \chi_{\epsilon}(F(z)) \chi_{\epsilon}(F(w)) \Phi(z) \Phi(w)\big] \, dA(z)dA(w)
\\
&\lesssim C''_{B,p} \cdot \epsilon^2 \cdot 
\int_{B\times B} \big(\mathbb{E}\big[ |\chi_{\epsilon}(F(z)) \chi_{\epsilon}(F(w))|^p\big]\big)^{1/p} \, dA(z)dA(w)
\\
&\lesssim C''_{B,p} \cdot \epsilon^2 \cdot \epsilon^{2q-4} \cdot 
\int_{B\times B} \big(\mathbb{E} \big[|F(z) F(w)|^{-pq}\big]\big)^{1/p} \, dA(z)dA(w)
\\
&\lesssim C''_{B,p,q} \cdot \epsilon^{2q-2} \cdot
\int_{B\times B} \max\{1,|z-w|^{2(1/p-q)}\} \, dA(z)dA(w).
\end{align*}
The previous integral is finite because $0<2(q-1/p)
<2(2/p-1/p)=2/p <2$, while the power of $\epsilon$ is positive, which proves \eqref{eq_r2}.
\smallskip

\noindent {\bf Step 3}. Finally, we combine \eqref{eq_r1} and \eqref{eq_r2} to conclude that $\var N(B) < \infty$ and obtain \eqref{eq_r0}.
\end{proof}

Lemma \ref{lem_lip} shows that  $N^\epsilon(B) \in L^2(\mathcal{F}_S,\mathrm{d}\mathbb{P})$ and, therefore, by Proposition \ref{prop_reg}, \[N(B)\in L^2(\mathcal{F}_S,\mathrm{d}\mathbb{P}).\] Hence, the number statistic $N(B)$ can be expanded as
\begin{align}\label{eq_chN}
N(B) = \sum_{M,N \geq 0} Q_{M,N}(N(B)), \qquad Q_{M,N}(N(B)) \in C_{M,N},
\end{align}
into chaotic components \eqref{eq_ch2} associated with the Gaussian space induced by $S$. We now calculate the chaos expansion of $N(B)$ explicitly in the case of radial twisted kernels.
\begin{theorem}\label{th_chaos}
Let $F$ be a \emph{zero-mean} GWHF satisfying the general assumptions and with a twisted kernel of the form $H(z)=P(|z|^2)$ where $P\colon \mathbb{R} \to \mathbb{C}$ is $C^6$.
Let $B \subset \mathbb{C}$ be a bounded Borel set and consider the number statistic $N(B)$. 
Then $N(B) \in L^2(\mathcal{F}_S, \mathrm{d}\mathbb{P})$, where $\mathcal{F}_S$ is the $\sigma$-algebra generated by $S$ given by \eqref{eq_defs}. 

Consider the chaos decomposition associated with $F$ and let \eqref{eq_chN} be corresponding expansion of $N(B)$.
Let $N,M \geq 0$. Then $Q_{M,N}(N(B))=0$ for $N\not=M$, while  
\begin{align}\label{eq:chaos_expansion}
Q_{N,N}(N(B)) &= \frac{1}{\pi}\sum_{k+l+j = N} c_{k,l} \int_B L_j\left(|F(z)|^2\right) L_k\left(\frac{|\mathcal{D}_1F(z)|^2}{1/2-\Delta H(0)}\right)L_{l}\left(\frac{|\mathcal{D}_2F(z)|^2}{-1/2-\Delta H(0)} \right)\,
dA(z),
\end{align}
where the integral is defined realization-wise
and 
\begin{align} \label{eq:c_kl_formula}
 c_{k,l} 
= \frac{1}{\pi^2}\int_{\C} \int_{\C} \bigl|(-\Delta H(0) + 1/2)|z|^2- (-\Delta H(0) - 1/2)|w|^2\bigr| L_k(|z|^2)\, L_l(|w|^2)  \,e^{-|z|^2-|w|^2} \,dA(z,w).
\end{align}
\end{theorem}
\begin{proof}
\noindent {\bf Step 1}. We have already observed that $N(B) \in L^2(\mathcal{F}_S)$. Let $\varepsilon>0$ and consider the smooth mollifier $\chi_\epsilon$ \eqref{eq_m1} and the regularized number statistic $N^{\epsilon}(B)$ \eqref{eq_m2}.
Since $\chi_\epsilon$ is radial, its expansion in the basis \eqref{eq_basis} has the form
\begin{align*}
\chi_\epsilon(\zeta) = \sum_{j \geq 0} a_{j,\epsilon} (-1)^j \frac{1}{j!} H_{j,j}(\zeta,\bar{\zeta})=\sum_{j \geq 0} a_{j,\epsilon} L_j(|\zeta|^2),
\end{align*}
with convergence in $L^2\big(\mathbb{C},e^{-|\zeta|^2}\,dA(\zeta)/\pi\big)$
and $(a_{j,\epsilon})_{j\geq 0} \in \ell^2$. 
We do not need the exact value of $a_{j, \epsilon}$; we just note that
\begin{align}\label{eq_ajt}
a_{j, \epsilon} \to \frac{1}{\pi} \int_{\C} \delta_0(\zeta) 
L_j(|\zeta|^2)
e^{-|\zeta|^2} \,dA(\zeta) =\frac{1}{\pi}, \qquad \mbox{as } \epsilon \to 0.
\end{align}
We also expand: 
\begin{align*}
\big|(-\Delta H(0) + 1/2)|\zeta_1|^2- (-\Delta H(0) - 1/2)|\zeta_2|^2 \big| 
&= \sum_{k,l  \geq 0 } c_{k,l} L_k(|\zeta_1|^2) L_l(|\zeta_2|^2) 
\end{align*}
with convergence in $L^2\big(\C, \frac{1}{\pi^2} e^{-|\zeta_1|^2-|\zeta_2|^2} dA(\zeta_1) dA(\zeta_2) \big)$.

We thus arrive at the expansion
\begin{align}\label{eq_pc1}
\chi_\epsilon(\zeta_1) \big|(-\Delta H(0) + 1/2)|\zeta_2|^2- (-\Delta H(0) - 1/2)|\zeta_3|^2 \big|
= \sum_{j,k,l\geq0} a_{j, \epsilon} c_{k,l} L_j(|\zeta_1|^2) L_k(|\zeta_2|^2) L_l(|\zeta_3|^2),
\end{align}
with convergence in $L^2\big(\C, \frac{1}{\pi^3} e^{-|\zeta_1|^2-|\zeta_2|^2-|\zeta_3|^2} dA(\zeta_1) dA(\zeta_2) dA(\zeta_3)\big)$.

\smallskip

\noindent {\bf Step 2}. Let $z \in \mathbb{C}$ be arbitrary but fixed. The covariance of the Gaussian vector \[(F(z), \mathcal{D}_1 F(z),\mathcal{D}_2 F(z))\] is given by \eqref{eq:cov_matrix}. Since $H(z)=P(|z|^2)$, it follows that $\partial H(0)=\bar\partial H(0)=\partial^2 H(0) = \bar\partial^2 H(0)=0$. Consequently,
\begin{equation} \label{eq:stand_gauss}
(\xi(z), \xi'(z), \xi''(z)) := 
\bigg( F(z), \frac{\mathcal{D}_1F(z)}{(-\Delta H(0) + 1/2)^{1/2}}, \frac{\mathcal{D}_2F(z)}{(-\Delta H(0) - 1/2)^{1/2}} \bigg)
\end{equation}
is a standard complex Gaussian vector and the map 
\begin{align*}
L^2\big(\C, \tfrac{1}{\pi^3} e^{-|\zeta_1|^2-|\zeta_2|^2-|\zeta_3|^2} dA(\zeta_1) dA(\zeta_2) dA(\zeta_3)\big) 
& \to L^2(\mathrm{d}\mathbb{P})
\\*
f & \mapsto f(\xi(z), \xi'(z), \xi''(z))
\end{align*}
is an isometric embedding. Thus \eqref{eq_pc1} translates into the almost sure equality
\begin{align}
\chi_\epsilon(F(z)) \big||\mathcal{D}_1F(z)|^2- |\mathcal{D}_2F(z)|^2 \big|
= \sum_{j,k,l\geq0} a_{j, \epsilon} c_{k,l} L_j(|\xi(z)|^2) L_k(|\xi'(z)|^2) L_l(|\xi''(z)|^2),
\end{align}
with convergence in quadratic mean for each $z \in \mathbb{C}$.

We shall invoke Lemma \ref{lem_lip} with the functions
\begin{align*}
&\phi_0(\zeta_1,\zeta_2,\zeta_3)=\chi_\epsilon(\zeta_1)
\big| |\zeta_2|^2-|\zeta_3|^2 \big|,
\\
&\phi_{j,k,l}(\zeta_1,\zeta_2,\zeta_3)= L_j(|\zeta_1|^2) L_k(|\zeta_2|^2) L_l(|\zeta_3|^2),
\end{align*}
which satisfy the hypothesis of the lemma (with constants that depend on $j,k,l$). With the notation of Lemma \ref{lem_lip},
\begin{align}\label{eq_opa}
&Q_{M,N}\left[
\sum_{h=1}^{L_n} \chi_\epsilon(F(z_h)) \big||\mathcal{D}_1F(z_h)|^2- |\mathcal{D}_2F(z_h)|^2 \big|\,
\,\big|B \cap Q_h \big|\right]
\\\notag
&\quad=\sum_{h=1}^{L_n} 
\sum_{j,k,l\geq0} a_{j, \epsilon} c_{k,l} 
Q_{M,N}\left[ L_j(\lvert\xi(z_h) \rvert^2) L_{k}(\lvert\xi'(z_h) \rvert^2) L_{l}(\lvert\xi''(z_h) \rvert^2)\right] \,\big|B \cap Q_h \big|
\\\label{eq_opb}
&\quad={\delta_{N=M}} \sum_{h=1}^{L_n} 
\sum_{j+k+l=N} a_{j, \epsilon} c_{k,l}L_j(\lvert\xi(z_h) \rvert^2) L_{k}(\lvert\xi'(z_h) \rvert^2) L_{l}(\lvert\xi''(z_h) \rvert^2) \,\big|B \cap Q_h \big|.
\end{align} 
\smallskip
By Lemma \ref{lem_lip}, as $n \to \infty$, the Riemann sums in \eqref{eq_opa} and \eqref{eq_opb} converge to the corresponding integrals in quadratic mean.
Since $Q_{M,N}$ is continuous in $L^2(\mathrm{d}\mathbb{P})$, we conclude that
\begin{align*}
Q_{M,N}[N^\epsilon(B)] = \delta_{N=M} \int_B
\sum_{j+k+l=N} a_{j, \epsilon} c_{k,l} L_j(|\xi(z)|^2) L_k(|\xi'(z)|^2) L_l(|\xi''(z)|^2) \, dA(z).
\end{align*}
By Proposition \ref{prop_reg}, $Q_{M,N}[N^\epsilon(B)]\longrightarrow Q_{M,N}[N(B)]$ in quadratic mean as $\epsilon\to0^+$, which, together with \eqref{eq_ajt} gives \eqref{eq:chaos_expansion} and the vanishing of the projection for $N\not=M$.
\end{proof}

\section{Non-Hyperuniformity of Uncharged Zeros: Proof of Theorem \ref{th2}} \label{sec_nonhyper}
In this section we prove Theorem~\ref{th2}, regarding the non-hyperuniformity of the uncharged zeros of a \emph{zero-mean} GWHF $F$ with the specific twisted kernel
\begin{equation}\label{eq_H}
H(z) = (1-|z|^2) e^{-\frac 1 2 |z|^2}.
\end{equation}
We analyze the zero set $\{F=0\}$ on the observation disk $B=B_R(0)$ centered at $0$.

\subsection{Lower bound}
Let us consider the chaos decomposition of the number statistic $N(B)$. The gist of the proof is to estimate the second order chaotic component $Q_{2,2}(N(B))$. By the orthogonality of the chaos decomposition (as in~\eqref{eq_ch}), we then have the estimate 
\[
\var[N(B)] \ge \var [Q_{2,2}(N(B))].
\]
We will show that \[\var [Q_{2,2}(N(B))] \gtrsim R^2.\]
(On the other hand, one can check that $\var [Q_{1,1}(N(B))] = O(R)$, which is not enough for our purpose.)

We shall inspect the explicit expression for $Q_{2,2}(N(B))$ given in Theorem \ref{th_chaos}. In our case,
\[
\Delta H(0) = \partial \bar \partial H(0) = -\frac32. 
\]
We follow the notation from Section~\ref{sec: chaos decomp} and specialize~\eqref{eq:stand_gauss} to
\begin{equation}\label{eq_xi}
(\xi(z), \xi'(z), \xi''(z)) = 
\Big( F(z), \tfrac{1}{\sqrt 2} \mathcal{D}_1F(z), \mathcal{D}_2F(z) \Big),
\end{equation}
whose entries are independent and distributed as $\mathcal{N}_\C (0,1)$, for every $z\in \C$. 
By Theorem~\ref{th_chaos}, the projection of $N(B)$ onto $C_{2,2}$ is given by
\begin{equation}\label{eq: Q22}
Q_{2,2}(N(B)) =  \frac{1}{\pi}\int_B \phi(z)\,
dA(z),
\end{equation}
where 
\begin{align}\label{eq: phi}
 \phi(z) &=  \sum_{k+l+j = 2} c_{k,l}L_j\left(|\xi(z)|^2\right) L_k\left(|\xi'(z)|^2\right)L_{l}\left(|\xi''(z)|^2\right)
\\ & = c_{1,0} L_1(|\xi(z)|^2)L_1(|\xi'(z)|^2) + c_{0,1}L_1(|\xi(z)|^2)L_1(|\xi''(z)|^2) + c_{1,1}L_1(|\xi'(z)|^2)L_1(|\xi''(z)|^2) \notag \\
&\qquad + c_{0,0} L_2(|\xi(z)|^2) + c_{2,0} L_2(|\xi'(z)|^2) + c_{0,2} L_2(|\xi''(z)|^2).\notag
\end{align}
Using~\eqref{eq:c_kl_formula}, we compute the coefficients explicitly (see the accompanying notebook \cite{jupy_gwhf_nonhyp}).
\begin{equation}\label{eq_cs}
c_{0,0}= \frac{5}{3}, \quad c_{0,1} = - \frac {1} {9}, \quad c_{1,0}=-\frac{14}{9}, \quad
c_{0,2}=c_{2,0}=\frac {8}{27}, \quad c_{1,1}=-\frac{16}{27}.
\end{equation}
The random variable $Q_{2,2}(N(B))$ has zero mean because it is orthogonal to $1 \in C_{0,0}$. Thus, from~\eqref{eq: Q22},
\begin{equation}\label{eq_stop}
\var [Q_{2,2}(N(B))] = 
\E \left[\frac{1}{\pi^2}\int_B \int_B \phi(z)\phi(w)\, dA(z) dA(w) \right] 
= \frac{1}{\pi^2}\int_B\int_B \E[\phi(z)\phi(w)]\, dA(z) dA(w),
\end{equation}
where the exchange of integral and expectation is justified by noting that, by~\eqref{eq: phi}, $\phi(z)\phi(w)$ is a polynomial in the coordinates of the Gaussian vector $(\xi(z),\xi'(z),\xi''(z),\xi(w),\xi'(w),\xi''(w))$.

The following proposition is an application of the so-called Feynman-diagram method and enables us to compute $\E[\phi(z)\phi(w)]$ term-by-term according to~\eqref{eq: phi}. 

\begin{prop}\label{prop_diag}
Suppose that $(\al,\bl,\cl,\dl)$ is a complex Gaussian random vector, where each coordinate has a standard complex Gaussian distribution, and $\E(\al\overline{\bl})=0$, $\E(\cl\overline{\dl})=0$.
Then:
\begin{enumerate}
\item
$
\E[L_1(|\al|^2) L_1(|\bl|^2) L_1(|\cl|^2) L_1(|\dl|^2)] = 
\Big|\E(\al \overline{\cl})\E(\beta \overline{\dl}) + \E(\al \overline{\dl}) 
\E(\bl \overline{\cl} )\Big|^2 
$,
\item 
$ \E [L_1(|\al|^2) L_1(|\bl|^2) L_2(|\cl|^2) ] = 2 \left| \E(\al \overline{\cl}) \E(\bl \overline{\cl}) \right|^2 $,
\item 
$ \E [L_2(|\al|^2) L_2(|\cl|^2 ) ] =  \left| \E(\al \overline{\cl})\right|^4$.
\end{enumerate}
\end{prop}
The proof of this proposition is provided in Section~\ref{sec_diag} below. 
Here we apply it to 
$\al, \bl \in \{\xi(z),\xi'(z),\xi''(z)\}$ and $\cl, \dl \in \{\xi(w),\xi'(w),\xi''(w)\}$ such that $\al\neq \bl$, $\cl\neq \dl$.
This allows us to continue from~\eqref{eq_stop}, using the form of $\phi$ given in~\eqref{eq: phi}. The expectation of each term is then expressed (via Proposition~\ref{prop_diag}) in terms of covariances between $(\xi(z),\xi'(z),\xi''(z))$ and $(\xi(w),\xi'(w),\xi''(w))$.
These covariance computations are explicitly carried out using the relations in Lemma~\ref{lem_precov}, and recalling the definitions in \eqref{eq_ts}, \eqref{eq_D}, \eqref{eq_H}, and~\eqref{eq_xi}. We include the details in Appendix~\ref{appendix}. Using~\eqref{eq_cs} and combining all the resulting expressions, we obtain that
\[
\E[\phi(z)\phi(w)]= g(|z-w|^2), \text{ where }
\]
\[
g(s)=
\frac{2}{729} e^{-2s} (2091 - 22110 s + 62628 s^2 - 77836 s^3 + 48325 s^4 - 
   15040 s^5 + 2156 s^6 - 116 s^7 + 2 s^8),
  \]
see the accompanying notebook for symbolic calculations \cite{jupy_gwhf_nonhyp}.
This leads to the explicit bound
\begin{equation}\label{eq: almost LB}
\var N(B) \geq 
\var Q_{2,2}(N(B)) = \frac{1}{\pi^2}\int_B\int_B g(|z-w|^2) dA(z) dA(w)= \frac{1}{\pi^2}\int_{B} \big(g(|\cdot|^2) * 1_{B}\big)(z)\,dA(z), 
\end{equation}
which we now inspect. One can easily verify that
\[
\int_\C g\big(|z|^2\big) \,dA(z) = \pi \int_0^\infty g(s) ds= \frac{7}{81}\pi,
\]
see for example the accompanying notebook \cite{jupy_gwhf_nonhyp}.
For us the important fact is that this number is not equal to zero. We also note that \[ 
\int_\C |z| |g(|z|^2)| dA(z) <\infty.
\]
We shall need the following standard 
estimate, which can be found 
in \cite[Lemma 8.3]{hkr22}.
\begin{lemma} \label{lem:famous}
Let $B_R=B_R(0)$. Let $h\colon \C \to \mathbb{R} $ be integrable with $\int_{\mathbb{C}} h(z)dA(z) \neq 0$ and 
$$
C_h:=\int_\C |z| |h(z)| dA(z) < \infty. 
$$
There exists a universal constant $C$ such that 
$$
\bigg| \int_{B_R} h * 1_{B_R} - \bigg( \int_\C h(z)dA(z) \bigg) |B_R| \bigg| \leq CC_hR, \quad R>0.
$$
\end{lemma}
Applying Lemma~\ref{lem:famous}, we may continue from~\eqref{eq: almost LB} to conclude that
\begin{align}\label{eq: end}
\var N(B) &\ge
\frac{1}{\pi^2}\int_{B_R} \big(g(|\cdot|^2) * 1_{B_R}\big)(z)\,dA(z) \notag
\\
&= \frac{1}{\pi^2}\int_{\C} g\big(|z|^2\big) \,dA(z) \cdot \pi R^2 + O(R) = \frac{7}{81} R^2 + O(R) \geq c R^2
\end{align}
for large enough $R$ and a strictly positive constant $c$. 
To conclude the proof of the lower bound, we need to observe that $\var N(B_R) >0$ for any $R>0$. This is the case, in general, for any Euclidean stationary point process on the plane with positive and finite first intensity (see, e.g., Lemma \ref{lemma_R} below). In our case $\{F=0\}$ is stationary because $F$ has zero mean (and is invariant under twisted shifts), while the first intensity of the zero set is finite and $\geq \tfrac{1}{\pi}$, as shown in \cite[Theorem 1.6]{hkr22}.

\subsection{Upper bound}

We now look into an upper bound for the variance of the number statistic.

Let $z \in \C \setminus\{0\}$, consider the Gaussian vector
\begin{align}\label{eq_X}
X(z) = (\mathcal{D}_1 F(z), \mathcal{D}_2 F(z),
\mathcal{D}_1 F(0), \mathcal{D}_2 F(0), F(z), F(0)),
\end{align}
and write its covariance as
\begin{align}\label{eq_split}
\Gamma(z) = \begin{bmatrix}
A & B \\
B^* & C
\end{bmatrix}
\end{align}
with $A \in \C^{4\times 4}$. Let $Z$ be a standard complex Gaussian vector in $\C^{4}$ (so that $\cov(Z)=I \in \C^{4 \times 4}$). The two point function \eqref{eq_sp2} can be expressed by Gaussian regression as
\begin{align*}
\tau(z,0,0) = 
(1-|H(z)|^2)^{-1}\cdot
\mathbb{E} \big[ f((A-BC^{-1}B^*)^{1/2} Z)
\big], \quad z\not= 0,
\end{align*}
where $f(u_1, \dots, u_4)=\tfrac{1}{\pi^2}\big||u_1|^2-|u_2|^2 \big| \cdot \big||u_3|^2-|u_4|^2 \big|$. We note that
\begin{align}\label{eq_xss}
|f(u) - f(u')| \lesssim |u-u'| \cdot (1+|u|^s+|u'|^s),
\qquad  u,u' \in \C^4,
\end{align}
and $|f(u)| \lesssim (1+|u|^s)$ for some $s>0$.

Let $\tilde{\Gamma}$ be obtained from $\Gamma$ by replacing with $0$ every entry that corresponds to correlations between a function of $z$ and a function evaluated at $0$. Thus $\tilde{\Gamma}$ is the covariance of a Gaussian vector defined analogously to \eqref{eq_X} but replacing $(\mathcal{D}_1 F(z), \mathcal{D}_2 F(z), F(z))$ and $(\mathcal{D}_1 F(0), \mathcal{D}_2 F(0), F(0))$ with two independent copies of those vectors. As a consequence, the one point function \eqref{eq_sp1} satisfies
\begin{align*}
\rho(z,0) \rho(0,0) = 
\mathbb{E} \big[ f((\tilde{A}-\tilde{B} \tilde{C}^{-1}\tilde{B}^*)^{1/2} Z)
\big],
\end{align*}
where $\tilde{\Gamma}$ is split into blocks analogous to \eqref{eq_split}. (Actually, $\tilde C=I$.)

By the fast decay of $H$ and its derivatives,
\begin{align*}
\|\Gamma - \tilde{\Gamma}\| \lesssim e^{-2c|z|^2},
\qquad \|\Gamma\|, \|\tilde{\Gamma}\| \lesssim 1,
\end{align*}
for some constant $c>0$. It follows that there exists $L>0$ such that for $|z| > L$, $\|C^{-1}\| \lesssim 1$, and
\begin{align*}
(1-|H(z)|^2)^{-1} = 1 + O(e^{-c|z|^2}).
\end{align*}
Taking into account that the square-root operation is continuous on the set of positive matrices, this implies that, for $|z|>L$,
\begin{align*}
\|(A-BC^{-1}B^*)^{1/2} - (\tilde{A}-\tilde{B} \tilde{C}^{-1}\tilde{B}^*)^{1/2}\| \lesssim e^{-c|z|^2},
\end{align*}
and, by \eqref{eq_xss},
\begin{align*}
&\big|\tau(z,0,0) - \rho(z)\rho(0) \big|
\\
&\,\leq \mathbb{E}
\Big[\big| f((A-BC^{-1}B^*)^{1/2} Z) - f((\tilde{A}-\tilde{B} \tilde{C}^{-1}\tilde{B}^*)^{1/2} Z) \big|
+
O(e^{-c|z|^2}) \big| f((A-BC^{-1}B^*)^{1/2} Z)\big| \Big]
\\
&\, \lesssim
\mathbb{E}
\Big[\big(\big| (A-BC^{-1}B^*)^{1/2} Z - (\tilde{A}-\tilde{B} \tilde{C}^{-1}\tilde{B}^*)^{1/2} Z) \big| + O(e^{-c|z|^2})\big)
\,\cdot \,
\\*
&\quad\qquad\qquad\qquad
\Big(1+\big| (A-BC^{-1}B^*)^{1/2} Z \big|^s+\big|(\tilde{A}-\tilde{B} \tilde{C}^{-1}\tilde{B}^*)^{1/2} Z) \big|^s \Big)
\Big]
\\
&\, \lesssim 
e^{-c|z|^2} \mathbb{E} \big[ 1+|Z|^{s+1} \big] \lesssim e^{-c|z|^2}.
\end{align*}
On the other hand, for $|z| \leq L$ we can invoke Proposition \ref{prop_r} and conclude that
\begin{align}\label{eq_b2p}
\big|\tau(z,0,0) - \rho(z)\rho(0) \big|
\lesssim e^{-c|z|^2}, \qquad z \in \C.
\end{align}
Since $F$ has zero mean, the point process $\{F=0\}$ is stationary and
\begin{align*}
\rho(z,0)=\rho(0,0), \qquad \tau(z,w,0)=\tau(z-w,0,0), \qquad z,w \in \mathbb{C},
\end{align*}
(this is special of the zero statistic; it does not apply to the more general $u$-crossings considered in Section \ref{sec_reg}.)
We next write the variance of the number statistic with help of Kac-Rice formulae, whose application was justified in the proof of Theorem~\ref{th_coro_r}, cf.\ \eqref{eq_coro3} and \eqref{eq_coro4},
\begin{align*}
&\mathbb{E}\big[N(B)\big] = \int_{B} \rho(z,0)\,dA(z) =  \rho(0,0) \cdot |B|,
\\
&\left(\mathbb{E}\big[N(B)\big]\right)^2 = \int_{B\times B} \rho(z,0)\rho(w,0) \, dA(z)dA(w)
=  \rho(0,0)^2 \cdot |B|^2 ,
\\
&\mathbb{E}\big[N(B)^2-N(B)\big]
=
\int_{B\times B} \tau(z,w,0) \, dA(z)dA(w)
= \int_{B\times B} \tau(z-w,0,0) \, dA(z)dA(w).
\end{align*}
Hence,
\begin{align}\label{eq_nh1}
\var[N(B)] =  \rho(0,0) \cdot |B| + 
\int_{B\times B} \left(\tau(z-w,0,0) - \rho(0,0)^2 \right)\, dA(z)dA(w).
\end{align}
Using \eqref{eq_b2p} and Lemma \ref{lem:famous} we estimate
\begin{align*}
\Big|\int_{B\times B} \left(\tau(z-w,0,0) - \rho(0,0)^2 \right)\, dA(z,w) \Big| &\lesssim
\int_{B\times B} e^{-c|z-w|^2} \, dA(z)dA(w)
\\
&= \alpha \cdot |B| + O(R),
\end{align*}
where $\alpha := \int_{\C} e^{-c|z|^2} \,dA(z)$.
Combining this with \eqref{eq_nh1} gives
\begin{align*}
\var[N(B)] \leq  (\rho(0,0)+\alpha) \cdot \pi R^2 + O(R) \leq C R^2,
\end{align*}
for some constant $C>0$.
This completes the proof of Theorem \ref{th2}. \qed

\subsection{Proof of Proposition~\ref{prop_diag}}\label{sec_diag}
The proof relies on the well-known diagram method (see Janson~\cite{janson1997gaussian}), which has been previously used in somewhat similar situations, e.g., in ~\cite{buckley2017fluctuations}, \cite{sodin2004random}, 
\cite{buckley2022gaussian}. 
We briefly recall this method.

Let $(\al_1,\ldots,\al_k)$ be a complex Gaussian random vector, normalized so that $\al_r\sim\mathcal{N}_\C(0,1)$ for all $1 \le r\le k$, and let $i_1,\dots i_k,j_1,\dots,j_k\in \mathbb N$. A \emph{complete Feynman diagram} is a graph with $\sum_{r=1}^k (i_r+j_r) $ vertices labelled by $\{\al_r, \overline{\al_r}\}_{r=1}^k$, such that:
\begin{itemize}
\item there are exactly $i_r$ vertices labelled by $\al_r$, and $j_r$ vertices labelled by $\overline{\al_r}$.
\item each vertex has degree $1$.
\item no edge joins a vertex with another vertex of the same label, or of its conjugate label.
\end{itemize}
The value of a diagram $\Gamma$ is defined as
\[
v(\Gamma) = \prod_{(a,b)\in E(\Gamma)} \E (a \,b),
\]
where $E(\Gamma)$ is the set of edges of $\Gamma$. 
Then, by~\cite[Theorem 3.12]{janson1997gaussian}, we have\footnote{
The relation to complex Hermite polynomials is given 
in~\cite[Example 3.31]{janson1997gaussian}.}
\[
\E \left[ H_{i_1,j_1}(\al_1, \overline{\al_1} ) \cdot \ldots\cdot H_{i_k,j_k}(\al_k, \overline{\al_k})  \right]
 = \sum_\Gamma v(\Gamma),
\]
where the sum is over all complete Feynman diagrams $\Gamma$, and $H_{i,j}$ are the complex Hermite polynomials introduced in~\eqref{eq_hermite}.
In the case where $i_r=j_r$ for all $r$, we may use Laguerre polynomials via~\eqref{eq_lh}, to obtain:
\begin{equation}\label{eq_lag}
\E \left[ L_{i_1}(|\al_1|^2) \cdots L_{i_k}(|\al_k|^2) \right]
 = \frac{(-1)^{\sum_r i_r}}{\prod_r i_r!}\sum_\Gamma v(\Gamma).
\end{equation}
Proposition~\ref{prop_diag} can now be proved by using formula~\eqref{eq_lag} with a suitable set of diagrams in each part. 
In the first part, we need to compute the product of four elements:
\[
\E[ L_1(|\al|)^2 L_1(|\bl|^2) L_1(|\cl|^2) L_1(|\dl|^2)].
\]
We illustrate the associated diagram counting in Fig.~\ref{fig:diagram}.
\begin{figure}
    \centering
    \includegraphics{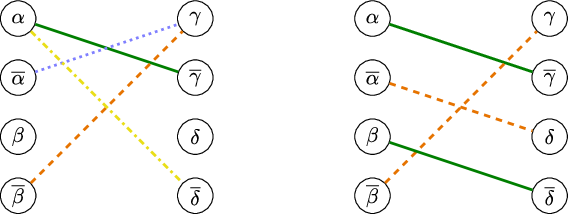}
    \caption{Diagram counting. 
    \textit{Left:} One must choose an edge connecting $\al$ either to $\overline\cl$ or $\overline{\dl}$ and another edge connecting $\cl$ to either $\overline{\al}$ or $\overline{\bl}$. 
    \textit{Right:} The diagram resulting from connecting $\al$ to $\overline{\cl}$ and $\cl$ to $\overline{\bl}$.}
    \label{fig:diagram}
\end{figure}
In the corresponding diagrams each label among $\{\al,\overline{\al},\bl, \overline{\bl}, \cl,\overline{\cl},\dl, \overline{\dl} \}$ appears exactly once. 
We note that all relevant diagrams are bi-partite graphs, with edges between  $V_1=\{\al,\overline{\al},\bl, \overline{\bl}\}$
and $V_2=\{\cl,\overline{\cl},\dl, \overline{\dl}\}$. Indeed, there are no edges within $V_1$ since, by definition, there are no edges between a label and its conjugate, and by independence, there are no edges between $\{\al,\overline{\al}\}$ and $\{\bl,\overline{\bl}\}$.
For the same reasons, there are no edges within $V_2$. We further note that $\E(\al \cl) =0$ for any circularly symmetric complex Gaussian random vector $(\al,\cl)$. Since the degree of each vertex must be $1$, 
there must be an edge joining $\al$ with $\overline{\cl}$ or with $\overline{\dl}$; the remaining vertex must be joined to $\bl$.
Similarly, there must be an edge joining $\cl$ with
$\overline{\al}$ or with $\overline{\bl}$; the remaining vertex will be joined with $\dl$. These choices are independent. Therefore, the sum of values of the resulting four diagrams is
\[
 \left(\E(\al \overline{\cl})\E(\bl \overline{\dl}) + \E(\al \overline{\dl}) 
\E(\bl \overline{\cl} ) \right)
\left(
\E(\cl \overline{\al}) \E(\dl \overline{\bl}) + 
\E(\cl \overline{\bl}) \E(\dl \overline{\al})
\right)
 = \Big|\E(\al \overline{\cl})\E(\beta \overline{\dl}) + \E(\al \overline{\dl}) 
\E(\bl \overline{\cl} )\Big|^2.
\]
Plugging this into \eqref{eq_lag}, we obtain the first item.

For the second part, we need to compute
\[
\E [L_1(|\al|^2) L_1(|\bl|^2) L_2(|\cl|^2) ].
\]
The corresponding diagrams are bi-partite graphs, with edges between $V_1 = \{\al,\overline{\al}, \bl, \overline{\bl}\}$ and 
$V_2 = \{\cl, \overline{\cl}, \cl, \overline{\cl}\}$ (that is, $V_2$ is a set of four vertices carrying two labels of $\cl$ and two labels of $\overline{\cl}$). By the same arguments as before, 
$\al$ must be joined to one of two copies of $\overline{\cl}$ (and the remaining copy must be joined to $\bl$), and the first copy of $\cl$ must be joined to either $\overline{\al}$ or $\overline{\bl}$ (and the remaining vertex must be joined to the second copy of $\cl$). These four diagrams have the value
\[
\E(\al\overline{\cl})\E(\bl\overline{\cl})\E (\overline{\al}\cl)
\E(\overline{\bl}\cl) = \left| \E(\al \overline{\cl}) \E(\bl \overline{\cl}) \right|^2.
\]
Considering the normalizing factor in~\eqref{eq_lag}, we establish the second item.

For the computation of $\E [ L_2(|\al|^2) L_2(|\cl|^2)]$ in the 
third part, the corresponding diagrams are bi-partite graphs with edges between $V_1 = \{ \al,\overline{\al},\al,\overline{\al}\}$ and $V_2 = \{\cl, \overline{\cl}, \cl, \overline{\cl}\}$. Again there are four diagrams, resulting from a choice of edge joining the first copy of $\al$ to one of the two copies of $\overline{\cl}$, and an edge joining the first copy of $\cl$ to one of the two copies of $\overline{\al}$. The value of each such diagram is $|\E(\al\overline{\cl})|^4$. The normalizing factor in~\eqref{eq_lag} is $\tfrac {1}{2! 2!}$, which yields the result.

\section{Applications}\label{sec_app}
\subsection{Gaussian entire functions}
\label{sec_gef_p}

We start by identifying Gaussian entire functions and their iterated covariant derivatives as GWHF.
\begin{lemma}\label{lemma_id}
Let $G=G_0+G_1$, where $G_0$ is the translation-invariant GEF \eqref{eq_g0} and $G_1\colon \C \to \C$ is entire with 
\begin{align}\label{eq_iee}
\sup_{z\in\C} |G_1(z)| e^{-\tfrac{1}{2}\abs{z}^2} < \infty.
\end{align}
Consider the iterated covariant derivative 
$\big(\bar \partial^*\big)^{n} G =
\big(\bar z -\partial)^{n} G$ with $n \in \mathbb{N}_0$ and set
\begin{align*}
F^{(n)}(z) = \frac{e^{-\tfrac{1}{2}\abs{z}^2}}{\sqrt{n!}}\big(\bar \partial^*\big)^{n} G(z).
\end{align*}
Then $F^{(n)}$ is a GWHF with twisted kernel
\begin{align}\label{eq_H_pure}
	H^{(n)}(z)= L_{n}(|z|^2) \cdot e^{-\tfrac{1}{2}\abs{z}^2},
	\end{align}
	where $L_{n}$ is the Laguerre polynomial of degree $n$ \eqref{eq_lagn}. In addition, conditions \eqref{A1}, \eqref{A2}, \eqref{A3}, \eqref{A4}, \eqref{A5}, and \eqref{A6} are satisfied.
\end{lemma}
\begin{proof}
When $G_1 \equiv 0$, the lemma is contained in \cite[Lemma 6.3]{hkr22}---and was proved by resorting to \cite{gh13,is16}---so we only check 
\eqref{A1}, which concerns the mean function
\begin{align*}
F_1^{(n)}(z) = \frac{e^{-\tfrac{1}{2}\abs{z}^2}}{\sqrt{n!}}\big(\bar \partial^*\big)^{n} G_1(z).
\end{align*}
The key observation is that
for an arbitrary real-smooth function $L\colon \C\to\C$,
the Wirtinger operators and the twisted derivatives are related by
\begin{align*}
\mathcal{D}_1 \big[ e^{-\tfrac{1}{2}\abs{z}^2} L(z) \big]= -e^{-\tfrac{1}{2}\abs{z}^2}\, \bar{\partial}^* L (z),
\\
\mathcal{D}_2 \big[ e^{-\tfrac{1}{2}\abs{z}^2} L (z)\big] = e^{-\tfrac{1}{2}\abs{z}^2} \,\bar{\partial} L (z).
\end{align*}
In particular,
\begin{align*}
F^{(n)}_1 = \frac{(-1)^n}{\sqrt{n!}}\big(\mathcal{D}_1\big)^n F^{(0)}_1.
\end{align*}
We also observe that each twisted shift of $F^{(0)}_1$ is a weighted analytic function:
\begin{align*}
\mathcal{T}_{\xi} F^{(0)}_1 (z)
= F^{(0)}_1 (z-\xi) e^{i \Im(z \bar{\xi})} = G_1(z-\xi) e^{-|z-\xi|^2/2 }e^{i \Im(z \bar{\xi})}
=  W_\xi G_1(z) e^{-\tfrac{1}{2}\abs{z}^2},
\end{align*}
where $W_\xi G_1(z) := G_1(z-\xi) e^{z \bar{\xi} - |\xi|^2/2}$ is the so-called Bargmann shift of $G_1$ \cite{zhu}.
Hence, by a Cauchy estimate, for each $k \geq 1$ we have 
\begin{align} \label{eq_fact1}
\big|\mathcal{D}_1^k \big[\mathcal{T}_{\xi} F^{(0)}_1 \big](0)\big| &= e^{-|0|^2/2} \big|(\bar{\partial}^*)^k \,W_\xi G_1 (0)\big|
= \big| \partial^k W_\xi G_1(0)\big| \nonumber \\
&\leq C_{k} \max_{|z|\leq 1 } | W_\xi G_1(z)| \leq C'_{k}  \max_{|z|\leq 1 } | \mathcal{T}_{\xi} F^{(0)}_1 (z)| \leq C'_{k} \|F^{(0)}_1\|_\infty.
\end{align}
Similarly 
\begin{align} \label{eq_fact2}
\big|\mathcal{D}_2 \mathcal{D}_1^k \big[\mathcal{T}_{\xi} F^{(0)}_1 \big](0)\big| &= e^{-|0|^2/2} \big|\bar \partial (\bar{\partial}^*)^k \,W_\xi G_1 (0)\big|
= k \big| \partial^{k-1} W_\xi G_1(0)\big| \leq C''_{k} \|F^{(0)}_1\|_\infty,
\end{align}
where the second equality follows from the following identity, which holds for an analytic function $f$:
\[\bar \partial (\bar \partial^*)^kf(z)=k (\bar \partial^*)^{k-1}f(z),
\]  
see, e.g., \cite[Eq. (3.4)]{haimi2019central}.
For $j=1,2$, we can now use \eqref{eq_fact1}, \eqref{eq_fact2}, the commutation property \eqref{eq_commute} and \eqref{eq_iee} to obtain
\begin{align*}
|\mathcal{D}_j F^{(n)}_1 (z)| &= \frac{1}{\sqrt{n!}}
|\mathcal{D}_j (\mathcal{D}_1)^n F^{(0)}_1 (z)|=
\frac{1}{\sqrt{n!}}|\mathcal{T}_{-z} \big[\mathcal{D}_j (\mathcal{D}_1)^n F^{(0)}_1\big] (0)|
\\
&=\frac{1}{\sqrt{n!}} \, | \mathcal{D}_j (\mathcal{D}_1)^n  \mathcal{T}_{-z} F^{(0)}_1(0)|
\leq C_n \|F^{(0)}_1\|_\infty < \infty,
\end{align*}
which shows that \eqref{A1} indeed holds.
\end{proof}

We can now prove the hyperuniformity of the zeros of a GEF with non-trivial mean.
\begin{proof}[Proof of Theorem \ref{th_5}]
By Lemma \ref{lemma_id}, we identify $G=G_0+G_1$ with a GWHF with twisted kernel
\begin{align*}
H(z)=e^{-\tfrac{1}{2}\abs{z}^2}
\end{align*}
by setting
\begin{align*}
F_j(z) = e^{-\tfrac{1}{2}\abs{z}^2} G_j(z), \qquad j=0,1.
\end{align*}
By Lemma \ref{lemma_id}, $F:=F_0+F_1$ satisfies the assumptions of Theorem \ref{th1}, so it suffices to observe that all its charges are equal to 1. This follows from the analyticity of $G$. Indeed, if $F(z)=0$ then
\begin{align*}
\jac F(z) = |\partial F(z)|^2 - |\bar\partial F(z)|^2 \geq 0 
\end{align*}
because
\begin{align*}
\bar\partial F(z) = \Big(\bar\partial G(z) -\frac{z}{2} G(z)\Big) e^{-\tfrac{1}{2}\abs{z}^2} = 0,
\end{align*}
while $\jac F(z) \not= 0$ by Lemma \ref{lem_pi}. Therefore $\charge_z = \sgn \jac F(z) =1$, as claimed.
\end{proof}

We now look into the critical points of translation-invariant GEF.
\begin{proof}[Proof of Theorems \ref{th_3} and \ref{th_4}]

Let $G=G_0$ be the zero-mean Gaussian entire function given by \eqref{eq_g0}. By Lemma \ref{lemma_id},
\[F(z)= e^{-\tfrac{1}{2}\abs{z}^2}\, \bar{\partial}^* G(z)\]
is a GWHF with twisted kernel $H(z) = (1-|z|^2)e^{-\abs{z}^2/2}$. Hence Theorem \ref{th_3} follows immediately from Theorem \ref{th2}.

Second, because multiplication by a smooth non-vanishing factor does not alter indices (see e.g., \cite[Lemma 6.5]{hkr22}) the charge of $F$ at a zero is exactly the index of $\bar{\partial}^* G$. Thus, Theorem \ref{th1} (or \cite[Theorem 1.12]{hkr22}) implies that
\begin{align}\label{eq_ia}
\var[\mathcal{N}_R^+ - \mathcal{N}_R^-] \leq C R, \qquad R>1,
\end{align}
for a constant $C>0$.
On the other hand
$\mathcal{N}_R^+ + \mathcal{N}_R^-$ is the total number of zeros of $\bar{\partial}^* G$ in $B_R(0)$ and therefore Theorem \ref{th2} gives 
\begin{align}\label{eq_ib}
k R^2 \leq 
\var[\mathcal{N}_R^+ + \mathcal{N}_R^-]\leq K R^2, \qquad R>1,
\end{align}
for adequate constants $k,K>0$.
As a consequence, we have 
\begin{align*}
\var[\mathcal{N}_R^+] \lesssim \var[\mathcal{N}_R^+ + \mathcal{N}_R^-] + \var[\mathcal{N}_R^+ - \mathcal{N}_R^-] \lesssim (K+C) R^2,
\end{align*}
while
\begin{align*}
k R^2 &\leq \var[\mathcal{N}_R^+ + \mathcal{N}_R^-] 
= \var[2\mathcal{N}_R^+ - (\mathcal{N}_R^+ - \mathcal{N}_R^-)]
\\
& \leq 8 \var[\mathcal{N}_R^+] + 2 \var[\mathcal{N}_R^+ - \mathcal{N}_R^-] \leq 8 \var[\mathcal{N}_R^+] + 2 C R.
\end{align*}
A similar argument applies to $\mathcal{N}_R^-$ and proves \eqref{eq_ic} for large $R$. To conclude the proof we observe that
the each of processes defining $\mathcal{N}_R^\pm$ is stationary and has finite first intensities --- c.f. Remark \ref{rem_shift_charge} or \cite[Proof of Lemma 3.3]{hkr22} --- so that Lemma \ref{lemma_R} below implies that $\var[\mathcal{N}_R^\pm]$ cannot vanish for any $R>0$. Hence, \eqref{eq_ic} also holds in the range $R \geq 1$.
\end{proof}

\subsection{Weighted entire functions}\label{sec_wgef}
Theorems \ref{th_3} and \ref{th_4} can also be formulated in terms of the \emph{weighted magnitude}
\begin{align}\label{eq_ampl}
A(z) = e^{-\tfrac{1}{2}\abs{z}^2} |G_0(z)|
\end{align}
associated with the translation-invariant Gaussian entire function $G_0$ given by \eqref{eq_g0}. Near points where the amplitude $A$ does not vanish, it is smooth and its gradient is related to the covariant derivative of $G_0$ by
\begin{align*}
|\nabla A(z)| = e^{-\tfrac{1}{2}\abs{z}^2} |\bar{\partial}^* G_0(z)| = |F(z)|,
\end{align*}
where $F(z) = e^{-\tfrac{1}{2}\abs{z}^2} \bar{\partial}^* G_0(z)$.
Hence each critical point $z_0$ of $A$ (in the usual real sense) that is not a zero of $A$ is a zero of $\bar{\partial}^* G_0(z)$. Conversely, a zero of $\bar{\partial}^* G_0$ is almost surely not a zero of $A$, as this would correspond to a degenerate zero of the GWHF $e^{-|z|^2/2} G_0(z)$, contradicting Lemma \ref{lem_pi}.

Thus the critical points of $A$ are either zeros of $A$ or critical points of $G_0$, as discussed Section \ref{sec_gef}. Let us further inspect a critical point $z_0$ of $A$ that is not a zero.
Near $z_0$ we can write
$G(z)=L(z)^2$ with $L$ analytic and compute
\begin{align*}
2 \partial A = A^{(1,0)} - i A^{(0,1)} = -{\frac{\overline{L}}{L}} \cdot F.
\end{align*}
The factor ${{\overline{L}}/{L}}$ is smooth (in the real sense) and non-zero near $z_0$ and therefore does not affect the corresponding charge (see e.g., \cite[Lemma 6.5]{hkr22}). Thus, the charge of $F$ at $z_0$ is
\begin{align*}
\charge_z = \sgn \Big[ \big[A^{(1,1)}\big]^2 - A^{(2,0)} A^{(0,2)} \Big],
\end{align*}
that is, 
the opposite of the sign of the determinant of the Hessian matrix of $A$ at $z_0$. As a consequence, $\charge_z=1$ if $z_0$ is a saddle point of $A$, while $\charge_z=-1$ if $A$ has a local maximum or local minimum at $z_0$. In addition, by an argument based on superharmonicity one can see that local minima of $A$ are not possible (as these correspond to zeros of $A$) \cite[Section 8.2.2]{gafbook}) \cite[Lemma 3.1]{efkr24}.

Hence,
\begin{align}\label{eq_As}
\mathcal{N}_R^+ = \# \mbox{saddle points of $A$ in }B_R(0),
\\\label{eq_Am}
\mathcal{N}_R^- = \# \mbox{local maxima of $A$ in } B_R(0).
\end{align}
The second order statistics of local extrema and saddle points of $A$ thus satisfy \eqref{eq_ia}, \eqref{eq_ib}, and \eqref{eq_ic}.

\subsection{Time-frequency analysis}\label{sec_app_tf}

Given a non-zero Schwartz function $g \in \mathcal{S}(\mathbb{R})$, the short-time Fourier transform (STFT) of a distribution $f\in\mathcal{S}'(\mathbb{R})$ is defined by \eqref{eq_stft}, where the integral is interpreted in the distributional sense (action on a test function). The \emph{modulation space} $M^\infty(\mathbb{R})$ consists of all tempered distributions with bounded STFT:
\begin{align*}
M^\infty(\mathbb{R}) = \{f \in \mathcal{S}'(\mathbb{R}): V_g f \in L^\infty(\mathbb{R}^2)\}.
\end{align*}
Historically, this space was first considered with respect to the Gaussian window function (in which case the STFT can be identified with the \emph{Bargmann transform} \cite{MR201959}). In fact, it is easy to show that different choices of (non-zero, Schwartz) window functions $g_1, g_2 \in \mathcal{S}(\mathbb{R}) \setminus \{0\}$
define the same space, and that moreover there exists a constant $C_{g_1,g_2}>0$ such that
\begin{align}\label{eq_normsminf}
\| V_{g_1} f \|_\infty \leq C_{g_1,g_2} \| V_{g_2} f \|_\infty, \qquad f \in M^\infty(\mathbb{R}),
\end{align}
see, e.g., \cite{benyimodulation} for the modern theory of modulation spaces.

The space $M^\infty(\mathbb{R})$ contains $L^2(\mathbb{R})$ and also all distributions commonly used in signal processing, such as the Dirac measure, sums of Dirac measures located along well-spread sets of points, and also their Fourier transforms \cite{benyimodulation}.

We are mainly interested in a so-called signal $f \in M^\infty(\mathbb{R})$ impacted by additive noise. Let $\mathcal{W}$ be standard complex white noise on $\mathbb{R}$, that is, $\mathcal{W} = \frac{1}{\sqrt{2}} \frac{d}{dt} \big(W_1 + i W_2\big)$, where $W_1$ and $W_2$ are independent copies of the Wiener process (Brownian motion with almost surely continuous paths), and the derivative is taken in the distributional sense. 
Then, almost every realization of $\mathcal{W}$ is a tempered distribution and we can consider $V_g \mathcal{W}$. For more details, see \cite[Section 6.1]{hkr22}, or \cite{bh} for a different approach.

The following Lemma identifies the STFT of a distribution impacted by additive complex white noise with a GWHF.

\begin{lemma}\label{lem_stft_gwhf}
Let $g \in \mathcal{S}(\mathbb{R})$ be normalized by $\|g\|_2=1$, $f \in M^\infty(\mathbb{R})$, and $\mathcal{W}$ standard complex white noise. Set
	\begin{align}\label{eq_f_stft}
	F(z) := e^{-i xy}  \cdot V_g \, (f+\mathcal{W}) \big(\bar{z}/\sqrt{\pi}\big), \qquad z=x+iy.
	\end{align}
	Then $F$ is a GWHF with twisted kernel
	\begin{align*}
	H(z) = e^{-i xy} \cdot
	V_g g \big(\bar{z}/\sqrt{\pi} \big), \qquad z=x+iy.
	\end{align*}
	In addition, conditions \eqref{A1}, \eqref{A2}, \eqref{A3}, \eqref{A4}, \eqref{A5}, and \eqref{A6} are satisfied. 
\end{lemma}
\begin{proof}
Since $g \in \mathcal{S}(\mathbb{R})$, it is easy to see that $H \in \mathcal{S}(\mathbb{R}^2)$ --- see, e.g., \cite[Proposition 1.42]{folland89} --- which proves \eqref{A6}. In the zero mean case, the rest of the lemma is contained in \cite[Lemma 6.1]{hkr22}, so we focus on the mean function
\begin{align*}
F_1(z) = F_{1,g}(z) := e^{-i xy}  \cdot V_g \, f \big(\bar{z}/\sqrt{\pi}\big), \qquad z=x+iy,
\end{align*}
where, for convenience, we stress the dependence on the window function $g$.
The assumption $f \in M^\infty(\mathbb{R})$ means that $F \in L^\infty(\mathbb{R}^2)$, so it remains to inspect the twisted derivatives.

Taking  the distributional interpretation of the formula defining the STFT \eqref{eq_stft} into account, a direct calculation gives
\begin{align*}
\mathcal{D}_1 F_{1,g}(z)
&= \sqrt{\pi} F_{1, \mathcal{P} g}(z) -\tfrac{1}{2\sqrt{\pi}} F_{1, g'} (z),
\\
\mathcal{D}_2 F_{1,g}(z)
&= -\sqrt{\pi} F_{1, \mathcal{P} g}(z) -\tfrac{1}{2\sqrt{\pi}} F_{1, g'} (z),
\end{align*}
where $\mathcal{P}g(t)=tg(t)$. Since $g'$ and $\mathcal{P}g$ are non-zero Schwartz functions, the norm equivalence \eqref{eq_normsminf} implies that, for $j=1,2$,
\begin{align*}
\|\mathcal{D}_j F_{1,g}\|_{\infty}
\lesssim \|F_{1, \mathcal{P}g}\|_{\infty}
+ \|F_{1, g'}\|_{\infty}
= \|V_{\mathcal{P}g} f \|_{\infty}
+ \|V_{g'} f \|_{\infty} \lesssim \|V_g f\|_\infty < \infty,
\end{align*}
because $f \in M^\infty(\mathbb{R})$.
\end{proof}
We can now justify the application of our results to time-frequency analysis.
\begin{proof}[Proof of Theorem \ref{th_sp}]
Let $g(t)=2^{1/4}e^{-\pi t^2}$, $t\in\mathbb{R}$. By Lemma \ref{lem_stft_gwhf},
\begin{align*}
F(z) := e^{-i xy} \,V_g (f+\mathcal{W})\big(\bar{z}/\sqrt{\pi}\big) =
2^{1/4} \,e^{-i xy}  \,V (f+\mathcal{W}) \big(\bar{z}/\sqrt{\pi}\big), \qquad z=x+iy
\end{align*}
is a GWHF and the assumptions of Theorem \ref{th1} are satisfied. In addition, the twisted kernel is $H(z)=e^{-|z|^2/2}$ and $F$ can be further related to a Gaussian entire function $G$ by
$F(z) = e^{-|z|^2/2} G(z)$. As in the proof of Theorem \ref{th_5}, all the charges of $F$ are $1$. The zeros of $S(f+\mathcal{W})=|V(f+\mathcal{W})|^2$ are related to those of $F$ by the map $z \mapsto \bar{z}/\sqrt{\pi}$, and Theorem \ref{th1} (or Theorem \ref{th_5}) implies \eqref{eq_vs}.

Secondly, in the zero mean case $f=0$, the (Gaussian) spectrogram of $\mathcal{W}$ is related to the weighted amplitude of $G=G_0$ by
$S\,\mathcal{W}(x/\sqrt{\pi},-\xi/\sqrt{\pi}) = \tfrac{1}{\sqrt{2}}\,(A(x+i\xi))^2$,
c.f. \eqref{eq_connection} and \eqref{eq_ampl}. Thus, the 
map $z=x+i\xi \mapsto \bar{z}/\sqrt{\pi}$ relates the set of local maxima of $S\,\mathcal{W}$ to that of $A$, denoted $\mathcal{N}_R^-$ in \eqref{eq_Am}, which in turn satisfies \eqref{eq_ic}. This gives \eqref{eq_vs_2}.
\end{proof}

\begin{proof}[Proof of Theorem \ref{th_sp_2}]
Let $g := \|h_1\|_2^{-1} h_1$. By Lemma \ref{lem_stft_gwhf},
\begin{align*}
F(z) := e^{-i xy}  \cdot V_{g} \, \mathcal{W} \big(\bar{z}/\sqrt{\pi}\big), \qquad z=x+iy,
\end{align*}
defines a zero-mean GWHF function with twisted kernel \[H(z)=e^{-ixy} V_{g}g(\bar{z}/\sqrt{\pi})= (1-|z|^2)e^{-\tfrac{1}{2}\abs{z}^2}, \qquad z=x+iy,\] as can be verified with a direct calculation (the last equality is a special case of the so-called \emph{Hermite-Laguerre connection} \cite[Theorem (1.104)]{folland89}; see also \cite[Section 6.4]{hkr22}). By Lemma \ref{lem_stft_gwhf}, the hypothesis of Theorem \ref{th2} is satisfied, and the desired conclusion follows.
\end{proof}

\begin{proof}[Proof of Theorem \ref{th_sp_3}]
Without loss of generality we may assume that $\|g\|_2=1$, and invoke Lemma \ref{lem_stft_gwhf}. 
If $z=a+ib$ is a zero of the GWHF \eqref{eq_f_stft}, then
\begin{align*}
\partial_x F(z)&= \frac{e^{-i ab}}{\sqrt{\pi}} \partial_x \big(V_g \, (f+\mathcal{W})\big) (a/\sqrt{\pi},-b/\sqrt{\pi}),
\\
\partial_y F(z)&= -\frac{e^{-i ab}}{\sqrt{\pi}} \partial_y\big(V_g \, (f+\mathcal{W})\big)(a/\sqrt{\pi},-b/\sqrt{\pi}),
\end{align*}
and, consequently, the corresponding charge is
\begin{align*}
-\sgn \Im [\partial_x F(z) \cdot \overline{\partial_y F(z)}] =
\newcharge_{\bar{z}/\sqrt{\pi}} (f+\mathcal{W}).
\end{align*}
Thus the desired conclusion follows from Theorem \ref{th1}.
\end{proof}

\section{Auxiliary results and postponed proofs}\label{sec_aux}
\subsection{Stationary number statistics cannot be deterministic}
We prove the following lemma, for which we could not find a citable reference.
\begin{lemma}\label{lemma_R}
Let $\mathcal{Z}$ be a stationary point process on $\mathbb{C}$ with positive and finite first intensity. Then
\[
\var \#\big(\mathcal{Z} \cap \overline{B}(0,R) \big) > 0, \qquad \mbox{for all } R>0.
\]
\end{lemma}
\begin{proof}
Suppose on the contrary that there exists $R>0$ with
$\var \big(\mathcal{Z} \cap \overline{B}(0,R) \big)=0$. Let $N := \mathbb{E} \big[\#\big(\mathcal{Z} \cap \overline{B}(0,R) \big)\big]$. Thus, for every $\zeta \in \mathbb{C}$, the event $\#\big[\mathcal{Z} \cap \overline{B}(\zeta,R) \big]= N$ has probability one. In particular, $N \in \mathbb{N} \cup \{0,\infty\}$, while by the assumption of the first intensity, $0<N<\infty$. Collecting exceptional events for all 
$\zeta \in \mathbb{Q} + i \mathbb{Q}$ we find an event $\Omega_0$ such that $\mathbb{P}(\Omega_0)=1$ and 
$\#\big[\mathcal{Z} \cap \overline{B}(\zeta,R) \big]= N$ for all $\zeta \in \mathbb{Q} + i \mathbb{Q}$. 

Fix a realization $\Lambda$ of $\mathcal{Z}$ within the event $\Omega_0$, so that
\begin{align}\label{eq_kla}
\#\big[\Lambda \cap \overline{B}(\zeta,R) \big]= N, \qquad \zeta \in \mathbb{Q} + i \mathbb{Q}.
\end{align}
Since $N \geq 1$, there exists at least one point $\lambda \in \Lambda$. Consider the set of exceptional points
\begin{align*}
\Gamma = \bigcup_{\lambda'\in\Lambda \setminus \{\lambda\}}
\left(\partial B(\lambda,R) \cap \partial B(\lambda',R)\right).
\end{align*}
By \eqref{eq_kla}, $\Lambda$ has finitely many points on any given compact set. Since 
$\partial B(\lambda,R) \cap \partial B(\lambda',R)$ has at most two points, it follows that $\Gamma$ is finite.  We can therefore choose a point $z \in \partial B (\lambda,R) \setminus \Gamma$. This choice of $z$ guarantees that 
\begin{align}\label{eq_kla2}
\partial B(z,R) \cap \Lambda = \{\lambda\}.
\end{align}

Since $\Lambda$ is locally finite, the set $\overline{B}(\zeta,R) \cap \big( \Lambda \setminus \{\lambda\} \big)$ remains unaltered for all centers $\zeta$ that are sufficiently close to $z$. Thus \eqref{eq_kla2} allows us to choose $\zeta_1, \zeta_2 \in \mathbb{Q} + i \mathbb{Q}$, two small perturbations of $z$, such that
\begin{align*}
\# \big[ \Lambda \cap \overline{B}(\zeta_1, R) \big] < \big[ \Lambda \cap \overline{B}(\zeta_2, R) \big],
\end{align*}
by simply including or excluding $\lambda$ from the observation disk. This contradicts \eqref{eq_kla}, and completes the proof.
\end{proof}
\subsection{Convolution estimates}\label{sec_ce}
\begin{lemma}\label{lemma_b}
There exists a constant $C>0$ such that
	\begin{align}\label{eq_ll1}
		\sup_{z \in \mathbb{C}} \int_{w:|w| = R} (1+|z-w|)^{-2} \, |dw| \leq C.
	\end{align}
\end{lemma}
\begin{proof}
By applying a rotation, we assume that $z=a \in \mathbb{R}$. We also assume that $R\geq 4$. Note first that
\begin{align}\label{eq_ll2}
\int_{\substack{w:|w| = R,\\|w-a| > \sqrt{R}}} (1+|w-a|)^{-2} \, |dw|
\leq 2\pi (1+\sqrt{R})^{-2} R \lesssim 1.
\end{align}
On the other hand, writing $w=Re^{i\theta}$ with $\theta \in (-\pi,\pi]$, we have $|w-a| \geq R |\sin(\theta)|$. Hence, if $|w-a| \leq \sqrt{R}$, then $|\sin(\theta)| \leq R^{-1/2} < 1/2$ and thus $|\theta| \leq c |\sin(\theta)|$ for an absolute constant $c>0$. Therefore $|w-a| \geq \tfrac{R}{c} |\theta|$, while $|\theta| \leq c R^{-1/2}$. With this information, we estimate
\begin{align}\label{eq_ll0}
	&\int_{\substack{w:|w| = R,\\|w-a| \leq \sqrt{R}}} (1+|w-a|)^{-2} \, |dw|
	\leq R \int_{-cR^{-1/2}}^{cR^{-1/2}} \big(1+|\tfrac{R}{c}\theta|\big)^{-2}\, d\theta
	\lesssim \int_{-\sqrt{R}}^{\sqrt{R}} (1+|t|)^{-2}\,dt  \lesssim 1.
\end{align}
Combining \eqref{eq_ll2} and \eqref{eq_ll0} we obtain \eqref{eq_ll1}.
\end{proof}

\begin{lemma}\label{lemma_L}
Let $0 \leq \delta < 1$. Then there exists a constant $C_\delta$ such that for all $R,L>0$:
\begin{align*}
	\sup_{z:|z|=R} 	\int_{w:|w|=R, |z-w| \leq L} |z-w|^{-\delta}\, |dw| \leq C_\delta L^{1-\delta}.
\end{align*}
\end{lemma}
\begin{proof}
By applying a rotation, we assume that $z=R$, and by rescaling, we assume that $R=1$. Assume first that $L \leq 1/2$, write $w=e^{i\theta}$ with $\theta \in (-\pi,\pi]$, and note that if $|w-1| \leq L$, then $|\sin(\theta)| \leq L < 1/2$, so $|\theta| \leq c |\sin(\theta)|$ for an absolute constant $c>0$. As a consequence,
$|w-1| \geq |\sin(\theta)| \geq |\theta|/c$, and
\begin{align*}
	\int_{\substack{w:|w|=1,\\|w-1| \leq L}} |w-1|^{-\delta}\, |dw| 
	\leq \int_{-{cL}}^{{cL}} \big(\tfrac{|\theta|}{c}\big)^{-\delta}\,d\theta \leq C_\delta L^{1-\delta}.
\end{align*}
Finally, if $L \geq 1/2$, the previous case gives
\begin{align*}
	\int_{\substack{w:|w|=1,\\|w-1| \leq 1/2}} |w-1|^{-\delta}\, |dw|\leq C_\delta \leq C'_\delta L^{1-\delta},
\end{align*}
while
\begin{align*}
	\int_{\substack{w:|w|=1,\\ |w-1| \geq 1/2}} |w-1|^{-\delta}\, |dw|
	\lesssim C_\delta \lesssim C'_\delta L^{1-\delta}.
\end{align*}
\end{proof}

\begin{lemma}\label{lemma_conv}
Let $0<p<2$. Then there exists a constant $C_p>0$ such that
\begin{align}\label{eq_l0}
\int_{\mathbb{C}} |z+w|^{-p} e^{-|z|^2} \, dA(z) \leq C_p (1+|w|)^{-p}, \qquad w \in \mathbb{C}.
\end{align}
\end{lemma}
\begin{proof}
First note that, by the Hardy-Littlewood rearrangement inequality,
\begin{align}\label{eq_l1}
\int_{\mathbb{C}} |z+w|^{-p} e^{-|z|^2} \, dA(z) \leq
\int_{\mathbb{C}} |z|^{-p} e^{-|z|^2} \, dA(z) < \infty,
\end{align}
while
\begin{align}\label{eq_l2}
\int_{|z+w|\geq|w|/2} |z+w|^{-p} e^{-|z|^2} \, dA(z) \leq 2^p
|w|^{-p} \int_{\mathbb{C}} e^{-|z|^2} \, dA(z) = C_p |w|^{-p}.
\end{align}
Second, if $|z+w| \leq |w|/2$, then $|w| \leq 2|z|$ and
\begin{align}\label{eq_l3}
	\int_{|z+w| \leq |w|/2} |z+w|^{-p} e^{-|z|^2} \, dA(z) \leq
	e^{-|w|^2/4} \int_{|u|\leq |w|/2} |u|^{-p} \, dA(u)
	\leq C'_p e^{-|w|^2/4} |w|^{2-p}.
\end{align}
Hence, \eqref{eq_l1} gives \eqref{eq_l0} for $|w| \leq 1$, while \eqref{eq_l2} and \eqref{eq_l3} cover the case $|w| >1$.
\end{proof}

\subsection{The planar chaos decomposition is independent of the choice of basis}\label{sec_chin}
Recall the notation of Section \ref{sec_plch}. As we now argue, the chaos decomposition \eqref{eq_ch} is independent of the choice of the orthonormal basis of $G$. While this is implicit in \cite{janson1997gaussian}, we offer the following short argument.

We further split chaotic subspaces as follows. Given a polynomial $p$ in the variables $\eta_1, \overline{\eta_1}$, $\ldots\,$, $\eta_k, \overline{\eta_k}$, where $\eta_1, \ldots, \eta_k \in G$,  we say that $p$ is of degree $\leq (M,N)$ if it is of total degree $\leq M$ in variables $\eta_1,\ldots, \eta_k$ and of total degree $\leq N$ in variables $\overline{\eta_1},\ldots, \overline{\eta_k}.$ We consider the polynomial spaces
\begin{align*}
\mathrm{Pol}_{M,N}(G)= \text{span} \{ p(\eta_{1}, \overline{\eta_1}, \ldots, \eta_{k}, \overline{\eta_k}): p \text{ polynomial of degree} \leq (M,N), \eta_1, \ldots, \eta_k \in G \}.
\end{align*}
These spaces split the chaotic spaces as follows.
\begin{prop}\label{prop_pol}
For a (circularly symmetric) Gaussian random function $F$ with probability space $(\Omega,\mathbb{P})$, the following decomposition holds:
\begin{align*}
C_{M,N}= \overline{\mathrm{Pol}_{M,N}(G)} \ominus \overline{\mathrm{Pol}_{M-1,N}(G)} \ominus \overline{\mathrm{Pol}_{M,N-1}(G)}, \qquad M,N \geq 0,
\end{align*}
where $\mathrm{Pol}_{M,N}(G)=0$ if $N<0$ or $M<0$.

As a consequence, the spaces $C_{M,N}$ are independent of the choice of orthonormal basis of the Gaussian space $G$.
\end{prop}
\begin{proof}
Let us consider, more generally, polynomial spaces associated with a (not necessarily closed) linear subspace of $G$. If $V$ is a linear subspace of $G$, we let
\begin{align*}
\mathrm{Pol}_{M,N}(V)= \text{span} \{ p(\eta_{1}, \overline{\eta_1}, \ldots, \eta_{k}, \overline{\eta_k}): p \text{ polynomial of degree } \leq (M,N), \eta_1, \ldots, \eta_k \in V \}.
\end{align*}
Let $V$ be the linear span (i.e., finite linear combinations) of the variables $\{\xi_j: j \geq 0\}$. Then
\begin{align*}
\overline{\mathrm{Pol}_{M,N}(G)}
= \overline{\mathrm{Pol}_{M,N}(V)},
\end{align*}
while
\begin{align*}
\mathrm{Pol}_{M,N}(V)= \mathrm{span} \bigg\{
\prod_{k=0}^\infty H_{\alpha_k, \beta_k}( \xi_k, \bar \xi_k):
|\alpha| \leq M, |\beta| \leq N \bigg\}.
\end{align*}
It follows from orthogonality of complex Hermite polynomials that the random variables
\begin{align*}
\bigg \{
\prod_{k=0}^\infty H_{\alpha_k, \beta_k}( \xi_k, \bar \xi_k):
|\alpha| =M, |\beta| = N \bigg\}
\end{align*}
are orthogonal to the spaces $\mathrm{Pol}_{M-1,N}(V)$ and $\mathrm{Pol}_{M,N-1}(V)$, and, therefore, also to 
$\mathrm{Pol}_{M-1,N}(G)$ and $\mathrm{Pol}_{M,N-1}(G)$. Consequently, 
\begin{align*}
 \overline{\mathrm{span}} \bigg\{
\prod_{k=0}^\infty H_{\alpha_k, \beta_k}( \xi_k, \bar \xi_k):
|\alpha| =M, |\beta| = N \bigg\}
= \overline{\mathrm{Pol}_{M,N}(G) }\ominus \overline{\mathrm{Pol}_{M-1,N}(G)}\ominus \overline{\mathrm{Pol}_{M,N-1}(G)},
\end{align*}
which proves the claim.
\end{proof}

\subsection{Proof of Lemma \ref{lem_lip}}\label{sec_plem_lip}
Let us denote $\Psi(z) := \big(F(z),\mathcal{D}_1 F(z), \mathcal{D}_2 F(z)\big)$. By 
Lemma \ref{lem_covariance} and \eqref{A6},
the vector $\Psi(z)$ has a suitably smooth covariance kernel and
\begin{align*}
\sup_{z\in \mathbb{C}} \|\cov[\Psi(z)]\| & \lesssim 1,
\\
\sup_{1 \leq k \leq L_n} \sup_{z \in Q_k} \mathbb{E} \big[| \Psi(z)- \Psi(z_k)|^2\big] & \lesssim \frac{1}{n}.
\end{align*}
By normality, this implies that for every $p \in [1,\infty)$,
\begin{align*}
C_p := \sup_{z\in \mathbb{C}} (\mathbb{E}[1+|\Psi(z)|^p])^{1/p} < \infty,
\end{align*}
while
\begin{align*}
\delta_p(n) :=
\sup_{1 \leq k \leq L_n} \sup_{z \in Q_k} \left(\mathbb{E} \left| \Psi(z)- \Psi(z_k)\right|^p \right)^{1/p} \to 0, \mbox{ as } n \to \infty.
\end{align*}

Let $p,p' \in (1,\infty)$ be conjugate H\"older exponents and let us estimate:
\begin{align*}
&\left(\mathbb{E} \bigg[\bigg| \int_B \phi(\Psi(z)) \, dA(z) - \sum_{k=1}^{L_n} \phi(\Psi(z_k))
\,\big|B \cap Q_k \big| \bigg|^2\bigg]\right)^{1/2}
\\
&\qquad\leq
\left(\mathbb{E}\bigg[\bigg(
\sum_{k=1}^{L_n} \int_{B \cap Q_k} \left|
\phi(\Psi(z))- \phi(\Psi(z_k))\right|\,dA(z)\bigg)^2\bigg]\right)^{1/2}
\\
&\qquad\leq C\left(
\mathbb{E}\bigg[\bigg(\sum_{k=1}^{L_n} \int_{B \cap Q_k} \left|\Psi(z)- \Psi(z_k)\right|
(1+|\Psi(z)|^s+|\Psi(z_k)|^s)
\,dA(z)\bigg)^2\bigg]\right)^{1/2}
\\
&\qquad\leq
{C} \sum_{k=1}^{L_n} \int_{B \cap Q_k} \left(\mathbb{E} \big[\left|\Psi(z)- \Psi(z_k)\right|^2
(1+|\Psi(z)|^s+|\Psi(z_k)|^s)^2\big]\right)^{1/2} \,dA(z)
\\
&\qquad\leq
{C} \sum_{k=1}^{L_n} \int_{B \cap Q_k} \big(\mathbb{E} \big[\left|\Psi(z)- \Psi(z_k)\right|^{2p}\big]\big)^{1/{2p}}
\big(\mathbb{E} \big[(1+|\Psi(z)|^s+|\Psi(z_k)|^s)^{2p'}\big]\big)^{1/{2p'}}\,dA(z)
\\
&\qquad\leq C \cdot |B| \cdot C_{2sp'}^{s} \cdot
\delta_{2p}(n) \longrightarrow 0,
\end{align*}
as $n \to \infty$.

\appendix

\section{Computations}\label{appendix}
We use the notation of Section \ref{sec_nonhyper} and present the computations needed to obtain an explicit expression for $\E [\phi(z) \phi(w)]$.
We recall that the chaos projection of $N(B)$ to $C_{2,2}$ is  
$\frac{1}{\pi}\int_B \phi(z) dA(z)$, where 
\begin{align*}
\phi(z)&= c_{1,0} L_1(|\xi(z)|^2)L_1(|\xi'(z)|^2) + c_{0,1}L_1(|\xi(z)|^2)L_1(|\xi''(z)|^2) + c_{1,1}L_1(|\xi'(z)|^2)L_1(|\xi''(z)|^2) \\
&+ c_{0,0} L_2(|\xi(z)|^2) + c_{2,0} L_2(|\xi'(z)|^2) + c_{0,2} L_2(|\xi''(z)|^2)
\end{align*}
and the constants $c_{k,l}$ are as in \eqref{eq_cs}.
Recall also that we are focusing on the case
$$H(z)=(1-|z|^2)e^{-\tfrac{1}{2}\abs{z}^2}.
$$
We will need the following derivatives of $H$:
\begin{align*}
\mathcal{D}_1H(z)&=(\partial-\bar z /2)\big[(1-|z|^2)e^{-|z|^2/2} \big] = 
-\bar z (2-|z|^2)e^{-|z|^2/2}, \\
\mathcal{D}_2H(z)&=(\bar \partial + z/2)\big[(1-|z|^2)e^{-|z|^2/2} \big]=
-ze^{-|z|^2/2}, \\
\mathcal{D}_1 \overline{\mathcal{D}_2}H(z)&=\overline{ \mathcal{D}_2 \overline{\mathcal{D}_1} H(z)}= (\partial -\bar z/2)(\partial + \bar z /2)\big[(1-|z|^2)e^{-|z|^2/2} \big] \\
&= -(\partial + \bar z /2)\big[ \bar z (2-|z|^2)e^{-|z|^2/2}\big] \\
&= -\bar z \overline{ \mathcal{D}_2 \big[ e^{-|z|^2/2}+ (1-|z|^2)e^{-|z|^2/2} \big]} = \bar z^2 e^{-|z|^2/2}, \\
\mathcal{D}_1 \overline{\mathcal{D}_1} H(z)&=-(\partial- \bar z/2) \big[z(2-|z|^2)e^{-|z|^2/2}\big] \\ 
&= 
-\bigg[(2-|z|^2)-|z|^2-\frac{|z|^2}{2}(2-|z|^2)-\frac{|z|^2}{2}(2-|z|^2) \bigg]e^{-|z|^2/2} \\
&= -(|z|^4-4|z|^2+2) e^{-|z|^2/2}=-2L_2(|z|^2)e^{-|z|^2/2}, \\
\mathcal{D}_2 \overline{\mathcal{D}_2}H(z) &= -(\bar \partial +z/2) \big( \bar z e^{-|z|^2/2}\big) =-e^{-|z|^2/2}. 
\end{align*}
These computations along with Lemma \ref{lem_precov} give us
\begin{align*}
\E \xi(z) \overline{\xi(w)} &= L_1(|z-w|^2)e^{-|z-w|^2/2+i \mathrm{Im}(z \bar w)}, \\
\E \xi'(z) \overline{\xi'(w)} &= -\frac12 \mathcal{T}_w \mathcal{D}_1 \overline{\mathcal{D}_1}H(z)= L_2(|z-w|^2) e^{-|z-w|^2/2+i \mathrm{Im}(z \bar w)}, \\
\E \xi''(z) \overline{ \xi''(w)} & = -\mathcal{T}_w \mathcal{D}_2 \overline{\mathcal{D}_2}H(z)=  e^{-|z-w|^2/2+i \mathrm{Im}(z \bar w)}, \\
\E \xi'(z) \overline{\xi(w)} &=
\frac{1}{\sqrt{2}}\mathcal{D}_1H(z-w)e^{i \mathrm{Im}(z \bar w)}= -\frac{1}{\sqrt{2}}\overline{(z-w)}(2-|z-w|^2) e^{-|z-w|^2/2+i \mathrm{Im}(z \bar w)},  \\
\E \xi(z) \overline{{\xi'(w)}}&= \frac{1}{\sqrt{2}}(z-w)(2-|z-w|^2) e^{-|z-w|^2/2+i \mathrm{Im}(z \bar w)}, \\
\E \xi''(z) \overline{\xi(w)} &= \mathcal{D}_2H(z-w) e^{i \mathrm{Im}(z \bar w)}= -(z-w) e^{-|z-w|^2/2+i \mathrm{Im}(z \bar w)}, \\
\E \xi(z) \overline{\xi''(w)}&= \overline{(z-w)}e^{-|z-w|^2/2+i \mathrm{Im}(z \bar w)}, \\
\E \xi''(z) \overline{\xi'(w)}&= -\frac{1}{\sqrt{2}}\mathcal{T}_w\mathcal{D}_2\overline{\mathcal{D}_1}H(z)= -\frac{1}{\sqrt{2}}(z-w)^2e^{-|z-w|^2/2+i \mathrm{Im}(z \bar w)}, \\
\E \xi'(z) \overline{\xi''(w)} &= -\frac{1}{\sqrt{2}}\overline{(z-w)}^2e^{-|z-w|^2/2+i \mathrm{Im}(z \bar w)}.
\end{align*}
To abbreviate notation, we will in the following write $s=|z-w|^2$. Applying Proposition \ref{prop_diag}, we obtain the following.
\begin{align*}
&\E\big( L_1(|\xi(z)|^2)L_1(|\xi'(z)|^2)L_1(|\xi(w)|^2)L_1(|\xi'(w)|^2)\big) 
 \\
 &=\big|\E \xi(z)\overline{\xi(w)}\cdot \E \xi'(z) \overline{\xi'(w)}+ \E \xi(z) \overline{\xi'(w)} \cdot \E \xi'(z) \overline{\xi(w)}\big|^2 = \big|L_1(s)L_2(s)-\frac12 s(2-s)^2\big|^2e^{-2s}, \\
&\E\big( L_1(|\xi(z)|^2)L_1(|\xi'(z)|^2)L_1(|\xi(w)|^2)L_1(|\xi''(w)|^2)\big) \\
&=\big|\E \xi(z)\overline{\xi(w)}\cdot \E \xi'(z) \overline{\xi''(w)}+ \E \xi(z) \overline{\xi''(w)} \cdot \E \xi'(z) \overline{\xi(w)}\big|^2 \\
&= \big|\frac{1}{\sqrt{2}} L_1(s)\overline{(z-w)}^2+\frac{1}{\sqrt{2}} \overline{(z-w)}^2(2-s)\big|^2 e^{-2s}= \frac{s^2}{2}\big|L_1(s)+(2-s)\big|^2e^{-2s}, \\
&\E\big( L_1(|\xi(z)|^2)L_1(|\xi'(z)|^2)L_1(|\xi'(w)|^2)L_1(|\xi''(w)|^2)\big) \\
&= |\E \xi(z)\overline{\xi'(w)}\cdot \E \xi'(z) \overline{\xi''(w)}+ \E \xi(z) \overline{\xi''(w)} \cdot \E \xi'(z) \overline{\xi'(w)}|^2 \\
&= \big|-\frac12(w-z)(2-s)\overline{(w-z)}^2-\overline{(z-w)}L_2(s)\big|^2e^{-2s} 
= s\big|s(2-s)/2-L_2(s)\big|^2 e^{-2s}, \\
&\E\big( L_1(|\xi(z)|^2)L_1(|\xi''(z)|^2)L_1(|\xi(w)|^2)L_1(|\xi''(w)|^2)\big) \\
&= |\E \xi(z)\overline{\xi(w)}\cdot \E \xi''(z) \overline{\xi''(w)}+ \E \xi(z) \overline{\xi''(w)} \cdot \E \xi''(z) \overline{\xi(w)}|^2 
=\big|L_1(s)-s\big|^2e^{-2s}, \\
&\E\big( L_1(|\xi(z)|^2)L_1(|\xi''(z)|^2)L_1(|\xi'(w)|^2)L_1(|\xi''(w)|^2)\big) \\
&= |\E \xi(z)\overline{\xi'(w)}\cdot \E \xi''(z) \overline{\xi''(w)}+ \E \xi(z) \overline{\xi''(w)} \cdot \E \xi''(z) \overline{\xi'(w)}|^2 \\
&= \big| \frac{1}{\sqrt{2}}(z-w)(2-s)- \frac{1}{\sqrt{2}} \overline{(z-w)}(w-z)^2 \big|^2 e^{-2s} \\
&= \frac12s \big|2-s-s|^2 e^{-2s}= 2s|1-s|^2e^{-2s}, \\
&\E\big( L_1(|\xi'(z)|^2)L_1(|\xi''(z)|^2)L_1(|\xi'(w)|^2)L_1(|\xi''(w)|^2)\big) \\
&= |\E \xi'(z)\overline{\xi'(w)}\cdot \E \xi''(z) \overline{\xi''(w)}+ \E \xi'(z) \overline{\xi''(w)} \cdot \E \xi''(z) \overline{\xi'(w)}|^2 
=\big|L_2(s)+\frac12 s^2\big|^2 e^{-2s}, \\
&\E L_1(|\xi(z)|^2)L_1(|\xi'(z)|^2)L_2(|\xi(w)|^2)= 2|\E \xi(z) \overline{\xi(w)}|^2 |\E \xi'(z) \overline{\xi(w)}|^2 
= 2 L_1(s)^2 \cdot \frac12 s (2-s)^2e^{-2s}, \\
&\E L_1(|\xi(z)|^2)L_1(|\xi'(z)|^2)L_2(|\xi'(w)|^2)= 2|\E \xi(z) \overline{\xi'(w)}|^2 |\E \xi'(z) \overline{\xi'(w)}|^2 
= 2 \frac12 s(2-s)^2L_2(s)^2e^{-2s}, \\ 
&\E L_1(|\xi(z)|^2)L_1(|\xi'(z)|^2)L_2(|\xi''(w)|^2)= 2|\E \xi(z) \overline{\xi''(w)}|^2  |\E \xi'(z) \overline{\xi''(w)}|^2 = 2s\cdot \frac12 s^2 e^{-2s}= s^3e^{-2s}, \\
&\E L_1(|\xi(z)|^2)L_1(|\xi''(z)|^2)L_2(|\xi(w)|^2)= 2|\E \xi(z) \overline{\xi(w)}|^2  |\E \xi''(z) \overline{\xi(w)}|^2 
= 2 L_1(s)^2se^{-2s}, \\
&\E L_1(|\xi(z)|^2)L_1(|\xi''(z)|^2)L_2(|\xi'(w)|^2)= 2|\E \xi(z) \overline{\xi'(w)}|^2  |\E \xi''(z) \overline{\xi'(w)}|^2 
= \frac12 s(2-s)^2s^2e^{-2s}, \\
&\E L_1(|\xi(z)|^2)L_1(|\xi''(z)|^2)L_2(|\xi''(w)|^2)= 2|\E \xi(z) \overline{\xi''(w)}|^2  |\E \xi''(z) \overline{\xi''(w)}|^2 
= 2se^{-2s},\\
&\E L_1(|\xi'(z)|^2)L_1(|\xi''(z)|^2)L_2(|\xi(w)|^2)= 2|\E \xi'(z) \overline{\xi(w)}|^2  |\E \xi''(z) \overline{\xi(w)}|^2 = s^2(2-s)^2e^{-2s}, \\
&\E L_1(|\xi'(z)|^2)L_1(|\xi''(z)|^2)L_2(|\xi'(w)|^2)= 2|\E \xi'(z) \overline{\xi'(w)}|^2  |\E \xi''(z) \overline{\xi'(w)}|^2 
= L_2(s)^2 s^2e^{-2s}, \\
&\E L_1(|\xi'(z)|^2)L_1(|\xi''(z)|^2)L_2(|\xi''(w)|^2)= 2|\E \xi'(z) \overline{\xi''(w)}|^2  |\E \xi''(z) \overline{\xi''(w)}|^2 
= s^2e^{-2s}, \\
&\E L_2(|\xi(z)|^2)L_2(|\xi(w)|^2) = 
|\E \xi(z) \overline{\xi(w)}|^4 =L_1(s)^4e^{-2s}, \\
&\E L_2(|\xi(z)|^2)L_2(|\xi'(w)|^2) = 
|\E \xi(z) \overline{\xi'(w)}|^4 =\frac14 s^2(2-s)^4e^{-2s}, \\
&\E L_2(|\xi(z)|^2)L_2(|\xi''(w)|^2) = 
|\E \xi(z) \overline{\xi''(w)}|^4 =s^2e^{-2s}, \\
&\E L_2(|\xi'(z)|^2)L_2(|\xi'(w)|^2) = 
|\E \xi'(z) \overline{\xi'(w)}|^4 =L_2(s)^4e^{-2s}, \\ 
&\E L_2(|\xi'(z)|^2)L_2(|\xi''(w)|^2) = 
|\E \xi'(z) \overline{\xi''(w)}|^4 =\frac14 s^4e^{-2s}, \\ 
&\E L_2(|\xi''(z)|^2)L_2(|\xi''(w)|^2) = 
|\E \xi''(z) \overline{\xi''(w)}|^4 =e^{-2s}.  
\end{align*}
These expressions are combined into an expression for $\E[\phi(z)\phi(w)]$ in the accompanying notebook \cite{jupy_gwhf_nonhyp}.

\section*{Acknowledgments}
The authors gratefully acknowledge support from the Research in Teams program ``Time-frequency analysis of random point processes'' of the Erwin Schrödinger International Institute for Mathematics and Physics of the University of Vienna. N.\ F. gratefully acknowledges the support of the Israel Science Foundation, grants 1327/19 and 3541/24.
This research was funded in whole or in part by the Austrian Science Fund (FWF): 10.55776/Y1199. For open access purposes, the authors have applied a CC BY public copyright license to any author-accepted manuscript version arising from this submission. The authors are grateful to Lukas Odelius for helpful comments.

\end{document}